\documentclass[11pt,reqno]{amsart}
\usepackage{amssymb,latexsym}
\usepackage{framed}
\usepackage{amsmath}
\usepackage{amsthm}
\usepackage{graphicx}
\usepackage{hyperref}
\usepackage{titletoc}
\usepackage{pdfsync}
\usepackage{color,xcolor}

\usepackage{comment}

\numberwithin{equation}{section}
\newtheorem{theorem}{Theorem}[section]
\newtheorem{proposition}[theorem]{Proposition}
\newtheorem{lemma}[theorem]{Lemma}

\usepackage{multirow}
\usepackage[font=bf,aboveskip=15pt]{caption}
\usepackage[toc,page]{appendix}

\theoremstyle{definition}

\newtheorem{remark}[theorem]{Remark}

\def\R{{\mathfrak R}}

\def\begeq{\begin{equation}}
\def\endeq{\end{equation}}

\def\lf{\left}
\def\ri{\right}

\def\R{\Bbb R}

\allowdisplaybreaks

\begin{document}
\title[]{Clustering of Boundary Interfaces  for an inhomogeneous Allen-Cahn equation on a smooth bounded domain}
\author{Lipeng Duan}
\address{Lipeng Duan,
\newline\indent School of Mathematics and Statistics, Central China Normal University,
\newline\indent Wuhan 430079, P. R. China.
}
\email{lpduan777@sina.com}

\author{Suting Wei}
\address{Suting Wei,
\newline\indent Department of Mathematics, South China
Agricultural University,
\newline\indent Guangzhou, 510642, P. R. China.}
\email{stwei@scau.edu.cn}

\author{Jun Yang$^\S$}
\address{Jun Yang,
\newline\indent School of Mathematics and Information Science,
Guangzhou University,
\newline\indent Guangzhou 510006, P. R. China.
}
\email{jyang2019@gzhu.edu.cn}

\date{\today}

\thanks{$\S$ Corresponding author: Jun Yang, jyang2019@gzhu.edu.cn}

\begin{abstract}
We consider the inhomogeneous Allen-Cahn equation
$$
\epsilon^2\Delta u\,+\,V(y)(1-u^2)\,u\,=\,0\quad \mbox{in}\  \Omega,
\qquad
\frac {\partial u}{\partial \nu}\,=\,0\quad \mbox{on}\  \partial \Omega,
$$
where $\Omega$ is a bounded domain in ${\mathbb R}^2$ with smooth boundary $\partial\Omega$ and $V(x)$ is a positive smooth function, $\epsilon>0$ is a small parameter,
$\nu$ denotes the unit outward normal of $\partial\Omega$.
For any fixed integer $N\geq 2$, we will show the existence of a clustered solution $u_{\epsilon}$ with $N$-transition layers
near $\partial \Omega$ with mutual distance $O(\epsilon|\ln \epsilon|)$,
provided that  the generalized  mean curvature $\mathcal{H} $ of $\partial\Omega$ is positive    and    $\epsilon$ stays away from a discrete set of values at
which resonance occurs.
Our result is an extension of those (with dimension two) by
A. Malchiodi, W.-M. Ni, J. Wei in Pacific J. Math. (Vol. 229, 2007, no. 2, 447-468)
and  A. Malchiodi, J. Wei in J. Fixed Point Theory Appl. (Vol. 1, 2007, no. 2, 305-336).

\vspace{3mm}


\vspace{2mm}

\end{abstract}

\maketitle

%\footnote{
%$^\S$ Corresponding author
%\textit{Jun Yang, jyang@gzhu.edu.cn}
% }

\section{Introduction}\label{section1}

We consider the   inhomogeneous Allen-Cahn equation
\begin{equation}
\label{originalproblem}
\epsilon^2\Delta u\,+\,V(y)(1-u^2)\,u\,=\,0\quad \mbox{in}\  \Omega,
\qquad
\frac {\partial u}{\partial \nu}\,=\,0\quad \mbox{on}\  \partial \Omega,
\end{equation}
where $\Omega$ is a smooth and bounded domain in ${\mathbb R}^{ \mathrm{d}}$,
$\epsilon$ is a small positive parameter,
$\nu$ is the unit outer normal to $\partial\Omega$,  $ V(y)$ is a positive smooth function on $ \bar{\Omega}$. The non-constant function $V$ represents the spatial inhomogeneity. The function $u$ represents a continuous realization of the phase present in a material confined to the region $\Omega$ at the point $y$ which, except for a narrow region, is expected to take values close to $+1$ or $-1$.
A component of the set $\{y\in \Omega : u(y)=0 \}$ is called an interface or an phase transition layer of $u$.

\medskip
Here we  mention another type inhomogeneous Allen-Cahn equation (called Fife-Greenlee mode)
\begin{align}\label{Fife-Greenleeproblem}
\epsilon^2\Delta{u} + \bigl({u}-\mathbf{a}({y})\bigr)(1-{u}^2) = 0\quad \mbox{in} \ \Omega,\ \ \
\frac{\partial{u}}{\partial\nu}=0 \quad \mbox{on}\ \partial\Omega.
\end{align}
Problem \eqref{Fife-Greenleeproblem} has been studied extensively in recent years.  See \cite{ab1, abc1, as, amp, DancerYan1, DancerYan2, delP1, delP2, delPKowWei2, donascimento,Du1, Du2, DuWei, Fife, FifeGreen, HaleSaka, MahMalWei, TWY, JWeiYang3} for backgrounds and references.
The case $ V(y)\equiv 1$ or $\mathbf{a}(y) =0$   corresponds to the Allen-Cahn equation \cite{ac}
\begin{equation}\label{AllenCahn}
\epsilon^2\Delta{u}+ {u}(1-{u}^2)=0\quad \mbox{in} \ \Omega,\ \ \
\frac{\partial{u}}{\partial\nu} = 0 \quad \mbox{on}\ \partial\Omega,
\end{equation}
for which extensive literature on transition layer solution is available,
see for instance \cite{abf1, acf, FlorPad,  KohnSter, Kow, MahNiWei, MahWei, Modica, NakaTana, PacarRitor, PadiTone, RabinStre, RabinStre2, SternZum}, and the references therein.

\medskip
For the inhomogeneous Allen-Cahn equation \eqref{originalproblem}, there are also some results recently. In the case of dimension $N=1$, it is shown in \cite{Nakashima} that problem \eqref{originalproblem} has interior layer solutions, and the transition layers can appear only near the local minimum and local maximum points of the coefficient $V$ and that at most one single layer can appear near each local minimum point of $V$.

\medskip
For {\em\bf{the interior phase transition phenomena (away from $\partial\Omega$)}} to problem \eqref{originalproblem} with inhomogeneity $V$ on higher dimensional domain,
if $ \tilde \Gamma$ is  a closed  curve in $\Omega\subset{\mathbb R}^2$  satisfying the stationary and  non-degeneracy  conditions with respect to $\int_{\tilde \Gamma} V^{\frac 12}$,
Z. Du and C. Gui \cite{DuGui} constructed a solution with a layer near $\tilde \Gamma$, see also \cite{LiNakashima}.
Later on, J.  Yang and X.  Yang  \cite{YangYang}   constructed  clustered interior phase transition layers, see also \cite{DuWang}.
On the other hand,
X. Fan, B. Xu and J. Yang \cite{FXY} constructed a solution with single {\em \bf{interior phase transition layer connecting $\partial\Omega$}} near a curve ${\hat\Gamma}$,
which connects perpendicularly the boundary $\partial\Omega$ and is  also stationary and  non-degenerate with respect to $\int_{\hat \Gamma} V^{\frac 12}$.
Clustered interior phase transition layers connecting $\partial\Omega$ can be found in the paper by S. Wei and J.  Yang \cite{SWeiYang2}.

\medskip
For the  {\em boundary interface phenomena}, here we will mention some works on the problem \eqref{AllenCahn}.
If $\Omega\subset{\mathbb R}^{\mathrm d}$ is a unit ball,  A. Malchiodi,  W.-M. Ni and J. Wei \cite{MahNiWei} constructed a radially symmetric solution $u_\epsilon$
 having $N$ interfaces  $ \{(r, \Theta): u_\epsilon(r) =0 \} = \bigcup_{j=1}^N \{(r, \Theta): r= r_j^\epsilon\}$ such that
$$
1=r_0^\epsilon> r_1^\epsilon > r_2^\epsilon>\cdots>r_N^\epsilon,
\qquad
r_{j-1}^\epsilon- r_j^\epsilon =O(\epsilon |\ln \epsilon|),\quad \forall\, j=1, \cdots, N.
$$
Here $(r, \Theta)$ are the sphere coordinates in ${\mathbb R}^{\mathrm d}$.
 For the non-radial case,  A. Malchiodi and J. Wei in \cite{MahWei}  showed the existence of  single  boundary interface under the condition that the mean curvature of $\partial \Omega$ is positive and $\epsilon$ stays away from a discrete set of values at which resonance occurs.

\medskip
In the present paper,  we will make an extension of the results for problem \eqref{AllenCahn} in \cite{MahNiWei, MahWei}, and consider the existence of {\em clustering boundary transition layers} for problem \eqref{originalproblem}.
It will be shown that the inhomogeneity term $V(y)$ as well as the boundary of $\Omega$  will play  an important role in the construction of solutions, see Remark \ref{remark12}.
To avoid too much tedious computations, we here only consider the two dimensional case, i.e. $\mathrm{d}=2$ in \eqref{originalproblem}.

\medskip
For simplicity of setting, let $\partial\Omega$ be a simple closed curve in $\R^2$  and $\ell=|\partial\Omega|$ be the total length of the boundary $\partial\Omega$. We consider the natural parameterization  $\gamma(\theta)$ of $\partial\Omega$ with positive orientation, where $\theta$ denotes arclength parameter measured from a fixed point of  $\partial\Omega$. For sufficiently small $\delta_0$, points $y$ near $\partial\Omega$ in $\R^2 $ can be represented in the form
\begin{align}\label{fermi}
y= \gamma(\theta) -t \nu(\theta), \quad |t|<\delta_0,\quad \theta\in [0,\ell),
\end{align}
where $ \nu(\theta) $ denotes the unit outer normal to $\partial\Omega$.
It is well known that $H(x)=\tanh\bigl(\frac{x}{\sqrt{2}}\bigr)$ is the unique heteroclinic solution of the problem
\begin{equation}
\label{blocksolution}
H{''}+(1-H^2)H=0\quad \mbox{in}\ \mathbb{R} ,
\quad
H(\pm \infty)=\pm 1,
\quad
H(0)=0.
\end{equation}
In the local coordinates $(t, \theta)$ as in \eqref{fermi}, the main theorem reads:

\begin{theorem}\label{main1}
Let $\Omega$ be a smooth and bounded domain in $\mathbb{R}^2$ and $V (y)$ be a positive smooth function in  $\bar \Omega $.
Assume that the generalized mean curvature of $\partial\Omega$ is positive, i.e.
\begin{align}\label{meancurvaturepositive}
\mathcal{H}(\theta)  := k(\theta)  - \frac{V_t(0, \theta ) }{ 2V(0, \theta) }>0\quad\mbox{on }\, \partial\Omega,
\end{align}
where $ k(\theta) $ is the curvature of $\partial \Omega$.
Then for any fixed positive integer $ N$,  there exists a sequence $\{\epsilon_i: i=1, 2,\cdots\}  $  of $\epsilon$ such that
problem \eqref{originalproblem} has a clustered solutions $u_{\epsilon_i} $ with $N $-phase transition layers at mutual distance $O(\epsilon_i | \ln { \epsilon_i} | )$.
Near the boundary  $\partial\Omega$,  $u_{\epsilon_i} $  has the form of
\begin{align}
u_{\epsilon_i}(y)=\sum_{j=1}^N (-1)^j  H\Big(V^\frac{1}{2} (0, \theta)    \frac{ t - \epsilon_i f_j(\theta)}{ \epsilon_i }   \Big)
\,+\,
\frac{1}{2}\Big((-1)^{N}+ 1\Big)
\,+\,
o(1).
\end{align}
and, away from $\partial \Omega,$
  \begin{align}
 u_{\epsilon_i}(y)   \rightarrow ( - 1 )^N \quad \text{as} \qquad    {\epsilon_i} \rightarrow 0.
\end{align}
The functions  $\{f_j: j=1,\cdots, N\} $ satisfy
\begin{align}\label{f1theta}
f_1(\theta)\,=\,    \frac{1}{2\sqrt{2} V^\frac{1}{2} (0, \theta)  }\Big[ \ln{\frac{1}{ N\epsilon_i}}
-   \ln\mathcal{H}(\theta)  +\ln \big(9\gamma_1V (0, \theta)\big)  \Big]
\,+\, O(\epsilon^{1/2}),
\end{align}
and, for $j=2, \cdots, N, $
\begin{equation}\label{consraintsoffk}
f_j(\theta)-f_{j-1}(\theta)= \frac{1}{\sqrt{2}V^\frac{1}{2} (0, \theta)  }\Big[ \ln{\frac{1}{(N+1-j)\epsilon_i}}
-   \ln\mathcal{H}(\theta)  + \ln \big(9\gamma_1 V (0, \theta)\big)  \Big]
\,+\, O(\epsilon^{1/2}),
\end{equation}
where the constant  $\gamma_1$ is defined in  \eqref{gamma01}.
\qed
\end{theorem}

\medskip
Some words are in order to explain the above results.

\begin{remark}\label{remark12}\

{\em
 Let us consider a Riemannian manifold ${\mathcal M}$  of $n$ dimension with volume element $\mathrm{d}{\mathcal V}_0$
and its $n-1$ dimensional submanifold ${\mathcal N}$ with mean curvature ${\mathbb H}$.
By the comments of F. Morgan (Page 835 in \cite{Morgan}), in density manifold ${\mathcal M}$
with volume element $\mathrm{d}{\mathcal V}=e^{\Psi}\mathrm{d}{\mathcal V}_0$,
M. Gromov first introduced the generalization of mean curvature ${\mathcal H}_\Psi$ of submanifold ${\mathcal N}$ in the form
$$
{\mathcal H}_\Psi={\mathbb H}-\frac{1}{n-1}\frac{\partial\Psi}{\partial{\tilde\nu}},
$$
where ${\tilde\nu}$ is the normal of ${\mathcal N}$.
We now consider ${\mathbb R}^2$ as a Riemannian manifold with the metric $g=V(y)\,({\mathrm d}{y}_1^2+{\mathrm d}{y}_2^2)$,
in which the volume density is $e^{\Psi}$ with $\Psi=\ln V^{1/2}$. The curve $\partial\Omega$ is a submanifold of the density manifold ${\mathbb R}^2$
and then its generalized mean curvature is $\mathcal{H}$ in \eqref{meancurvaturepositive}.
If $V\equiv 1$, then $\mathcal{H}=k$ and \eqref{meancurvaturepositive} is exactly the requirement of positive curvature of $\partial\Omega$ in \cite{MahWei}.
It is then obvious that Theorem \ref{main1} is a natural extension of the results of dimension two in \cite{MahNiWei} and \cite{MahWei}.

\medskip
As a submanifold of the density Riemannian manifold $\R^2$,  $\partial\Omega$ has similar effects to the interfaces as described in the Introduction section of \cite{MahWei}
due to the homogeneous Neumann boundary condition. Whence, we will also encounter resonance phenomena in the procedure of constructing solutions in Theorem \ref{main1}.
However, there is not any resonance phenomenon in the radial symmetric case for the clustering of multiple interfaces in \cite{MahNiWei}.
This is the reason that we shall do much more analysis than \cite{MahNiWei} to get more accurate asymptotical expressions of the parameters $f_1, \cdots, f_N$
as well as the formula of the approximate solution $u_3$ in \eqref{u3}.

\medskip
The asymptotic formulae in \eqref{f1theta}-\eqref{consraintsoffk} will be formally derived in Section \ref{section3.1}.
The behaviors in \eqref{f1theta}-\eqref{consraintsoffk} show that the boundary layers squeeze more and more tightly toward $\partial\Omega$.
The fact is due to the effect of $\partial\Omega$ toward the boundary layers caused by the homogeneous Neumann boundary condition in \eqref{originalproblem}.
This is quite different from the results in previous papers (such as \cite{delPKowWei3}, \cite{delPKowWeiYang} on clustering interfaces for \eqref{AllenCahn},
and also \cite{SWeiYang2}, \cite{YangYang} on interior clustering interfaces for \eqref{originalproblem})
in which multiple layers in the cluster distribute evenly along the limit set (The distances between neighbouring layers are almost the same in the main order).

\medskip
As we have mentioned that, in order to handle the resonance, we shall first construct a good approximate solution in such a way that
it will locally solve problem \eqref{originalproblem} very well. This can be carried out by adjusting functions $f_1,\cdots, f_N$
and then choosing correction terms step by step in Section \ref{section3}.
In fact, we will set
\begin{equation}\label{1.11}
f_j(\theta)=\dot{f}_j(\theta) + {\bar f}_j(\theta) +\check{f}_j(\theta) \,+\,{\tilde f}_j(\theta),\quad \forall\, j=1,\cdots, N.
\end{equation}
The readers can refer to \eqref{f01}-\eqref{f2N}, \eqref{barfexpression1}-\eqref{barfexpressionN}, \eqref{checkfnorm}
for $\dot{f}_j(\theta),  {\bar f}_j(\theta), \check{f}_j(\theta)$.
On the other hand, all functions ${\tilde f}_1, \cdots, {\tilde f}_N$ with constraints in \eqref{ddotfnorm} will be found by the reduction method in Sections \ref{section6}-\ref{section7}.
Please note that they will solve a small perturbation form of the following system,
\begin{align}
\begin{aligned}\label{todaac0}
 &-\, \epsilon^2 \gamma_0 \beta {\tilde f}''_j
 \,+\,
 \epsilon\,6 \sqrt 2  \beta^2
 \gamma_{1,j}  {\mathbf k}_j    \Big[e^{-  \sqrt 2 \beta  ({\tilde f}_j-{\tilde f}_{j-1})}-1 \Big]
 \\[2mm]
& \,-\,
 \epsilon\,  6 \sqrt 2   \beta^2  \gamma_{2,j}    {\mathbf k}_{j+1} \,\Big[  e^{-  \sqrt 2 \beta  ({\tilde f}_{j+1}-{\tilde f}_j)}  -1 \Big] \approx 0,
\quad
 j=1, \cdots, N,
\end{aligned}
\end{align}
for a universal constant $\gamma_0$ in \eqref{gamma01}.
The functions $\beta$, $\gamma_{1,j}, \gamma_{2,j},  {\mathbf k}_j$ are given in \eqref{beta}, \eqref{ga1n}, \eqref{ga2n}, \eqref{mathbf kj}.
Therefore, the asymptotical behaviors of the components of $f_1,\cdots, f_N$ in \eqref{1.11} will imply \eqref{f1theta}-\eqref{consraintsoffk}.
}
\qed
\end{remark}

\medskip
The remaining part of this paper is devoted to the proof of Theorem \ref{main1}, which will be organized as follows:

\begin{itemize}
\item[1.]
In Section \ref{section2}, we will write down the equations in local forms.

\item[2.]
In Section \ref{section3}, we shall construct the first approximation to a real solution in \eqref{firstappro1} and then compute the error.
In order to improve the approximation, suitable correction terms will be added step by step and a good approximate solution $u_3$ will be given in \eqref{u3}.
These tedious analysis will constitute the main part the present paper.

\item[3.]
Next, we  set up the inner-outer gluing scheme \cite{delPKowWei1} in Section \ref{section4}, so that we can deduce the projected problem (\ref{Projectedproblem1})-(\ref{projectedproblem3}) and give the existence of its solutions in Proposition \ref{proposition3.1}.

\item[4.]
In order to get a real solution, the well-known infinite dimensional reduction method will be needed in Sections \ref{section6}-\ref{section7}.
The final step is to adjust the parameters ${\tilde f}_1, \cdots, {\tilde f}_N$, which satisfy a nonlinear coupled  system of second order differential equations, see (\ref{ts2a}).
This will be done in Section \ref{section7}.
Note that we also need suitable analysis from \cite{delPKowWei2,delPKowWeiYang, JWeiYang3} to deal with the resonance phenomena in Lemma \ref{lemma6point3}.

\item[5.]
Some tedious computations and a useful lemma  will be given in Appendices \ref{appendixA}-\ref{appendixB}.
For the convenience of the readers, a collection of notation and conventions will be provide at the beginning of Section \ref{section3}.
\end{itemize}

\medskip
\section{Local forms of the problem}\label{section2}

For the convenience of expressions, by the rescaling $y=\epsilon \tilde{y}$, problem \eqref{originalproblem} can be rewritten as
\begin{equation}
\label{originalproblem1}
\Delta u\,+\,V(\epsilon {\tilde y})(1-u^2)\,u\,=\,0\quad \mbox{in}\  \Omega_\epsilon,
\qquad
\frac {\partial u}{\partial \nu}\,=\,0\quad \mbox{on}\  \partial \Omega_\epsilon,
\end{equation}
where $\Omega_\epsilon=\Omega/\epsilon$, $\partial\Omega_\epsilon=\partial\Omega/\epsilon$.

\medskip
Recall the coordinates $(t, \theta)$ given in \eqref{fermi} and
let
\begin{equation}
(s, z)=\frac{1}{\epsilon}(t, \theta),
\qquad
z\in \Big[0,\frac{\ell }{\epsilon}\,\Big),\quad s \in \Big(0,\frac{\delta_0}{\epsilon}\Big),
\end{equation}
 be the natural stretched  coordinates associated to the curve  $\partial\Omega_\epsilon=\partial\Omega/\epsilon$.
Any $\tilde y$ near $\partial\Omega_\epsilon$ in $\R^2 $ can be represented in the form
\begin{equation}\label{sz}
\tilde y= \frac{1}{\epsilon}\gamma(\epsilon z) -s \nu(\epsilon z).
\end{equation}
We then get
\begin{equation}
 \frac{\partial {\tilde y}}{\partial s}=-\nu,
  \qquad
 \frac{\partial {\tilde y}}{\partial z}=\gamma_\theta-\epsilon s \nu_\theta.
\end{equation}
Since
$$
|\gamma_\theta|=1,
\quad
|\nu|=1,
\quad\mbox{and}\quad
\nu_\theta= k  \gamma_\theta,\quad \gamma_\theta\cdot\nu=0,
$$
the metric matrix is
\begin{equation}
\mathfrak{g}\,=\,\left(\begin{array}{cc}
\mathfrak{g}_{11} & \mathfrak{g}_{12}
\\
\mathfrak{g}_{21} & \mathfrak{g}_{22}
\end{array}
\right)={
\left(\begin{array}{ccc}
1& 0\\
0& \big(1-\epsilon  k s\big)^2
\end{array}
\right )}.
\end{equation}
So the determinant of the metric matrix is
$$\det{\mathfrak{g}}=\big(1-\epsilon s k\,\big)^2,$$
and
\begin{equation}
\mathfrak{g}^{-1}\,=\,\left(\begin{array}{cc}
\mathfrak{g}^{11} & \mathfrak{g}^{12}
\\
\mathfrak{g}^{21} & \mathfrak{g}^{22}
\end{array}
\right)={
\left( \begin{array}{ccc}
1 & 0\\
0& \big(1-\epsilon   k s \,\big)^{-2}\end{array}
\right)}.
\end{equation}
The Laplacian-Beltrami operator has the form
\begin{equation}
\begin{split}
\Delta_\mathfrak{g}
%&=\frac{1}{\sqrt{\det g}}\,\frac{\partial}{\partial x_i}\,(\sqrt{\det g} \,{g}^{ij}\,\frac{\partial}{\partial x_i})\\[1mm]
&=\frac{1}{1-\epsilon ks} \frac{\partial}{\partial s} \Big[ \big(1-\epsilon ks\big) \frac{\partial}{\partial s}    \Big] +\frac{1}{1-\epsilon ks}\frac{\partial}{\partial z} \Big[ \big(1-\epsilon ks\big)^{-1} \frac{\partial}{\partial z}    \Big]
\\[1mm]
&=\frac{\partial^2}{\partial s^2}- \frac{\epsilon k}{1-\epsilon ks} \frac{\partial }{\partial s} + \frac{1}{(1-\epsilon ks)^2 }\frac{\partial^2}{\partial z^2}+\frac{\epsilon^2 s k'}{(1-\epsilon ks)^3 }\frac{\partial}{\partial z}.
\end{split}
\end{equation}
\medskip
  In the coordinates $(s,z)$, we will give the computations of $ \nabla u$ as follows
 \begin{equation*}
 \nabla u \,=\, \mathfrak{g}^{ij} \,\frac{\partial u }{\partial {\hat{y}_i} } \,  \frac{\partial { \tilde y} }{\partial {\hat y}_j}
 \,=\,\frac{\partial u}{\partial s }  \,  \frac{ \partial { \tilde y}}{ \partial s} 
 \, +  \, 
 \big(1-\epsilon s  k \,\big)^{-2} \,\frac{\partial u}{\partial z }  \,   \frac{ \partial { \tilde y}}{ \partial z} ,
 \end{equation*}
 where
 $\hat{y}_1=s,\, \hat{y}_2=z $.
Since $ \frac{\partial {\tilde y}}{\partial s}=-\nu$, then
the normal derivative $ \frac{\partial u}{\partial \nu }$ has a local form as follows
\begin{equation}
\frac{\partial u}{\partial \nu } \, =\,   \nabla u  \cdot \nu
=\,  -\Big[\frac{\partial u}{\partial s }  \,  \frac{ \partial { \tilde y}}{ \partial s}\cdot  \frac{ \partial { \tilde y}}{ \partial s}
 \, +  \,
  \big(1-\epsilon s  k \,\big)^{-2}  \, \frac{\partial u}{\partial z } \,  \frac{ \partial { \tilde y}}{ \partial z} \cdot   \frac{ \partial { \tilde y}}{ \partial s} \, \Big]
\, = \,  - \frac{\partial u(s,z) }{\partial s }.
\end{equation}

\medskip
Hence, problem  \eqref{originalproblem1} can be locally recast as follows
\begin{equation}
\label{originalproblem2}
u_{ss}+ u_{zz} +B_1(u) +V(\epsilon s,\epsilon z)F(u)\,=\,0,
\quad\forall\, (s, z)\in(0, \delta_0/\epsilon)\times(0, \ell/\epsilon),
\end{equation}
with boundary condition
\begin{equation}\label{boundarycondition}
{\mathcal{D}}(u)=0,\quad \forall\, z\in[0, \ell/\epsilon).
\end{equation}
In the above, the boundary operator is given by
\begin{equation}\label{boundary-operator}
{\mathcal{D}}(u)\, := \, - \frac{\partial u(s,z) }{\partial s }\big|_{s=0},
\end{equation}
and the nonlinear term is
\begin{equation}\label{Fu}
F(u)=u-u^3.
\end{equation}
Moreover, the linear differential operator $B_1$ has the form
\begin{equation}
B_1(u)=- \frac{\epsilon k u_s}{1-\epsilon ks}+\Big[ \frac{1}{(1-\epsilon ks)^2 } -1 \Big]u_{zz}+\frac{\epsilon^2 s k'}{(1-\epsilon ks)^3 }u_z.
\end{equation}

\medskip
Furthermore, we can expand the operator $B_1$ as
\begin{equation}\label{B1u}
B_1(u)=-\big(\epsilon k +\epsilon^2 sk^2\big) u_s+B_0(u),
\end{equation}
with
\begin{equation}\label{B0u}
B_0(u) = \epsilon^2sa_1u_z+ \epsilon s a_2 u_{zz} +\epsilon^3 s^2a_3 u_s,
\end{equation}
for certain smooth functions $a_j(t,\theta),\,j=1,2,3.$
In the neighborhood of $\partial \Omega_\epsilon$, by taking the Taylor expansion,
we also expand
\begin{equation}\label{Vexpansion}
V(\epsilon s,\epsilon z)=V(0,\epsilon z)+\beta_1(\epsilon z)\epsilon s+\frac{1}{2}\beta_2(\epsilon z)\epsilon^2s^2+a_4(\epsilon s,\epsilon z)\epsilon^3s^3,
\end{equation}
for a smooth function $a_4(t,\theta)$ and
\begin{equation}\label{beta1beta2}
\beta_1(\theta)=V_t(0, \theta),
\qquad
\beta_2(\theta)=V_{tt}(0, \theta).
\end{equation}
Then $u$ solves \eqref{originalproblem2} if and only if
\begin{equation}\label{S(v)}
{\bf{S}}(u) := u_{ss}+ u_{zz} +V(0,\epsilon z)F(u)
+B_2(u) +B_3(u)
+B_4(u)=0,
\end{equation}
for $ (s, z)\in(0, \delta_0/\epsilon)\times(0, \ell/\epsilon),$ where
\begin{align}
\label{B2}
B_2(u)=B_0(u)+a_4(\epsilon s,\epsilon z)\epsilon^3s^3u(1-u^2), %=O(\epsilon^3),
\\[2mm]
\label{B3}
B_3(u)=-\big(\epsilon k+\epsilon^2 sk^2 \big)  u_s,
\\[2mm]
\label{B4}
B_4(u)=  \Big[\epsilon s\beta_1+\frac{1}{2}\epsilon^2s^2\beta_2\Big]   (u-u^3).
\end{align}

\medskip
\section{Local approximate solutions}\label{section3}
The main objective of this section is to construct a suitable approximate solution in local coordinates $(s, z)$ near $\partial \Omega_\epsilon$
and then evaluate its error terms.

\medskip
In the following, we will use $H(x)$ in \eqref{blocksolution} as the basic block to construct solutions with phase transition layers.
It is well known that $H$ is odd and enjoys the following behaviors
\begin{equation}\label{asymptoticofH}
\begin{split}
H(x)&=\pm \big( 1-2e^{-\sqrt{2}|x|} \big) +O\big(e^{-2\sqrt{2}|x|}\big),\quad\; {\rm as}\  x\rightarrow \pm\infty,
\\[1mm]
H{'}(x)&=2\sqrt{2}e^{-\sqrt{2}|x|}+O\big(e^{-2\sqrt{2}|x|}\big),\qquad \quad  {\rm as}\  |x|\rightarrow +\infty.
\end{split}
\end{equation}
From equation \eqref{blocksolution}, it is easy to derive that
\begin{equation*}
1-H^2(x)\,=\,\sqrt{2}\,H_x(x),\qquad
\int_\mathbb{R}{H_x^2}\,{\rm d}x=\,\frac{2\sqrt{2}
}{3}.
\end{equation*}
Integrating by parts, we have
\begin{equation}\label{identity1}
2\int_{\mathbb R} x\, H_x\,H_{xx}  \,\mathrm {d }x= -\int_{\mathbb R} H_x^2 \,\mathrm {d }x=-\frac{2\sqrt 2}{3}.
\end{equation}
From  \eqref{blocksolution}, it is also trivial to derive that
\begin{equation}\label{identity2}
3\int_{\mathbb R} (1-H^2)  e^{-\sqrt{2}\, x } \, H_x  \,\mathrm {d }x=-\int_{\mathbb R} (H_{xxx} -2H_x) e^{-\sqrt{2}\, x } \,\mathrm {d }x= 8.
\end{equation}

\bigskip
\noindent{\bf Notation:}
{{\it
We pause here to set the conventions
\begin{equation}\label{gamma01}
\gamma_0= \int_\mathbb{R}\,{H_x^2}\, {\mathrm{d}}x=\frac{2\sqrt{2}}{3},
\qquad
\gamma_1= \int_\mathbb{R}\,{ \, e^{-\sqrt{2} x}}H_{x}^2 \, {\mathrm{d}}x
 \, =\, \int_\mathbb{R}\,{ e^{\sqrt{2}x}}\, H_{x}^2 \, {\mathrm{d}}x\, =\, \frac{8}{3\sqrt 2}.
\end{equation}
The set ${\mathfrak S}$ represents the strip in $\mathbb{R}^2$ of the form
\begin{equation}
\label{huaxies}
{\mathfrak S}=\Big\{(s,z): s\in  (0,  +  \infty ),\quad 0<z<{\ell}/{\epsilon}\Big\}.
\end{equation}
The functions $\beta, \beta_1, \beta_2$ are defined in \eqref{beta}, \eqref{beta1beta2}.
\qed
}}

\subsection{Formal derivation of the interfaces}\label{section3.1}

For any positive integer $N$, we assume that the locations of
the phase transition layers are characterized by functions $s\,=\,f_j(\epsilon z),\, j=1,\cdots,N$ such that
\begin{eqnarray*}%\label{condetwo}
 f_j:\,(0,\ell)\to\mathbb{R},
\qquad
0<f_1<f_2<\cdots<f_N,
\end{eqnarray*}
\begin{equation*}%\label{boundarycondition}
f_{j}(0)\,=\,f_j(\ell),
\qquad
f'_{j}(0)\,=\,f'_j(\ell).
\end{equation*}
Recalling $H$ given in (\ref{blocksolution}) and setting
\begin{equation}\label{beta}
\beta(\theta)=V^{1/2}(0, \theta),
\end{equation}
we set the profile of a solution
\begin{equation*}
 v(s,z)\,\equiv\,\sum_{j=0}^N\, (-1)^{j}\,H\Big(\beta(\epsilon z)\big(s-f_j(\epsilon z)\big)\Big)
\,-\,\frac{(-1)^{N}-1}{2},
\quad
s>0,\, z\in(0, \ell/\epsilon),
\end{equation*}
where $f_0= -f_1$.
The requirement that $f_0=-f_1$ will imply that $\,H\Big(\beta(\epsilon z)\big(s-f_0(\epsilon z)\big)\Big)$ is an even extension
of the function $-H\Big(\beta(\epsilon z)\big(s-f_1(\epsilon z)\big)\Big)$ with respect to $\partial\Omega_\epsilon$
such that $v(s,z)$ can approximately satisfy the homogeneous Neumann boundary condition.
The error is
\begin{align*}
{\bf{S}}(v)=&-\epsilon k\beta \sum_{j=1}^N (-1)^{j} H'\big(\beta(s-f_j)\big)
\,+\,\epsilon \beta_1  \sum_{j=1}^N s (H_j-H_j^3)
\\
&\,-\,\beta^2\Big( v^3-  \sum_{j=1}^N  H_j^3\Big)
\,+\,o(\epsilon).
\end{align*}
We shall adjust $f_0, \cdots, f_N$ with
$$
f_0<0<f_1<\cdots<f_N \quad \mbox{and}\quad f_0= -f_1,
$$
such that ${\bf{S}}(v)$ is approximately orthogonal to $H'\big(\beta(s-f_j)\big)$, $j=1,\cdots, N$.
It is formally equivalent to solving
\begin{align*}
e^{-  \sqrt 2 \,\beta  \,(f_1-f_0)}
\ - \    e^{-  \sqrt 2 \,\beta  \,(f_2-f_1)}
\ -\ \epsilon\frac{\sqrt 2}{12\beta}\Big[k -\frac{\beta_1}{2\beta^2}\Big]\approx 0,
\\[2mm]
 e^{-  \sqrt 2 \,\beta  \,(f_2-f_1)}
\ - \   e^{-  \sqrt 2 \,\beta  \,(f_3-f_2)}
\ -\ \epsilon\frac{\sqrt 2}{12\beta}\Big[k -\frac{\beta_1}{2\beta^2}\Big]\approx 0,
\\[2mm]
\qquad\qquad\cdots\cdots\qquad\qquad\cdots\cdots\qquad\qquad\cdots\cdots\qquad\qquad\cdots\cdots\qquad\qquad
\\[2mm]
e^{-  \sqrt 2 \,\beta  \,(f_{N-1}-f_{N-2})}
\ - \   e^{-  \sqrt 2 \,\beta  \,(f_N-f_{N-1})}
\ -\ \epsilon\frac{\sqrt 2}{12\beta}\Big[k -\frac{\beta_1}{2\beta^2}\Big]\approx 0,
\\[2mm]
e^{-  \sqrt 2 \,\beta  \,(f_N-f_{N-1})}
\ -\ \epsilon\frac{\sqrt 2}{12\beta}\Big[k -\frac{\beta_1}{2\beta^2}\Big]\approx 0.
\end{align*}
These imply that
\begin{align*}
e^{-  \sqrt 2 \,\beta  \,(f_1-f_0)}
\ \approx\ &N\epsilon\frac{\sqrt 2}{12\beta}\Big[k -\frac{\beta_1}{2\beta^2}\Big],
\\[2mm]
e^{-  \sqrt 2 \,\beta  \,(f_2-f_1)}
\ \approx\ &(N-1)\epsilon\frac{\sqrt 2}{12\beta}\Big[k -\frac{\beta_1}{2\beta^2}\Big],
\\[2mm]
\qquad\qquad\cdots\cdots\qquad&\qquad\cdots\cdots\qquad\qquad
\\[2mm]
e^{-  \sqrt 2 \,\beta  \,(f_{N-1}-f_{N-2})}
\ \approx\ &2\epsilon\frac{\sqrt 2}{12\beta}\Big[k -\frac{\beta_1}{2\beta^2}\Big],
\\[2mm]
e^{-  \sqrt 2 \,\beta  \,(f_N-f_{N-1})}
\ \approx\ &\epsilon\frac{\sqrt 2}{12\beta}\Big[k -\frac{\beta_1}{2\beta^2}\Big].
\end{align*}
Whence, we have
\begin{align*}%\label{location}
f_1\ \approx&\ \frac{1}{2\sqrt{2}\beta}\Bigg\{\ln\frac{1}{N\epsilon}
-\ln\frac{\sqrt 2}{12\beta}\Big[k -\frac{\beta_1}{2\beta^2}\Big]\Bigg\},
\qquad\qquad
f_0=-f_1,
\\[2mm]
f_{j}-f_{j-1}\ \approx&\ \frac{1}{\sqrt{2}\beta}\Bigg\{\ln\frac{1}{(N-j+1)\epsilon}
-\ln\frac{\sqrt 2}{12\beta}\Big[k -\frac{\beta_1}{2\beta^2}\Big]\Bigg\} ,\quad j= 2,\cdots, N.
\end{align*}
These asymptotic formulae will guide us to set up the parameters $f_1, \cdots, f_N$ in the sequel.

\subsection{The first approximation}
By previous formal calculations, we first set the parameters as the following
\begin{equation}\label{f01}
\begin{split}
f_0(\theta)\, =\, &    {\dot f}_0(\theta)+\mathfrak{f}_0(\theta)
\quad\mbox{with}\
{\dot f}_0(\theta)=-\frac{1}{2\sqrt{2} \beta(\theta)} \ln{\frac{1}{ N\epsilon}},
\\[2mm]
f_1(\theta)\, =\, &    {\dot f}_1(\theta)+\mathfrak{f}_1(\theta)
\quad\mbox{with}\
 {\dot f}_1(\theta)=\frac{1}{2\sqrt{2} \beta(\theta)} \ln{\frac{1}{ N\epsilon}},
\end{split}
\end{equation}
and for $j=2, \cdots, N,$
\begin{equation} \label{f2N}
f_j (\theta)\, =\,      \dot{f}_j(\theta) + \mathfrak{f}_j(\theta)
\quad\mbox{with}\
\dot{f}_j(\theta)=
\frac{1}{2\sqrt{2} \beta(\theta)} \ln{\frac{1}{ N\epsilon}}
+
\frac{1}{\sqrt{2} \beta(\theta)}  \ln {\, \frac{{\,(N-j) \,}!}{ {\, (N-1)\,}! \epsilon^{j-1} }\,}.
\end{equation}
Note that $\mathfrak{f}_0(\theta),\cdots,\mathfrak{f}_N(\theta)$ are of order $O(1)$ with respect to $\epsilon$ and will be set up step by step in the rest of the present section,
see \eqref{mathfrakfj} and \eqref{hatfj}.
For the convenience of the notation, we will also set
\begin{equation}\label{f0fminus1}
  f_{N+1}(\theta)=+\infty,
\qquad f_0=-f_1,\qquad f_{-1}=-f_2.
\end{equation}

\medskip
By recalling the definition of $H$ given in \eqref{blocksolution}, we let
\begin{equation}\label{HJX}
H_j(x_j)= (-1)^j H( x_j )
\qquad \mbox{with }\  x_j = \beta(\epsilon z)(s-f_j(\epsilon z)),
\end{equation}
\begin{equation}\label{BarHJX}
\bar{H}_j( \bar{ x}_j  )= (-1)^{j+1} H\,( \bar{ x}_j \,)
\qquad \mbox{with }\  \bar{ x}_j=  \beta(\epsilon z)(s+ f_j(\epsilon z) ),
\end{equation}
and define the {\textbf {first approximate solution}} to \eqref{S(v)} in the form
\begin{equation}
\begin{split}
 \label{firstappro1}
u_1(s, z)= & \sum_{j=1}^N   H_j(x_j)
\,+\, \frac{1}{2}\Big((-1)^{N}+ 1\Big)
\,+\, \sum_{j=1}^N  \bar{H}_j( \bar{x}_j  )
\,-\, \frac{1}{2}\Big((-1)^{N+1}+ 1\Big)
\\[1mm]
= &  \sum_{j=1}^N   H_j(x_j)
\,+\,\sum_{j=1}^N  \bar{H}_j( \bar{x}_j  )
\,+\,(-1)^N,
\quad\forall\, (s, z)\in \mathfrak{S}.
\end{split}
\end{equation}
Note that $u_1(s, z) \approx  H_j(x_j)   $ when $ | s- f_j|$ is not too large.

\medskip
The error of the first approximate solution $u_1(s,z)$ is
\begin{equation}\label{S(u1)}
\begin{split}
{\bf{S}}(u_1) \,=\, &u_{1,ss}+ u_{1,zz} +\beta^2 F(u_1)
+B_2(u_1) +B_3(u_1)
+B_4(u_1)
\\
\,=\, &u_{1,zz} +\beta^2 F(u_1)- \beta^2 \sum_{j=1}^N  \Big[ F(H_j)+F(\bar{H}_j)  \Big] +B_2(u_1) +B_3(u_1)
+B_4(u_1),
\end{split}
\end{equation}
 where we have used the fact that
 \begin{equation}
u_{1,ss} + \beta^2 \sum_{j=1}^N   \Big[ F(H_j)+F(\bar{H}_j)  \Big]\,= \, 0.
\end{equation}

\subsubsection{The computations of error terms}

It is easy to get that
$$ \big(\beta(s-f_j)\big)_z = \epsilon \Big[\beta'(s-f_j) - \beta f_j' \Big]= \epsilon \Big[\frac{\beta'}{\beta}x_j - \beta f_j' \Big], $$
$$ \big(\beta(s+ f_j)\big)_z = \epsilon \Big[\beta'(s+ f_j) + \beta f_j' \Big]= \epsilon \Big[\frac{\beta'}{\beta}\bar{x}_j +  \beta f_j' \Big], $$
and
$$
\big(\beta(s-f_j)\big)_{zz}
= \epsilon^2\Big[\beta''(s-f_j) - 2\beta'f_j' -\beta f_j''\Big]=\epsilon^2\Big[\frac{\beta''}{\beta}x_j - 2\beta'f_j' -\beta f_j''\Big],
$$
$$
\big(\beta(s+ f_j)\big)_{zz}
= \epsilon^2\Big[\beta''(s+ f_j) +  2\beta'f_j'  + \beta f_j''\Big]= \epsilon^2\Big[\frac{\beta''}{\beta} \bar{x}_j +  2\beta'f_j' + \beta f_j''\Big].
$$
The computations of all components in \eqref{S(u1)} can be done in the following way.

\medskip
\noindent {\textbf{(1).}}
The definition of $u_1$ will imply that the first term in \eqref{S(u1)} can be rewritten as follows
\begin{align}\label{uzz}
  u_{1,zz}\,=\,&\sum_{j=1}^N  H_{j,x_jx_j}  \big| \big(\beta(s-f_j)\big)_z\big|^2
  +\sum_{j=1}^NH_{j,x_j} \big(\beta(s-f_j)\big)_{zz}\nonumber
  \\
  &+  \sum_{j=1}^N  \bar{H}_{j, \bar{x}_j\bar{x}_j}  \big| \big(\beta(s+ f_j)\big)_z\big|^2
  +\sum_{j=1}^N\bar{H}_{j,\bar{x}_j} \big(\beta(s+ f_j)\big)_{zz}
  \\
=\,&\epsilon^2\sum_{j=1}^N  H_{j,x_jx_j}   \Big[\Big( \frac{\beta'}{\beta} \Big)^2 x_j^2 -2 \beta' f_j'x_j+\beta^2 (f_j')^2 \Big]
     +\epsilon^2\sum_{j=1}^NH_{j,x_j} \Big[\frac{\beta''}{\beta}x_j - 2\beta'f_j' -\beta f_j''\Big] \nonumber
     \\
  &  +  \epsilon^2\sum_{j=1}^N  \bar{H}_{j, \bar{x}_j\bar{x}_j}   \Big[\Big( \frac{\beta'}{\beta} \Big)^2 \bar{x}_j^2 + 2 \beta' f_j' \bar{x}_j+\beta^2 (f_j')^2 \Big]
     +\epsilon^2\sum_{j=1}^N\bar{H}_{j, \bar{x}_j} \Big[\frac{\beta''}{\beta} \bar{x}_j +  2\beta'f_j' + \beta f_j''\Big].  \nonumber
\end{align}

\medskip
\noindent {\textbf{(2).}}
 Recalling the expression of $B_3$ in \eqref{B3}, we obtain that
 \begin{equation}\label{B3u}
 \begin{split}
 B_3(u_1)=&-\big(\epsilon k+\epsilon^2 sk^2 \big)  u_{1,s}
 \\[1mm]
=&-\epsilon k\beta \sum_{j=1}^N H_{j,x_j}
-\epsilon^2s k^2  \beta \sum_{j=1}^N H_{j,x_j}           -\epsilon k\beta       \sum_{j=1}^N  \bar{H}_{j, \bar{x}_j}  -\epsilon^2s k^2  \beta \sum_{j=1}^N \bar{H}_{j, \bar{x}_j}
\\[1mm]
=&-\epsilon k\beta \sum_{j=1}^NH_{j,x_j}
-\epsilon^2 k^2 \sum_{j=1}^N x_jH_{j,x_j}
-\epsilon^2 k^2 \beta  \sum_{j=1}^N  f_jH_{j,x_j}
\\[1mm]
& -\epsilon k\beta \sum_{j=1}^N  \bar{H}_{j, \bar{x}_j}
-\epsilon^2 k^2   \sum_{j=1}^N  \bar{x}_j\bar{H}_{j, \bar{x}_j}
+\epsilon^2 k^2 \beta  \sum_{j=1}^N  f_j\bar{H}_{j, \bar{x}_j}.
 \end{split}
 \end{equation}

\medskip
\noindent {\textbf{(3).}}
 Recalling  the definition of $B_2$ as in \eqref{B2}, it is easy to derive that
 \begin{equation}
 	B_2(u_1)=o(\epsilon^{2+\sigma}).
 \end{equation}
This is due to the setting of $f_1,\cdots, f_N$ in \eqref{f2N}
and also the asymptotics of $H$ in \eqref{asymptoticofH}.

\medskip
\noindent {\textbf{(4).}}
 According to the expression of  $B_4$ as in \eqref{B4}, it is derived that
 \begin{align}\label{B4u}
 B_4(u_1)\,=\,&\Big[\,\epsilon s \beta_1+\frac{1}{2}\epsilon^2s^2\beta_2\, \Big]   F(u_1)     \nonumber
 \\
 \,=\,&\epsilon s\beta_1 \sum_{j=1 }^N   \big( F( H_j) + F(\bar{H}_j) \big)
+\epsilon s \beta_1 \Big[\, F(u_1) - \sum_{j=1}^N    \big( F( H_j) + F(\bar{H}_j) \big)  \, \Big]    \nonumber
 \\
 &+\frac{1}{2}\epsilon^2s^2\beta_2\, \sum_{j=1}^N \,  \big( F( H_j) + F(\bar{H}_j) \big)      \nonumber
\\
&  + \frac{1}{2}\epsilon^2s^2\beta_2
\Big[\, F(u_1) - \sum_{j=1}^N   \big( F( H_j) + F(\bar{H}_j) \big)   \, \Big]\nonumber
\\
 \,=\,&\epsilon \beta_1 \beta^{-1}\sum_{j=1}^N x_jF( H_j)
+\epsilon \beta_1\sum_{j=1}^N f_jF( H_j)
+ \epsilon s\beta_1 \sum_{j=1 }^N     F(\bar{H}_j)
\\
&+\frac{1}{2}\epsilon^2 \beta_2 \beta^{-2}  \sum_{j=1}^N x_j^2 F( H_j)
+ \epsilon^2\beta_2 \beta^{-1}  \sum_{j=1}^N x_j f_j F( H_j) \nonumber
\\
& +\frac{1}{2}\epsilon^2 \beta_2   \sum_{j=1}^N f_j^2 F( H_j)
+\frac{1}{2}\epsilon^2s^2\beta_2\sum_{j=1}^N   F(\bar{H}_j)   \nonumber
\\
&+ \epsilon \beta_1 s \Big[\, F(u_1) - \sum_{j=1}^N \big( F( H_j) + F(\bar{H}_j) \big)  \, \Big]    \nonumber
\\[1mm]
 &+\frac{1}{2} \epsilon^2\beta_2s^2
\Big[\, F(u_1) - \sum_{j=1}^N \big( F( H_j) + F(\bar{H}_j) \big)  \, \Big]. \nonumber
 \end{align}

\medskip
\noindent {\textbf{(5).}}
Now, we focus on the estimates of the nonlinear error terms
 \begin{equation}\label{nonlinearcomputations}
  F(u_1)\, -\,  \sum_{j=1}^N  \big[ F(H_j) +  F(\bar{H}_j)\big] .
 \end{equation}
We denote the following sets:
\begin{equation}\label{{mathfrak A}_1}
{\mathfrak A}_1=\left\{\, (s, z)\in \Big(0, \frac{\delta_0 }{\epsilon}\Big)\times
\Big[0, \frac{\ell }{\epsilon}\Big)
\,   :\,   0\, \leq
    s\, \leq\, \frac{  f_{1}(\epsilon z)+f_{2}(\epsilon z)  }{2}\, \right\},
\end{equation}
and for $n= 2,\cdots, N$
\begin{equation}\label{{mathfrak A}_n}
{\mathfrak A}_n=\left\{\, (s, z)\in \Big(0, \frac{\delta_0 }{\epsilon}\Big)\times
\Big[0, \frac{\ell }{\epsilon}\Big)
\,   :\,   \frac{  f_{n-1}(\epsilon z)+f_n(\epsilon z)  }{2}\, \leq\,
    s\, \leq\, \frac{  f_{n}(\epsilon z)+f_{n+1}(\epsilon z)  }{2}\, \right\}.
\end{equation}

\medskip
%First,  let us fix $n$ with  $2\leq n \leq N-1$ and do estimates for
%\begin{equation}\label{term3.14}
% F(u_1)\, -\,  \sum_{j=1}^N  \big[ F(H_j) +  F(\bar{H}_j)\big], \quad \text{if} \quad(s,z)\in {\mathfrak A}_n.
%\end{equation}
Let us fix $n$ with  $2\leq n \leq N-1$.
 Since
 \begin{equation*}
 H(x) =   \pm \big(1 - 2 e^{-\sqrt{2}|x|} \big)
 \, +\,O \big( e^{-2\sqrt{2}|x|}\big)
 \quad
 \hbox{as } x\to \pm \infty,
 \end{equation*}
there hold that
for $j  < n$,
\begin{equation}\label{r5}	
H\big(\beta(s-f_j)\big)-1
=\, -  2 e^{ -\sqrt{2}\beta( f_n  - f_j)} e^{ -\sqrt{2}   x_n}
\,+\,
O \big( e^{ - 2\sqrt{2}  |  x_n + \beta(f_n - f_j) |} \big) ,
\quad\forall\, (s,z)\in {\mathfrak A}_n,
\end{equation}
while for $j> n$,
\begin{equation}\label{r6}
H\big(\beta(s-f_j)\big)+1
=\, 2 e^{ -\sqrt{2} \beta ( f_j - f_n)} e^{ \sqrt{2}  x_n}
\,  +\,
O \big( e^{ - 2\sqrt{2}  |x_n+ \beta(f_n - f_j) |}  \big),
\quad\forall\, (s,z)\in {\mathfrak A}_n.
\end{equation}
We  derive that for any $j$ with $|j-n|\geq 3$
\begin{equation}
\big|F( H_j)\big|
\,\leq\,
C e^{ - \sqrt{2}  |x_n +\beta( f_n - f_{j}) |}
\,\leq\,
 \epsilon^{2+\sigma}\, e^{-  \sqrt 2   |x_n| },
 \quad\forall\, (s,z)\in {\mathfrak A}_n.
\label{r1}\end{equation}
for some $ \sigma > 0$ small, due to the fact
\begin{equation*}
|f_n - f_j| \,= \, \sum_{l =j}^{n-1} \, \frac{1}{\sqrt{2} \beta} \ln{\frac{1}{(N-l)\epsilon}}+O(1).
\end{equation*}
On the other hand, we have, for $(s,z)\in {\mathfrak A}_n$
\begin{align}
\begin{aligned}
F\big(H\big(\beta(s-f_{n-1})\big)\big)
&=F\big( H \big( x_n+\beta (f_n-f_{n-1})\big)\big)
\\[2mm]
&= F'( 1 ) \, {\mathbf b}_{1n} \ +\  \frac 12 F''( 1+  \lambda_1 {\mathbf b}_{1n}  ) \, {\mathbf b}_{1n}^2 ,
\label{r2}
\end{aligned}
\end{align}
\begin{align}
\begin{aligned}
F\big(H\big(\beta(s-f_{n-2})\big)\big)
&=F\big( H \big (x_n+\beta(f_n-f_{n-2})\big) \big)
\\[2mm]
&=   F'( 1 ) \, {\mathbf b}_{2n} \ +\ \frac 12 F''( 1+  \lambda_2 {\mathbf b}_{2n} ) \, {\mathbf b}_{2n}^2 ,
\label{r21}
\end{aligned}
\end{align}
\begin{align}
\begin{aligned}
F\big(H\big(\beta(s-f_{n+1})\big)\big)
&=F\big( H \big( x_n + \beta(f_n - f_{n +1})\big)\big)
\\[2mm]
&=   F'( 1)\, {\mathbf b}_{3n} \ +\ \,\frac 12 F''( 1 -  \lambda_3 {\mathbf b}_{3n}  ) \, {\mathbf b}_{3n}^2,
\label{r3}
\end{aligned}
\end{align}
\begin{align}
\begin{aligned}
F\big(H\big(\beta(s-f_{n+2})\big)\big)
&=F\big( H\big( x_n+ \beta (f_n - f_{n +2})\big) \big)
\\[2mm]
&=   F'( 1)\, {\mathbf b}_{4n} \ +\ \,\frac 12 F''( 1 -  \lambda_4 {\mathbf b}_{4n}  ) \, {\mathbf b}_{4n}^2,
\label{r31}
\end{aligned}
\end{align}
for certain constants $\lambda_1,\cdots, \lambda_4\in (0, 1)$.
Here we have denoted
\begin{align*}
{\mathbf b}_{1n} &:=  H \big( x_n+\beta (f_n-f_{n-1})\big) -1,
\qquad
{\mathbf b}_{2n} :=  H \big( x_n+\beta (f_n-f_{n-2})\big) -1,
\\[1mm]
{\mathbf b}_{3n} &:=  H \big(  x_n+ \beta(f_n - f_{n +1}) \big)+1,
\qquad
{\mathbf b}_{4n} :=  H \big(  x_n+ \beta(f_n - f_{n +2}) \big) +1.
\end{align*}
For convenience purpose, we will denote that
\begin{align}
{\mathbf A}_{ 1 n} &:=  F''( 1+  \lambda_1 {\mathbf b}_{1n}  ),
\qquad
{\mathbf A}_{2 n}  :=  F''( 1+  \lambda_2 {\mathbf b}_{2n} ),\label{A1A2}
\\[1mm]
{\mathbf A}_{3n} &:=  F''( 1-  \lambda_3 {\mathbf b}_{3n} ),
\qquad
{\mathbf A}_{4n} :=  F''( 1-  \lambda_4 {\mathbf b}_{4n} ).\label{A3A4}
\end{align}

\medskip

\medskip
 Note that, for $(s,z) \in  {\mathfrak A}_n$ with $n=1,\cdots,N$
\begin{equation}\label{expressionofu1}
(-1)^{n} u_1 =   H\big(\beta(s-f_n)\big)   - {\mathbf b}_{1n} + {\mathbf b}_{2n} - {\mathbf b}_{3n} +{\mathbf b}_{4n}  +O(\epsilon^{2+\sigma}).
\end{equation}
According to the asymptotic of these terms given in (\ref{r5})-(\ref{r6}), we obtain
\begin{equation}\label{b1+b3}
\begin{split}
&{\mathbf b}_{1n}\,-\,{\mathbf b}_{2n}\,+\,{\mathbf b}_{3n}\,-\,{\mathbf b}_{4n}
  \\[1mm]
%  &\,=\,- 2  e^{- \sqrt 2 \,x_n} e^{- \sqrt 2 \,\beta \,(f_n-f_{n-1})} +O(e^{- 2\sqrt 2  \,|x_n+\,\beta(\,f_n-f_{n-1}\,)|}) \nonumber\\[1mm]
%   &\,\quad+\, 2  e^{- \sqrt 2 \,\beta \,x_n} e^{- \sqrt 2 \,\beta \,(f_n-f_{n-2})} +O(e^{- 2\sqrt 2  \,|x_n+\,\beta(\,f_n-f_{n-2}\,)|})\nonumber
% \\[1mm]
% &\,\quad +\,  2  e^{ \sqrt 2  \,x_n} e^{- \sqrt 2 \,\beta \,(f_{n+1}-f_{n})} +O(e^{- 2\sqrt 2  \,|x_n+\,\beta( \,f_{n}-f_{n+1}\,)|})   \nonumber
% \\[1mm]
% &\,\quad +\, -2  e^{ \sqrt 2  \,x_n} e^{- \sqrt 2 \,\beta\, (f_{n+2}-f_{n})} +O(e^{- 2\sqrt 2  \,|x_n+\,\beta( \,f_{n}-f_{n+2}\,)|})  \nonumber
%   \\[1mm]
\,=&\,-2 \epsilon(N-n+1) e^{- \sqrt 2 x_n} e^{-  \sqrt 2 \beta (\mathfrak{f}_{n}-\mathfrak{f}_{n-1})}
\,+\, O \big( e^{- 2\sqrt 2 |x_n+\beta(f_n-f_{n-1})|} \big)
 \\[1mm]
  &\, + \, 2 \epsilon(N-n) \,  e^{ \sqrt 2  \,x_n} e^{- \sqrt 2 \,\beta  \,(\mathfrak{f}_{n+1}-\mathfrak{f}_n)} \,+\, O\big(e^{- 2\sqrt 2  \,|x_n+\,\beta(\,f_n-f_{n-1})|})
\\[1mm]
  &\, +\, 2 \epsilon^2(N-n+1) (N-n+2) e^{- \sqrt 2 x_n} e^{-  \sqrt 2 \beta(\mathfrak{f}_{n}-\mathfrak{f}_{n-2})} \,+\, O\big(e^{- 2\sqrt 2 |x_n+\beta(f_n-f_{n-2})|}\big)
  \\[1mm]
  &\, -\, 2  \epsilon^2 (N-n) (N-n-1) e^{ \sqrt 2  x_n} e^{-  \sqrt 2 \beta (\mathfrak{f}_{n+2}-\mathfrak{f}_{n})} \,+\, O\big(e^{- 2\sqrt 2 |x_n+\beta(f_n-f_{n+2})|}\big).
\end{split}
\end{equation}
Since we have set the notation $f_0=-f_1, f_{-1}=-f_2$ in \eqref{f0fminus1},
\eqref{b1+b3} also holds for $n=1, 2, N-1, N$.
For more details, we refer the readers to the computations in Appendix \ref{appendixA}.
Thus for some $\lambda_5\in (0, 1)$,
\begin{equation}\label{r4}
\begin{split}
(-1)^{n}F(u_1)
\,=\,&
F\big((-1)^{n}u_1\big)
\\[1mm]
\,=\,&
F\Big(H\big(\beta(s-f_n)\big)\Big)
\,-\, F'\Big( H\big(\beta(s-f_n)\big)\Big)\, \big[{\mathbf b}_{1n}-{\mathbf b}_{2n}+{\mathbf b}_{3n}-{\mathbf b}_{4n}\big]
\\[1mm]
&\,+\,
 \frac{1}{2}  {\mathbf A}_{5n}\big[{\mathbf b}_{1n}-{\mathbf b}_{2n}+{\mathbf b}_{3n}-{\mathbf b}_{4n}\big]^2+ O(\epsilon^{2+\sigma}),
\end{split}
\end{equation}
where
\begin{equation}\label{A5}
{\mathbf A}_{5n} := F''\Big( H\big(\beta(s-f_n)\big)- \lambda_5\big({\mathbf b}_{1n}-{\mathbf b}_{2n}+{\mathbf b}_{3n}-{\mathbf b}_{4n}\big)\Big).
\end{equation}

\medskip
Combining relations (\ref{r1})-(\ref{r4}) and using the facts
$$
F'(1)-F'(H) \, = \,-3\big(1-H^2\big),
\qquad
|{\mathbf b}_{1n}|+|{\mathbf b}_{2n}|+|{\mathbf b}_{3n}|+|{\mathbf b}_{4n}|
= O(\epsilon \, e^{- \sqrt 2 \, |x_n|}), $$
we obtain, for $(s,z)\in {\mathfrak A}_n$ with $n=1,\cdots,N$
\begin{equation}
\begin{split}
&(-1)^{n} \Big[\, F(u_1) - \sum_{j=1}^N \big( F(H_j) +  F(\bar{H}_j) \big) \, \Big]
\\
&\,=\,-3(1-H_n^2)\,({\mathbf b}_{1n} -{\mathbf b}_{2n}+{\mathbf b}_{3n}- {\mathbf b}_{4n})
\,+\,\frac 12 {\mathbf A}_{1n} \, {\mathbf b}_{1n}^2\,-\,\frac 12 {\mathbf A}_{2n} \, {\mathbf b}_{2n}^2
\\[1mm]
&\quad
\,+\,\frac 12 {\mathbf A}_{3n}\, {\mathbf b}_{3n}^2
\,-\,\frac 12{\mathbf A}_{4n} \, {\mathbf b}_{4n}^2\,+\,
 \frac{1}{2} {\mathbf A}_{5n}  \big[{\mathbf b}_{1n}-{\mathbf b}_{2n}+{\mathbf b}_{3n}-{\mathbf b}_{4n}\big]^2\, + \,
O\big(  \epsilon^{2+\sigma} e^{-\sigma|x_n|} \big),
\end{split}
\end{equation}
where the expressions  of ${\mathbf b}_{1n} -{\mathbf b}_{2n}+{\mathbf b}_{3n}- {\mathbf b}_{4n}$ are given in \eqref{b1+b3},
where ${\mathbf A}_{1n},\cdots, {\mathbf A}_{4n}$ are defined in \eqref{A1A2}-\eqref{A3A4}.

\medskip
We now make a conclusion for the computations of terms in \eqref{nonlinearcomputations}.
For fixed $n$ with $2\,\leq \,n\leq \, N$, let the sets $\mathfrak{B}_n$ be in the form
\begin{equation*}
{\mathfrak B}_n=\left\{\, (s, z)\in \Big(0, \frac{\delta_0 }{\epsilon}\Big)\times
\Big[0, \frac{\ell }{\epsilon}\Big)
\,   :\,   \frac{  f_{n-1}(\epsilon z)+f_n(\epsilon z)  }{2}\,-{\tilde\delta}\, \leq\,
    s\, \leq\, \frac{  f_{n}(\epsilon z)+f_{n+1}(\epsilon z)  }{2}\,+{\tilde\delta}\, \right\},
\end{equation*}
and
\begin{equation*}
{\mathfrak B}_1=\left\{\, (s, z)\in \Big(-{\tilde\delta}, \frac{\delta_0 }{\epsilon}\Big)\times
\Big[0, \frac{\ell }{\epsilon}\Big)
\,   :\,    -{\tilde\delta}  \, \leq\,
    s\, \leq\, \frac{ f_{1}(\epsilon z)+f_{2}(\epsilon z)  }{2}\,+{\tilde\delta}\, \right\},
\end{equation*}
with constant ${\tilde\delta}>0$ sufficient small.
We denote a smooth partition of unity subordinate  $\{\chi_n\}$ to the open sets $\{{\mathfrak B}_n\,:\, n=1,\cdots, N\}$, satisfy
\begin{equation}\label{cutoffchi}
\begin{split}
&\chi_n(s,z)\, = \, 1\quad\mbox{when} \quad (s,z) \in {\mathfrak A}_n,
\qquad
0\,\leq \,\chi_n(s,z)\, \leq \, 1,
\\
&\chi_n(s,z)\in C_0^{\infty}({\mathfrak B}_n)
\qquad \text{and} \qquad
\sum_{n=1}^N\chi_n(s,z)\,=\, 1\quad \mbox{on }   {\mathfrak S}.
\end{split}
\end{equation}
Combining  all the terms in the above, we know that in the whole strip ${\mathfrak S}$
\begin{equation}\label{expressnon}
\begin{split}
&F(u_1) - \sum_{j=1}^N \big[ F(H_j) +  F(\bar{H}_j)\big]  \, = \,  \sum_{n=1}^N \,  \chi_n \, \Big[\, F(u_1) - \sum_{j=1}^N \big( F(H_j) +  F(\bar{H}_j) \big) \, \Big]
\\
&\,=\sum_{n=1}^N(-1)^n  \, \chi_n \, \Big\{\,-3(1-H_n^2)\,({\mathbf b}_{1n} -{\mathbf b}_{2n}+{\mathbf b}_{3n}- {\mathbf b}_{4n})
\,+\,\frac 12 {\mathbf A}_{1n} \, {\mathbf b}_{1n}^2\,-\,\frac 12 {\mathbf A}_{2n} \, {\mathbf b}_{2n}^2
\\[1mm]
&\qquad\qquad\qquad \quad
+\frac 12 {\mathbf A}_{3n} {\mathbf b}_{3n}^2
-\frac 12{\mathbf A}_{4n} {\mathbf b}_{4n}^2
+\frac{1}{2} {\mathbf A}_{5n}  \big[{\mathbf b}_{1n}-{\mathbf b}_{2n}+{\mathbf b}_{3n}-{\mathbf b}_{4n}\big]^2
 \,\Big\} +
O (  \epsilon^{2+\sigma} ).
\end{split}
\end{equation}
Note that \eqref{expressnon} will give the estimates of  $\beta^2 F(u_1)- \beta^2 \sum_{j=1}^N  \Big[ F(H_j)+F(\bar{H}_j)  \Big]$ in \eqref{S(u1)} and also the estimates for the corresponding components in \eqref{B4u}.

\subsubsection{Rearrangements of the error terms}

In  this part, by substituting all the above computations into \eqref{S(u1)},
 we will rearrange the error components in ${\bf S} (u_1)$ according to the order of $\epsilon$.

\medskip
According to the asymptotic behavior  of $f_j, \,j=1,\cdots, N $ as in \eqref{f01}-\eqref{f2N}, we know that the largest terms in ${\bf S} (u_1)$ are of order  $\epsilon |\ln \epsilon|$, which can be written as the following
\begin{equation*}
 \epsilon |\ln\epsilon| E_{01} \, = \,  \epsilon \, \beta_1 \, \sum_{j=1}^N \, { f}_j \, \Xi_{0,j},
\end{equation*}
where we have denoted
\begin{equation}\label{xi1j}
\Xi_{0,j} \, =\,   F( H_j)    \, =\,    H_j - H_j^3,\qquad j=1,\cdots, N.
\end{equation}

\medskip
Recalling the expression of ${\mathbf b}_{1j} $ and ${\mathbf b}_{3j}$ as in \eqref{b1} and \eqref{b3}, we can select the terms of order $\epsilon$ in ${\bf S}(u_1)$, which can be defined by
\begin{equation*}
\epsilon \sum_{j=1}^N \Xi_{1,j} = \epsilon \sum_{j=1}^N \Big[- k\beta H_{j,x_j}
+\, \beta^{-1}\, \beta_1  \,x_j   F( H_j)\Big],
%\,=\,-\epsilon k\beta \sum_{j=1}^N H_{j,x_j}
%+\epsilon\, \beta^{-1}\, \beta_1   \sum_{j=1}^N  \,x_j F( H_j) ,
\end{equation*}
and
\begin{equation*}
\epsilon E_{02}  \,=\,-\,3 \beta^2  \sum_{j=1}^N (-1)^{j} \chi_j  (1-H_j^2)\,({\mathbf b}_{1j} +{\mathbf b}_{3j} ).
\end{equation*}

The terms of order $\epsilon^2 |\ln\epsilon|^2 $ in ${\bf S}(u_1)$ can be defined as
\begin{equation}\label{E03}
 \epsilon^2 |\ln\epsilon|^2 E_{03}  = \epsilon^2 \beta^2  \sum_{j=1}^N \,  {f'_j }^2H_{j,x_jx_j}    + \frac 12 \epsilon^2 \beta_2    \sum_{j=1}^N  f_j^2 F(H_j).
\end{equation}
This is due to the definitions of $f_1,\cdots, f_N$ in \eqref{f2N}.

\medskip
We can also select the $\epsilon^2 |\ln \epsilon|$ terms in ${\bf S}(u_1)$, which can be expressed as
\begin{equation}\label{E04}
\begin{split}
\epsilon^2 |\ln\epsilon| \, E_{04}
=\,&
\epsilon^2 \, \sum_{j=1}^N \Big[  \, (  -2 \beta' f'_j - \beta f''_j  )  \, H_{j,x_j}\, -  2\, \beta' \, f'_j\, x_j   \, H_{j,x_jx_j}  \Big]
\\
&  + \epsilon^2  \beta_2 \beta^{-1}  \sum_{j=1}^N   \,f_jx_j F(H_j)-\epsilon^2  k^2  \beta   \sum_{j=1}^N   f_jH_{j,x_j},
\end{split}
\end{equation}
 and
\begin{equation}\label{E05}
\epsilon^2 |\ln \epsilon|\sum_{j=1}^N \Xi_{2,j}
\,=\, - \epsilon  k\beta \sum_{j=1}^N  \bar{H}_{j, \bar{x}_j}
-\epsilon \beta_1 s\sum_{j=1 }^N  F(\bar{H}_j).
\end{equation}

For terms of order $\epsilon^2$, we set
\begin{equation}\label{E06}
\begin{split}
\epsilon^2 E_{06}
\,=\, &\epsilon^2  \sum_{j=1}^N \Big[\Big(  \frac{{\beta'} }{\beta} \Big)^2 x_j^2H_{j,x_jx_j}+\frac{\beta'' }{\beta}x_jH_{j,x_j} \Big]
\\
&+ \frac{1}{ 2 }  \epsilon^2\beta_2 \beta^{-2} \sum_{j=1}^N  x_j^2 F(H_j)-\epsilon^2  k^2   \sum_{j=1}^N x_j H_{j,x_j},
\end{split}
\end{equation}
and also
\begin{align}\label{E032}
  \, E_{07}
\,&=\,\beta^2 \sum_{j=1}^N (-1)^{j} \chi_j \Big\{3(1-H_j^2)({\mathbf b}_{2j}+ {\mathbf b}_{4j})
+\frac 12 {\mathbf A}_{1j}  {\mathbf b}_{1j}^2 \nonumber
\\
& \qquad \qquad \qquad \qquad+\frac 12 {\mathbf A}_{3j}\, {\mathbf b}_{3j}^2
+ \frac{1}{2} {\mathbf A}_{5j}\big[{\mathbf b}_{1j}+{\mathbf b}_{3j}\big]^2   \Big\}
\\
&\quad    \,+\,\epsilon \beta_1\beta^2  \sum_{j=1}^N   (-1)^{j} \chi_j \Big( \frac{x_j }{\beta} + { f}_j\Big)\Big[-3(1-H_j^2)\,({\mathbf b}_{1j} +{\mathbf b}_{3j})  \Big].\nonumber
\end{align}

From the definition of $\bar{x}_j=\beta (s+f_j)$ and the asymptotic of $H$ as in \eqref{asymptoticofH}, it is  easy to derive that
\begin{align*}
E_{08}\,=\,&\epsilon^2\sum_{j=1}^N  H_{j, \bar{x}_j\bar{x}_j}   \Big[\Big( \frac{\beta'}{\beta} \Big)^2 \bar{x}_j^2 + 2 \beta' f_j' \bar{x}_j+\beta^2 (f_j')^2 \Big]
+\epsilon^2\sum_{j=1}^N\bar{H}_{j, \bar{x}_j} \Big[\frac{\beta''}{\beta} \bar{x}_j +  2\beta'f_j' + \beta f_j''\Big]
\\
&-\epsilon^2 k^2   \sum_{j=1}^N  \bar{x}_j\bar{H}_{j, \bar{x}_j}
+\epsilon^2 k^2 \beta  \sum_{j=1}^N  f_j\bar{H}_{j, \bar{x}_j}
+\frac{1}{2}\epsilon^2s^2\beta_2\sum_{j=1}^N   F(\bar{H}_j)
\\
&+\frac{1}{2}\epsilon^2\beta_2s^2\Big[F(u_1) - \sum_{j=1}^N \big[F( H_j)+ F(\bar{H}_j)\big]\Big]
+B_2(u_1)
\\
\,=\,&O(\epsilon^{2+\sigma}).
\end{align*}

Combining the above expressions of components for ${\bf S} (u_1)$, we can rewrite the error for $(s, z)\in {\mathfrak S}$ in the form
\begin{equation}\label{S(u12)}
\begin{split}
{\bf S} (u_1) \, =&\,   \epsilon \, \beta_1 \, \sum_{j=1}^N \,\big( \dot{f}_j+  \mathfrak{ f}_j\big)  \, \Xi_{0,j} + \epsilon \sum_{j=1}^N \Xi_{1,j}   +\epsilon E_{02}+ \epsilon^2 |\ln\epsilon|^2\, E_{03} \nonumber
\\
\quad & +\epsilon^2 |\ln\epsilon|  \, E_{04}+ \epsilon^2 |\ln\epsilon| \sum_{j=1}^{N}\Xi_{2,j}+\epsilon^2 E_{06} +  E_{07} +    O(\epsilon^{2+\sigma}).
\end{split}
\end{equation}
By the symmetry, we have the  boundary error of $u_1$
\begin{equation}\label{boundary1}
\begin{split}
 {\mathcal{D}}(u_1) &= -\frac{\partial u_1}{\partial s}\big|_{s=0} = -\beta \sum_{j=1}^N   \, \big[  H_{j,x_j}    + \bar{H}_{j, \bar{x}_j}  \big]  \big|_{s=0}
 \\
 &=-\beta \sum_{j=1}^N   \,(-1)^{j} \big[  H'(\beta(s-f_j))   - H'(\beta(s+f_j)) \big]  \big|_{s=0}=0.
\end{split}
\end{equation}

\subsubsection{More precise estimates of the error ${\bf S}(u_1)$}
For more precise estimates of the error terms, we will set
\begin{equation}\label{mathfrakfj}
\mathfrak{f}_j ( \theta) ={\bar f}_j(\theta) +\hat{f}_j( \theta),
\quad
\forall\, j=1,\cdots,N, \qquad 	
\end{equation}
and \begin{equation*}
%f_0(\theta) = -f_1 (\theta), \quad
 {\bar f}_0(\theta)= - {\bar f}_1(\theta), \quad \hat{f}_0( \theta)=  -\hat{f}_1( \theta).
\end{equation*}
Here ${\bar f}_1,\cdots, {\bar f}_N$  are of order $O(1)$  and  will be determined by solving system \eqref{berf},
see \eqref{barfexpression1}-\eqref{barfexpressionN}. On the other hand,
$\hat{f}_j$'s will be chosen  with the constraints (c.f. \eqref{hatfj}-\eqref{ddotfnorm})
\begin{equation}\label{epsilon12}
 \|{\hat f}_j\|_\infty\leq C\epsilon^{1/2},
\quad
\forall\, j=1,\cdots,N.
\end{equation}

\medskip
Recalling the expressions of ${\mathbf b}_{1,j}, {\mathbf b}_{3,j}$ as in \eqref{b1} and \eqref{b3}, we can  rewrite $ \epsilon E_{02} $  as following
\begin{align*}
\epsilon E_{02}  \,=\,&-\,3 \beta^2  \sum_{j=1}^N (-1)^{j} \chi_j  (1-H_j^2)\,({\mathbf b}_{1j} +{\mathbf b}_{3j} )
\,=\, \epsilon E_{02}^1+ \epsilon^2 E_{02}^2,
\end{align*}{}
where
\begin{align*}
 \epsilon E_{02}^1
\, :=\, -  6\sqrt 2 \epsilon \,\beta^2\, \sum_{j=1}^N\, { \chi_j}  H_{j,x_j} \, \Big[&- (N-j+1) \,  e^{- \sqrt 2 \,x_j } e^{\,-  \sqrt 2 \,\beta  \,(\mathfrak{f}_{j}-\mathfrak{f}_{j-1})\, }
\\
& +  (N-j) \,  e^{ \sqrt 2  \,x_j} e^{\,-  \sqrt 2 \,\beta  \,(\mathfrak{f}_{j+1}-\mathfrak{f}_j)\,}  \Big],
\end{align*}
and
\begin{align*}
\epsilon^2 E_{02}^2:
 = - 6 \sqrt 2 \epsilon^2  \beta^2  	\sum_{j=1}^N   \,\chi_j  H_{j,x_j} \,\Big[ &  (N-j+1)^2  e^{- 2\sqrt 2 \,x_j } e^{\,-  2\sqrt 2 \,\beta  \,( \mathfrak{f}_{j}- \mathfrak{f}_{j-1})\, }
\\
&-  (N-j )^2 e^{ 2\sqrt 2  \,x_j} e^{\,-  2\sqrt 2 \,\beta  \,( \mathfrak{f}_{j+1}- \mathfrak{f}_j)\,} \Big]   + O( \,  \epsilon^{2+\sigma}  \,   ).
\end{align*}

%Now, for every fixed $n$, $2\,\leq \,n\leq \, N$, let the sets $\mathfrak{B}_n$ be in the form
%\begin{equation*}
%{\mathfrak B}_n=\left\{\, (s, z)\in \Big(0, \frac{\delta_0}{\epsilon}\Big)\times
%\Big(0, \frac{\ell }{\epsilon}\Big)
%\,   :\,   \frac{  f_{n-1}(\epsilon z)+f_n(\epsilon z)  }{2}\,-{\tilde\delta}\, \leq\,
%    s\, \leq\, \frac{  f_{n}(\epsilon z)+f_{n+1}(\epsilon z)  }{2}\,+{\tilde\delta}\, \right\},
%\end{equation*}
%and when $n=1$
%\begin{equation*}
%{\mathfrak B}_1=\left\{\, (s, z)\in \Big(0, \frac{\delta_0}{\epsilon}\Big)\times
%\Big(0, \frac{\ell }{\epsilon}\Big)
%\,   :\,   0\,-{\tilde\delta}\, \leq\,
%    s\, \leq\, \frac{ f_{1}(\epsilon z)+f_{2}(\epsilon z)  }{2}\,+{\tilde\delta}\, \right\},
%\end{equation*}
%with  universal constant ${\tilde\delta}>0$.
%We denote $\{\chi_n\}$ a smooth partition of unity subordinate to the open sets $\{{\mathfrak B}_n\}_{n=1}^{=\infty} $, that is
%\begin{equation}\label{cutoffchi}
%\chi_n\, = \, 1\quad\mbox{in}\ {\mathfrak A}_n,
%\qquad
%0\,\leq \,\chi_n\, \leq \, 1,
%\qquad
%\chi_n\in C_0^{\infty}({\mathfrak B}_n),
%\qquad
%\sum_{n=1}^N\chi_n\,=\, 1\quad \mbox{on} \ (0,\,\delta_0/\epsilon).
%\end{equation}

\medskip
We now decompose $ \epsilon E_{02}^1$.
For $j=1, \cdots, N$, let
\begin{align}
\label{ga0n}
\gamma_{0,j}\,:= &\,\int_{\R}\,\chi_{j}\,H_{j,x_j}^2\, {\mathrm{d}}x_j,
\\
\label{ga1n}
\gamma_{1,j}\,:=&\, \int_{\R}\,\chi_{j}\,H_{j,x_j}^2\, e^{\,-\sqrt{2}x_j\,}\, {\mathrm{d}}x_j,
\\
\label{ga2n}
\gamma_{2,j}\,:=&\, \int_{\R}\,\chi_{j}\,H_{j,x_j}^2\, e^{\,\sqrt{2}\, x_j\,}\, {\mathrm{d}}x_j.
\end{align}
It is easy to derive the following relations, for all $j=1, \cdots, N$,
\begin{equation}\label{relationofgamma_ingamma_i}
\gamma_{0,j}\,=\, \gamma_0\,+\,O(\,\epsilon\, ),
\qquad
\gamma_{i,j}\,=\,\gamma_1\,+\,O\big( \sqrt{\epsilon}\big), \quad i=1,2,
\end{equation}
where $ \gamma_0, \gamma_1$ are defined  as in  \eqref{gamma01}.
For $j=1, \cdots, N+1$, we set the functions
 \begin{equation}\label{mathbfdn}
{\mathbf d}_j(\theta)\,\,:=\, \left\{
\begin{array}{cr}
  \, (N-j+1)\, e^{-\sqrt{2}\, \beta(\theta)\, \big({\bar f}_{j}(\theta)-{\bar f}_{j-1}(\theta)\big)},  \quad &\text{when $  j=1, \cdots, N,  $ }
\\[1mm]
 0,  \, \quad &\text{  when $ j \, = \,   N+1$. }
\end{array}
\right.
\end{equation}
Therefore, by recalling $\mathfrak{f}_j$'s in \eqref{mathfrakfj}, we can  further decompose   $ \, \epsilon \, E_{02}^1$ as the following
\begin{align*}
\epsilon E_{02}^1
\,&=\,  - 6 \sqrt 2 \epsilon  \beta^2 \sum_{j=1}^N\, \chi_{j}  H_{j,x_j} \Big[-  e^{- \sqrt 2 x_j} {\mathbf d}_{j}  e^{-\sqrt{2} \beta ({\hat f}_{j}-{\hat f}_{j-1})}
+    e^{ \sqrt 2 x_j} {\mathbf d}_{j+1} e^{-  \sqrt 2 \beta (\hat{f}_{j+1}-\hat{f}_j)}  \Big]
\\[1mm]
&\,=\,\epsilon \, \sum_{j=1}^N\, \chi _{j}\Xi_{3,j}
\,+\,
\epsilon \, \sum_{j=1}^N\, \chi_{j}\Xi_{4,j}
\,+\,
\epsilon \, \sum_{j=1}^N\, \chi_{j}\Xi_{5,j}.
\end{align*}
In the above, we have denoted
\begin{align}
\begin{aligned}\label{Xi4}
 \Xi_{3,j}(x_j,z)
:=&
 -  \, 6\sqrt{2}  \beta^2 H_{j,x_j}
\Big[{- \mathbf d}_{j } e^{-\sqrt{2}\beta ({\hat f}_{j}-{\hat f}_{j-1})} e^{ -\sqrt{2} x_j }
 + {\mathbf d}_{j+1} e^{-  \sqrt 2 \beta (\hat{f}_{j+1}-\hat{f}_j)} e^{ \sqrt{2} x_j } \Big]
\\[2mm]
&+
\, 6\sqrt{2} \beta^2 H_{j,x_j}
\Big[- \frac{ \gamma_{1,j} }{\gamma_{0,j}} {\mathbf d}_{j}e^{-\sqrt{2} \beta ({\hat f}_{j}-{\hat f}_{j-1})}
 +\frac{ \gamma_{2,j} }{\gamma_{0,j}}{\mathbf d}_{j+1} e^{-  \sqrt 2 \beta (\hat{f}_{j+1}-\hat{f}_j)}  \Big],
\end{aligned}
\end{align}
in such a way that
\begin{equation}\label{orthorgnalityXi5n}
\int_{\R}\,{\chi_{j} \,\Xi_{3,j}(x_j, z) H_{j,x_j}}\,{\mathrm{d}}x_j\,=\,0,
\end{equation}
due to the definitions of $\gamma_{0,j}$, $\gamma_{1,j}$, $\gamma_{2,j}$ in (\ref{ga0n}), (\ref{ga1n}), \eqref{ga2n}.
Similarly, we have   also  set
\begin{align}
-6\sqrt{2} \beta^2 H_{j,x_j}
\Big[-\frac{ \gamma_{1,j} }{\gamma_{0,j}} {\mathbf d}_{j}e^{-\sqrt 2\beta(\hat{f}_{j}-\hat{f}_{j-1})}
 +  \frac{ \gamma_{2,j} }{\gamma_{0,j}} {\mathbf d}_{j+1} e^{-\sqrt 2\beta(\hat{f}_{j+1}-\hat{f}_j)}  \Big]\nonumber
\,=\, \Xi_{4,j}+ \Xi_{5,j},
\end{align}
with
\begin{align}
 \Xi_{4,j}(x_j, z) :=-\frac{1}{\gamma_{0,j}}\, 6\sqrt{2}  \, \beta^2 \, H_{j,x_j}\,\big( - \gamma_{1,j} {\mathbf d}_{j}
 \,+\, \gamma_{2,j}{\mathbf d}_{j+1} \big),
 \end{align}
\begin{equation}\label{Xi7n}	
\begin{split}
\, \Xi_{5,j}(x_j, z) :=
-\frac{ 1}{\gamma_{0,j}}\, 6\sqrt{2} \beta^2 \,H_{j,x_j}
 \Big[ & - \gamma_{1,j}\, {\mathbf d}_{j}\,  \big(\,e^{\,-  \sqrt 2 \,\beta  \,(\hat{f}_{j}-\hat{f}_{j-1})\,}-1 \, \big)
 \\[1mm]
 \,&+ \, \gamma_{2,j}\, {\mathbf d}_{j+1} \, \big(  e^{\,-  \sqrt 2 \,\beta  \,(\hat{f}_{j+1}-\hat{f}_j)\,}  -1 \big) \Big].
 \end{split}
\end{equation}

On the other hand, we can also rewrite $ \epsilon^2 E_{02}^2$ as follows
\begin{align}\label{Xi6j}
 \epsilon^2 E_{02}^2 = & - 6 \sqrt 2\epsilon^2\beta^2  	 \sum_{j=1}^N   \,{ \chi_j}  H_{j,x_j} \,\Big[    {\mathbf d}_j^2  e^{- 2\sqrt 2 \,x_j } e^{\,-  2\sqrt 2 \,\beta  \,( \hat{f}_{j}- \hat{f}_{j-1})\, } \nonumber
\\[1mm]
&\qquad \qquad \qquad \qquad \qquad \quad
 - {\mathbf d}_{j+1}^2 e^{ 2\sqrt 2  \,x_j} e^{\,-  2\sqrt 2 \,\beta  \,( \hat{f}_{j+1}- \hat{f}_j)\,} \Big]  + O( \,  \epsilon^{2+\sigma}  \, )
\\[1mm]
 := \,&   \epsilon^2 |\ln \epsilon| \sum_{j=1}^N {\chi_j}  \Xi_{6,j} +     O( \,  \epsilon^{2+\sigma}  \,  ).\nonumber
 \end{align}

\medskip
From the above expression of components, we see that
\begin{equation}\label{S(u13)}
\begin{split}
{\bf S} (u_1) \, =&\,   \epsilon \, \beta_1 \, \sum_{j=1}^N \, ({\dot f}_j+{\bar f}_j+{\hat f}_j) \, \Xi_{0,j} + \epsilon \sum_{j=1}^N \Xi_{1,j}+\epsilon \sum_{j=1}^N  \chi_{j}\Xi_{3,j}
\\[1mm]
\quad &+\epsilon \sum_{j=1}^N {\chi_j}  \Xi_{4,j}
+\epsilon \sum_{j=1}^N {\chi_j} \Xi_{5,j}
+ \epsilon^2 |\ln\epsilon| \sum_{j=1}^{N}\Xi_{2,j}
+  \epsilon^2 |\ln \epsilon| \sum_{j=1}^N {\chi_j} \Xi_{6,j}
\\[2mm]
\quad & + \epsilon^2 |\ln\epsilon|^2\, E_{03}+\epsilon^2 |\ln\epsilon|  \, E_{04}
+\epsilon^2 E_{06} +   E_{07} +    O(\epsilon^{2+\sigma}).
\end{split}
\end{equation}

\subsection{The second approximation}
By adding correction terms, we now want to construct a further approximation to a solution which
eliminates the terms of orders $O(\epsilon|\ln \epsilon|)$ and  $O(\epsilon)$, and then compute the new error.

\subsubsection{Choosing correction terms}
To cancel the first term   $\epsilon \, \beta_1 \, \sum_{j=1}^N \dot{  f}_j \, \Xi_{0,j}$, we can choose the  correction term  as
\begin{equation}\label{phi11}
\begin{split}
\epsilon |\ln \epsilon|   \phi_{11}(s,z)
:\,=\,& \epsilon \frac{\beta_1}{\beta^2} \sum_{j=1}^N\dot{ f}_j(\epsilon z) \psi_j ( x_j)
+\epsilon \frac{\beta_1}{\beta^2}  \sum_{j=1}^N \dot {f}_j(\epsilon z)   \psi_j  (-  x_j-2\beta f_j  )
\\
\,=\, &\epsilon \frac{\beta_1}{\beta^2} \sum_{j=1}^N\dot{ f}_j(\epsilon z) \psi_j (\beta(s-f_j))
+\epsilon \frac{\beta_1}{\beta^2}  \sum_{j=1}^N \dot {f}_j(\epsilon z)   \psi_j  (-  \beta(s+f_j) ),
\end{split}
\end{equation}
where  $ \psi_j (x_j)=  (-1)^j \, \psi ( x_j )$ and   the function $\psi (x) $  is the unique,  odd solution  to the problem
\begin{equation} \label{equationphi11}
 \, \psi_{ xx}  +(1-3\, H ^2)\, \psi \,  = -\,( H- H^3)  \quad \text{in }  \mathbb{R},
\qquad
 \psi (\pm \infty )\,=\,0,
\qquad
\int_{{\mathbb R}}\,\psi(x)   H_{x } \,{\mathrm{d}}x\,=\,0.
\end{equation}
Indeed, it is easy to see that
 \begin{equation*}
 \psi (x) = \frac 12 x H_x.
\end{equation*}
Since the Neumann boundary condition  is required, we have added the  terms of   $   \psi_j  (-  \beta(s+f_j) )$  in \eqref{phi11} so that
\begin{equation*}
\epsilon |\ln \epsilon|   \, \frac{\partial  \phi_{11}(s,z)}{\partial  s } \Big|_{s=0} \, =\,  0.
\end{equation*}

\medskip
In order to cancel the term $\epsilon \chi_j\Xi_{3,j}$,  we will choose
\begin{equation}\label{varepsilonpsi*}
\begin{split}
\epsilon\,\phi_{12}(s,z)
\, :=\,&
\epsilon\, \sum_{j=1}^N \,  \psi^{*}_j(x_j, z)
+ \epsilon  \sum_{j=1}^N \psi^{*}_j(- x_j- 2\beta f_j, z)
\\
=\,&\epsilon\, \sum_{j=1}^N \,  \psi^{*}_j(\beta(s-f_j), z)
+ \epsilon  \sum_{j=1}^N \psi^{*}_j(- \beta(s+f_j), z),
\end{split}	
\end{equation}
where the function $\psi^{*}_j(x_j, z)$ is a unique solution to the following problem
\begin{equation}\label{definitionofpsij}
\psi^{*}_{j,zz}+ \beta^{2}\,\Big[\psi^{*}_{j,x_jx_j}\, +\, {F'}(H_j)\psi^{*}_j\Big]\,=\,-  \chi_{j}\Xi_{3,j}(x_j, z)\quad \mbox{ in } \mathbb{R} \times \Big[0,\frac{\ell}{\epsilon} \Big),
\end{equation}
\begin{equation}\label{condition1}
\psi^{*}_j(x_j, 0)=\psi^{*}_j\Big(x_j, \frac{\ell}{\epsilon}\Big),
\quad
\psi^{*}_{j,z}(x_j, 0)=\psi^{*}_{j,z}\Big(x_j, \frac{\ell}{\epsilon}\Big),
\end{equation}
\begin{equation}\label{condition2}
\int_\mathbb{R}{\psi^{*}_j H_{j, x_j}}\,{\mathrm{d}}x_j=0,\quad 0<z<\frac{\ell}{\epsilon}.
\end{equation}
 As a direct consequence of Lemma \ref{lemma of solve linearized pro11} and (\ref{orthorgnalityXi5n}),
 we know that there exists  a unique solution $\psi_j^{*}(x_j, z)$  to the problem \eqref{definitionofpsij}-\eqref{condition2}.
Moreover, we can get
\begin{align*}
\epsilon   \, \frac{\partial  \phi_{12}(s,z)}{\partial  s } \Big|_{s=0} \, =\,  0.
\end{align*}

 \medskip
  On the other hand, if the terms
\begin{equation*}
\epsilon \beta_1 \bar{f}_j \Xi_{0,j}+ \epsilon\,\Xi_{1,j}
\,+\,\epsilon\,\chi_{j}\Xi_{4,j}, \quad\forall\,  j=1,\cdots,N,
\end{equation*}
can be cancelled in a similar way, we shall choose ${\bar f}_j$'s such that
\begin{align}\label{eqaution1}
\int_\mathbb{R}\Big[\,\epsilon \beta_1 \bar{f}_j \Xi_{0,j}+
\epsilon\,\Xi_{1,j}
\,+\,\epsilon\,\chi_{j}\Xi_{4,j}\,\Big]
H_{j,x_j}\,{\mathrm{d}}x_j\,=\,0,
\quad\forall\, j=1,\cdots,N.
\end{align}
In fact, the equalities in \eqref{eqaution1}  are equivalent to
\begin{equation}\label{berf}
\begin{split}
\frac{2\sqrt 2}{3}k -\frac{\sqrt 2}{3}\frac{\beta_1}{\beta^{2}}
=6\sqrt 2\beta^2   \Big[ (N-j+1)\, \gamma_{1,j}\,  e^{-  \sqrt 2 \,\beta  \,(\bar{f}_{j}-\bar{f}_{j-1} )}- (N-j)\,\gamma_{2,j} \,   e^{-  \sqrt 2 \,\beta  \,(\bar{f}_{j+1}-\bar{f}_j )}       \Big]
\end{split}
\end{equation}
for $ j= 1,\cdots,N$, where $\gamma_{i,j}, i =1,2$ are defined  in \eqref{ga1n}, \eqref{ga2n}.
The relations in \eqref{berf} can be rewritten in the forms
\begin{align}
\frac{2\sqrt 2}{3}k -\frac{\sqrt 2}{3}\frac{\beta_1}{\beta^{2}}
&= 6\sqrt 2  \beta^2    \,  \Big[  N\, \gamma_{1,1}\,  e^{-  \sqrt 2 \,\beta  \,(\bar{f}_{1}-\bar{f}_{0} )}- (N-1)\,\gamma_{2,1} \,   e^{-  \sqrt 2 \,\beta  \,(\bar{f}_{2}-\bar{f}_1 )}       \Big],  \label{equationoff1}
\\[2mm]
\frac{2\sqrt 2}{3}k -\frac{\sqrt 2}{3}\frac{\beta_1}{\beta^{2}}
&= 6\sqrt 2  \beta^2  \,  \Big[  (N-1)\, \gamma_{1,2}\,  e^{-  \sqrt 2 \,\beta  \,(\bar{f}_{2}-\bar{f}_{1} )}- (N-2)\,\gamma_{2,2} \,   e^{-  \sqrt 2 \,\beta  \,(\bar{f}_{3}-\bar{f}_2 )}       \Big],
\\[2mm]
\cdots  \cdots \cdots \cdots \cdots \nonumber
\\[2mm]
\frac{2\sqrt 2}{3}k -\frac{\sqrt 2}{3}\frac{\beta_1}{\beta^{2}}
&= 6\sqrt 2  \beta^2 \,  \Big[   2 \gamma_{1,N-1}\,  e^{-  \sqrt 2 \,\beta  \,(\bar{f}_{N-1}-\bar{f}_{N-2} )}- \,\gamma_{2,N-1}e^{-  \sqrt 2 \,\beta  \,(\bar{f}_{N}-\bar{f}_{N-1} )} \Big],
\\[2mm]
\frac{2\sqrt 2}{3}k -\frac{\sqrt 2}{3}\frac{\beta_1}{\beta^{2}}
&=   6\sqrt 2   \beta^2  \, \gamma_{1,N}\,  e^{-  \sqrt 2 \,\beta  \,(\bar{f}_{N}-\bar{f}_{N-1} )}.\label{equationoffN}
\end{align}
Combining  the assumption in \eqref{meancurvaturepositive}, the constraint $ \bar{f}_0= -\bar{f}_1$
and  the relations \eqref{relationofgamma_ingamma_i}, there exists a unique solution $\bar{\bf f} =(\bar{f}_1, \cdots, \bar{f}_N)^{T}$ to the nonlinear algebraic system \eqref{equationoff1}-\eqref{equationoffN}. Moreover, we can conclude that
\begin{align}\label{barfexpression1}
 \bar{f}_0= -\bar{f}_1 \,  \approx  \,  \frac{1}{2\sqrt 2 \beta} \Big[ \ln\Big( \, k  -\frac{\beta_1}{2\beta^{2}}\,    \Big) - \ln (9\,\gamma_1\,\beta^2) \Big],
\end{align}
\begin{align}\label{barfexpression2}
\bar{f}_2  \approx -  \frac{3}{2\sqrt 2 \beta}  \Big[  \ln \Big(  \, k  -\frac{\beta_1}{2\beta^{2}}\,    \Big) - \ln (9\,\gamma_1\,\beta^2) \Big],
\end{align}
$$ \cdots  \cdots \cdots \cdots \cdots
$$
\begin{align}\label{barfexpressionN}
\bar{f}_N  \approx -  \frac{2N-1}{2\sqrt 2 \beta}  \Big[ \ln \Big(  k  -\frac{\beta_1}{2\beta^{2}}\Big)  - \ln (9\,\gamma_1\,\beta^2)  \Big],
\end{align}
where $\gamma_1$ is the constant defined in \eqref{gamma01}.

\medskip
Then, by \eqref{eqaution1} we can choose
\begin{equation}
\begin{split}
\epsilon\,\phi_{13}(s,z)
\, :=\,&
\epsilon\, \sum_{j=1}^N \,  \psi^{**}_j(x_j, z)
+ \epsilon \sum_{j=1}^N \,  \psi^{**}_j( -x_j -2\beta f_j, z)
\\
=\,&\epsilon\, \sum_{j=1}^N \,  \psi^{**}_j(\beta(s-f_j), z)
+ \epsilon \sum_{j=1}^N \,  \psi^{**}_j( -\beta(s+f_j), z),
\end{split}	
\end{equation}
as a correction term,
where $\psi^{**}_j(x_j, z)$   satisfies
\begin{equation}\label{psi**1}
\psi^{**}_{j,zz} + \beta^{2}  \, \Big[\, \psi^{**}_{j,x_jx_j }\, +\, {F'}(H_j)\psi^{**}_j\, \Big] \,=\,-\, \Big[ \beta_1 \bar{f}_j \Xi_{0,j}
+\,\Xi_{1,j}
\,+\,\chi_{j}\Xi_{4,j}\,\Big]\quad     \text{in}\,   \mathbb{R}\times \Big[0,\frac{\ell}{\epsilon} \Big),
\end{equation}
\begin{equation}\label{psi**2}
\psi^{**}_j(x_j, 0)=\psi^{**}_j\Big(x_j, \frac{\ell}{\epsilon}\Big),
\quad \psi^{**}_{j,z}(x_j, 0)=\psi^{**}_{j,z}\Big(x_j, \frac{\ell}{\epsilon}\Big),
\end{equation}
\begin{equation}\label{psi**3}
\int_\mathbb{R}{\psi^{**}_j H_{j, x_j}}\,{\mathrm{d}}x_j=0,\quad   0<z<\frac{\ell}{\epsilon}.
\end{equation}
It can be checked that
\begin{align*}
\epsilon   \, \frac{\partial  \phi_{13}(s,z)}{\partial  s } \Big|_{s=0} \, =\,   0.
\end{align*}

\medskip
Finally, we define the first correction as
\begin{equation}\label{phi1}
\begin{split}
\phi_1(s,z) &= \epsilon |\ln \epsilon| \, \phi_{11}(s,z) + \epsilon \, \phi_{12}(s,z) \, +\, \epsilon  \, \phi_{13}(s,z)
\\[1mm]
 \, &=\, \epsilon \, \beta^{-2} \beta_1\, \sum_{j=1}^N  \Big[  \dot { f}_j  \,\psi_j( x_j) +   \dot {f}_j      \psi_j  (\tau_j  ) \Big]
+ \epsilon \sum_{j=1}^N \Big[  \psi^{*}_j(x_j, z) +   \psi^{*}_j (\tau_j, z) \Big]
 \\
 &\quad
 + \epsilon \, \sum_{j=1}^N \,  \Big[\psi^{**}_j(x_j, z)+
 \psi^{**}_j( \tau_j, z)\Big],
\end{split}
\end{equation}
where we have denoted
 \begin{align}\label{tauj}
\tau_j(s,z)  = - x_j- 2\beta f_j=  - \beta (s +f_j).
\end{align}
Then we can get easily that  the boundary error for $\phi_1$ satisfies
\begin{align}\label{boundary2}
{\mathcal{D}} (\phi_1)\, =\,  0.
\end{align}
We take
\begin{align}\label{secondapproximatesolution}
u_2(s,  z)
=u_1(s, z)
+\phi_1(s,z)
\end{align}
as the {\textbf{second approximate solution}}.

\subsubsection{The error to the second approximation}
The new error can be computed as the following
 \medskip
\begin{equation}\label{S(v+phi1)}
\begin{split}
{\bf S}(u_2) \,=&\,{\bf S}(u_1)  +  L_{u_1}(\phi_1)
+ N_{u_1}(\phi_1)  +B_3(\phi_1)
\\[2mm]
&\,  +B_ 2(u_1+\phi_1)-B_2(u_1)
+B_4(u_1+\phi_1)-B_4(u_1),
\end{split}
\end{equation}
with
\begin{equation}\label{Lu1}
L_{u_1} (\phi_1) \, =\,\phi_{1,zz} + \phi_{1,ss } \,+\, \beta^2 \, (1-3u_1^2)\phi_1,\quad N_{u_1}(\phi_1)\, =\, \beta^2  ( -3\, u_1\,\phi_1^2-\phi_1^3).
\end{equation}
We will give the details of computations for the new error components in the sequel.

\medskip
First, recalling  the definition of $ \tau_j $ as in \eqref{tauj}, it is easy to obtain that
\begin{equation}\label{tauz}
\tau_{j,z}  (s,z) =  \big( - \beta s-\beta f_j \big)_z\, =\, -\epsilon \Big[ \frac{\beta' }{\beta}x_j+ \beta f'_j+ 2\beta' f_j\Big],
\end{equation}
\begin{equation}\label{tauzz}
\tau_{j, zz}(s,z)  =  \big( - \beta s-\beta f_j\big)_{zz}
\, =\,-\epsilon^2 \Big[  \frac{\beta'' }{\beta}x_j+ 2 \beta'' f'_j+ \beta f''_j+2\beta'f'_j \Big].
\end{equation}
Here we declare that  we will use the symbols $   \psi_j, \psi_j^*, \psi_j^{**}$ to denote $     \psi_j(x_j), \psi_j^*(x_j,z), \psi_j^{**}(x_j,z) $  in the sequel  for simplicity.

\medskip
According to the definition of $\phi_1$ as in \eqref{phi1}, we have
\begin{align*}
L_{u_1} (\phi_1)
\,=&\,  \epsilon \sum_{j=1}^N  \beta_1  \dot{ f}_j \Big[    \psi_{j, x_jx_j} + (1-3 H_j^2)\psi_j \Big]
+ \epsilon  \sum_{j=1}^N  \Big\{ \psi^*_{j, zz}  +\beta^2 \big[  \psi^*_{j, x_jx_j} + (1-3 H_j^2) \psi_j^*  \big] \Big\} \nonumber
\\[2mm]
\quad &+ 2\epsilon^2\sum_{j=1}^N \psi^{*}_{j,zx_j} \Big[\frac{\beta'}{\beta}  x_j- \beta f_j' \Big]+ \epsilon  \sum_{j=1}^N  \Big\{\psi^{**}_{j,zz}+\beta^2 \big[ \psi^{**}_{j,x_jx_j} +  (1-3 H_j^2)\psi_j^{**}  \big]   \Big\}
\nonumber
\\
\quad &+2\epsilon^2\sum_{j=1}^N \psi^{**}_{j,zx_j}   \Big[\frac{\beta'}{\beta} x_j- \beta f_j' \Big]
 + \epsilon \beta_1 \sum_{j=1}^N  ( 3 H_j^2 - 3 u_1^2)\dot { f}_j \psi_j
 \nonumber
\\[2mm]
\quad  &
+ \sum_{j=1}^N \Bigg\{\epsilon \beta_1 \dot{ f}_j  \Big[\psi_{j,\tau_j  \tau_j }(\tau_j ) + (1-3 u_1^2)\psi_j(\tau_j)  \Big]
\\[2mm]
  &\quad +\epsilon  \beta^2   \Big[    \psi^*_{j,\tau_j \tau_j}(\tau_j, z ) + (1-3u_1^2)\psi^*_j(\tau_j,z)  \Big]+ \epsilon \psi^{*}_{j,zz}(\tau_j, z)+2\epsilon \psi^{*}_{j,\tau_j z}(\tau_j, z)\tau_{j,z}  \nonumber
  \\[2mm]
&\quad +  \epsilon \beta^2\Big[\psi^{**}_{j, \tau_j  \tau_j }(\tau_j, z ) + (1-3 u_1^2)\psi^{**}_j(\tau_j,z)  \Big] + \epsilon \psi^{**}_{j,zz}(\tau_j, z)
+2\epsilon \psi^{**}_{j,\tau_j z}(\tau_j, z)\tau_{j,z } \Bigg\}
\nonumber
\\
& + \epsilon \beta^2 \sum_{j=1}^N   ( 3 H_j^2 - 3 u_1^2)(  \psi_j^*+ \psi^{**}_j )
+O(\epsilon^{2+\sigma})\nonumber
\\
=&-\epsilon  \beta_1  \sum_{j=1}^N \, \dot { f}_j \, \Xi_{0,j} -\epsilon \, \sum_{j=1}^N \, { \chi_j}    \Xi_{3,j}
-\epsilon \, \sum_{j=1}^N  \Big(   \beta_1 \bar{f}_j \Xi_{0,j} +\,  \Xi_{1,j} \,+ \, \chi_j\Xi_{4,j}\Big)        \nonumber
\\[2mm]
\quad  &    + \epsilon \beta_1  \, \sum_{j=1}^N     ( 3 H_j^2 - 3 u_1^2)  \dot { f}_j \, \psi_j+ \epsilon \beta^2 \sum_{j=1}^N    ( 3 H_j^2 - 3 u_1^2)( \,  \psi_j^*+ \psi^{**}_j )
 \nonumber
 \\[2mm]
 \quad  & +\epsilon^2  |  \ln \epsilon|  \,  \sum_{j=1}^N  \Xi_{7,j}
 +\epsilon^2   \,  \sum_{j=1}^N  \Xi_{8,j} +    \mathcal{K}_1+O(\epsilon^{2+\sigma}),\nonumber
 \end{align*}
  where
\begin{equation}\label{Xi7j}
\epsilon^2  |  \ln \epsilon| \,  \Xi_{7,j}
:= - 2\epsilon^2  \beta \sum_{j=1}^N \big( \dot f_j'  +\bar f '_j\big)\Big[\psi^{*}_{j,zx_j} (x_j, z)+\psi^{**}_{j,zx_j} (x_j, z)\Big] ,
\end{equation}
\begin{align} \label{Xi8jnew}
\epsilon^2 \,   \Xi_{8,j} :=& 2\epsilon^2\,\frac{\beta'}{\beta} x_j\,\Big[  \psi^{*}_{j,zx_j} (x_j, z)   + \psi^{**}_{j,zx_j} (x_j, z) \Big]    \nonumber
\\[1mm]
\quad &
- 2 \epsilon^2\, \Big( \frac{\beta' }{\beta}x_j+ \beta f'_j+ 2\beta' f_j\Big)  \Big[\psi^{*}_{j,\tau_j z}(\tau_j, z)  + \psi^{**}_{j,\tau_j z}(\tau_j, z)  \Big],	
\end{align}
and
\begin{equation}\label{mathcalK1}
\begin{split}
\mathcal{K}_1\,: =&  \sum_{j=1}^N   \Big\{ \epsilon  \, \beta_1 \, \dot{ f}_j  \, \Big[    \psi_{j,\, \tau_j \tau_j}(\tau_j ) + (1-3 u_1^2)\psi_j(\tau_j)  \Big]
\\[1mm]
& \qquad   +  \epsilon  \beta^2 \Big[    \psi^*_{j,\, \tau_j \tau_j}(\tau_j, z ) + (1-3u_1^2)\psi^*_j(\tau_j,z) \Big]  + \epsilon \psi^{*}_{j,zz}(\tau_j, z)
\\[1mm]
& \qquad +  \epsilon \beta^2 \Big[    \psi^{**}_{j, \tau_j \tau_j}(\tau_j, z ) + (1-3 u_1^2)\psi^{**}_j(\tau_j,z)  \Big]
+ \epsilon \psi^{**}_{j,zz}(\tau_j, z) \Big\}.
\end{split}
\end{equation}

\medskip
According to the definition of $B_3$ as in \eqref{B3}, we have
 \begin{align}\label{decomB3}
B_3(\phi_1)= &\, -\epsilon^2\,  \beta^{-1} \,  k  \sum_{j=1}^N   \beta_1 \, {\dot f}_j   \, \psi^{}_{j, x_j}(x_j)
-\epsilon^2\beta  k  \sum_{j=1}^N \Big[\psi^{*}_{j, x_j}(x_j, z)+  \psi^{**}_{j, x_j}(x_j, z)\Big]\nonumber
\\
&+\epsilon^2  \sum_{j=1}^N   \Big[\beta^{-1} k   \beta_1 {\dot f}_j \psi^{}_{j, \tau_j}(\tau_j )
+ \beta \,  k \big( \psi^{*}_{j, \tau_j}(\tau_j,z) +  \psi^{**}_{j, \tau_j} (\tau_j,z)\big) \Big] +O(\epsilon^{2+\sigma})
\\
\, =&\,\epsilon^2 |\ln \epsilon|   \sum_{j=1}^{N} \Xi_{9,j}  +  \mathcal{K}_2+ O(\epsilon^{2+\sigma}),\nonumber
 \end{align}
 with
 \begin{equation}\label{Xi8j}
\epsilon^2 |\ln \epsilon|\,  \Xi_{9,j}
\,: =\,    -\epsilon^2\beta^{-1}  k  \beta_1 \dot{f}_j  \psi_{j, x_j}(x_j ) -  \epsilon^2\beta\,  k    \Big[\psi^{*}_{j, x_j}(x_j, z) + \psi^{**}_{j, x_j}(x_j, z)  \,\Big],
 \end{equation}
  \begin{equation}
  \mathcal{K}_2:= \epsilon^2  \sum_{j=1}^N   \Big[ \,  \beta^{-1} \,  k   \beta_1 \, {\dot f}_j  \, \psi^{}_{j, \tau_j}(\tau_j )
+ \beta \,  k \big( \psi^{*}_{j, \tau_j}(\tau_j,z) +  \psi^{**}_{j, \tau_j} (\tau_j,z)\big) \Big].
 \end{equation}
Moreover,  we  have
 \begin{equation*}
    B_2(u_1+\phi_1)-B_2(u_1)=O(\epsilon^{2+\sigma}).
 \end{equation*}

 \medskip
Recalling the definition of the function $B_4$ in  \eqref{B4u} and together with \eqref{mathfrakfj} and \eqref{epsilon12},  by calculating, we can obtain
\begin{align}\label{decomB4}
&B_4(u_1+\phi_1)-B_4(u_1) \nonumber
\\[1mm]
&  =  \Big[ \epsilon s\, \beta_1+\frac{1}{2}\epsilon^2s^2\beta_2\Big]  \sum_{j=1}^N\, (1-3H_j^2)   \phi_1   \nonumber
\\[1mm]
&  \quad+   \Big[ \epsilon s\, \beta_1+\frac{1}{2}\epsilon^2s^2\beta_2\Big] \sum_{j=1}^N\Big[   (3H_j^2-3u_1^2)   \phi_1  - 3 u_1\phi_1^2 -\phi_1 ^3  \Big]
\\[1mm]
& = \,  \epsilon^2  \beta_1\,  \sum_{j=1}^N \Big(\frac{x_j}{\beta} + f_j \Big)   (1 -3H_j^2)   \Big[\beta_1 {\dot f}_j \, \beta^{-2}\, \psi_j(x_j) +\psi_j^*(x_j,z) + \psi^{**}_j(x_j,z) \Big] + O(\epsilon^{2+\sigma})\nonumber
\\[1mm]
& =\epsilon^2  |\ln \epsilon |  \sum_{j=1}^N  \Xi_{10,j}  + O(\epsilon^{2+\sigma}),\nonumber
\end{align}
with
\begin{equation}\label{Xi9j}
\epsilon^2  |\ln \epsilon |  \Xi_{10,j}
\, :=\,  \epsilon^2  \beta_1\Big(\frac{x_j}{\beta} + {\dot f}_j +{\bar f}_j \Big) (1 -3H_j^2) \Big[\beta_1 {\dot f}_j  \beta^{-2} \psi_j(x_j) +\psi_j^*(x_j,z) + \psi^{**}_j(x_j,z) \Big].
\end{equation}
For the nonlinear term $N_{u_1}(\phi_1)$, we have that
\begin{equation*}
\begin{split}
N_{u_1}(\phi_1)\, & = \, \beta^2 ( -3\, u_1\,\phi_1^2-\phi_1^3) = \,  -3\, \beta^2 \, u_1  \, \phi_1^2 + O(\epsilon^{2+\sigma}).
\end{split}
\end{equation*}

Combining \eqref{S(u13)},  \eqref{S(v+phi1)}, %\eqref{Lui+phizz},
\eqref{decomB3}, %\eqref{Lui+phizz},
\eqref{decomB4},
the new error can be expressed  as the following
\begin{equation}
\begin{split}
\label{S(u211)}
 {\bf S} (u_2) \, = \,& \epsilon \beta_1\sum_{j=1}^{N}  \hat{f}_j \Xi_{0,j}
 +\epsilon \sum_{j=1}^N \chi_j \Xi_{5,j}+3 \beta^2 \epsilon \sum_{j=1}^{N}  (  H_j^2 -  u_1^2) (\beta_1 \beta^{-2}  \dot{ f}_j \psi_j+\psi_j^*+ \psi^{**}_j )
\\[1mm]
& + \epsilon^2 |\ln\epsilon| \sum_{j=1}^{N}\Xi_{2,j}
+\epsilon^2 |\ln \epsilon|   \sum_{j=1}^{N} \chi_j\Xi_{6,j}
+\epsilon^2  |  \ln \epsilon|  \sum_{j=1}^N  \Xi_{7,j}+\epsilon^2 \sum_{j=1}^N  \Xi_{8,j}
\\[1mm]
 \quad  &  +\epsilon^2  |\ln \epsilon | \sum_{j=1}^N  \Xi_{9,j}
+\epsilon^2 |\ln \epsilon| \sum_{j=1}^N \Xi_{10,j}
+\epsilon^2 |\ln\epsilon|  E_{04}
+\epsilon^2 |\ln\epsilon|^2  E_{03}
\\[1mm]
 \quad  &   +\epsilon^2 E_{06} +E_{07}
 +  \mathcal{K}_1+\mathcal{K}_2
 +  N_{u_1}(\phi_1) + O(\epsilon^{2+\sigma}).
\end{split}
\end{equation}

\subsubsection{More precise estimates of the error ${\bf S}(u_2)$} In order to get the more precise estimates of the error, we set
\begin{equation} \label{hatfj}
\hat{f_j}(\theta) \,=\, \check{f}_j(\theta) \,+\,{\tilde f}_j(\theta) , \quad    \forall\, j = 1,\cdots,N,
\end{equation}
and  in particular, let
\begin{align}
\check{ f}_0= -\check{f}_1, \quad {\tilde f}_0= - {\tilde f}_1.
\end{align}
\medskip
The parameters $\check{f}_1, \cdots, \check{f}_n$ will be found by solving an algebraic system in \eqref{checkf} so that
\begin{align}\label{checkfnorm}
 \|\check { f} _j \|_{L^\infty(0,\ell)} \leq C\epsilon^{\frac 12},\quad j=1,\cdots, N.
\end{align}
In all what follows, we will make the following assumption on the parameter ${\tilde{\bf f}} = ({\tilde f}_1, \cdots, {\tilde f}_N)^{T} $
\begin{align}\label{ddotfnorm}
\| {{\tilde{\bf f}} }\| := \epsilon \| {\tilde{\bf f}} ''\|_{L^2 (0, \ell ) } + \epsilon^{\frac 12 } \| {\tilde{\bf f}} '\|_{L^2 (0, \ell ) }
+  \| {\tilde{\bf f}}\|_{L^\infty (0, \ell ) }  = O(\epsilon^{\frac 12+ \sigma} ),
\end{align}
which will be determined in Sections \ref{section6} and \ref{section7}, by using the reduction method.

\medskip
 Similar to  \eqref{mathbfdn},  we set the functions
 \begin{equation}\label{mathbf kj}
{\mathbf k}_j (\theta)\,\,:=\, \left\{
 \begin{array}{cl}
  \, {\mathbf d}_j (\theta)e^{-\sqrt 2 \, \beta(\theta) \big(\check{f}_{j}(\theta)-\check{f}_{j-1}(\theta)\big) },  \quad &\text{when $  j=1, \cdots, N,  $ }
 \\[1mm]
0,  \, \quad &\text{when $ j \, = \,   N+1$. }
\end{array}
\right.
\end{equation}
Therefore, we can decompose  the term $\, \Xi_{5,j}$ (c.f. \eqref{Xi7n}) as follows
 \begin{align} \label{decomposeXi5j}
\Xi_{5,j} \,= \, \Xi_{5,j}^{1}+ \, \Xi_{5,j} ^{2}
 \end{align}
  where
 \begin{equation} 	\label{Xi5j^1}
\begin{split}
\Xi_{5,j}^1 :=
%  - 6 \sqrt 2\, \frac{1}{\gamma_{0,j} } \, {\chi_j} \, \beta^2  \,H_{j,x_j}  \,
%  \Big\{ & - \gamma_{1,j}\, {\mathbf d}_{j}\,  \big(\,e^{\,-  \sqrt 2 \,\beta  \,(\check{f}_{j}-\check{f}_{j-1})\,}-1 \, \big)
%  \\[1mm]
%  \,&+ \, \gamma_{2,j}\, {\mathbf d}_{j+1} \, \big(  e^{\,-  \sqrt 2 \,\beta  \,(\check{f}_{j+1}-\check{f}_j)\,}  -1 \big) \Big\}
%  \\[1mm]
% =
 - 6 \sqrt 2\, \frac{1}{\gamma_{0,j} }  \beta^2  \,H_{j,x_j}  \,
 \Big[& - \gamma_{1,j}\,  \big(\,{\mathbf k}_{j}- {\mathbf d}_{j} \, \big)
 \, + \, \gamma_{2,j}\,  \big(  {\mathbf k}_{j+1} \,  -{\mathbf d}_{j+1} \, \big) \Big],
 \end{split}
\end{equation}
and
\begin{equation}
\begin{split}
\label{Xi5j^2}
 \Xi_{5,j} ^2 :=
 - 6 \sqrt 2 \,   \frac{1}{\gamma_{0,j} }\, \beta^2 \, H_{j,x_j}  \,
 \Big[ & - \gamma_{1,j}\,   {\mathbf k}_{j} \,  \big(\,e^{\,-  \sqrt 2 \,\beta  \,({\tilde f}_{j}-{\tilde f}_{j-1})\,}-1 \, \big)
 \\[1mm]
 \,&+ \, \gamma_{2,j}\,     {\mathbf k}_{j+1} \,  \big(  e^{\,-  \sqrt 2 \,\beta  \,({\tilde f}_{j+1}-{\tilde f}_j)\,}  -1 \big)   \Big].
\end{split}
\end{equation}

Recalling the definition of $E_{07}$ as in \eqref{E032} and the expression of ${\mathbf b}_{1j}, {\mathbf b}_{3j}$ as in \eqref{b1} and \eqref{b3}, we  can rewrite $E_{07}$ as
\begin{align}
E_{07} = E_{07}^1+ E_{07}^2.
\end{align}
Here
\begin{align}
E_{07}^1 &=  \beta^2 \sum_{j=1}^N (-1)^{j} {\chi_j}  \Big[\,\frac 12 {\mathbf A}_{1j} \, {\mathbf b}_{1j}^2
 +\frac 12 {\mathbf A}_{3j}\, {\mathbf b}_{3j}^2 +  \frac 12  {\mathbf A}_{5j}  {\mathbf b}_{1j}^2+ \frac 12  {\mathbf A}_{5j}  {\mathbf b}_{3j}^2 \,\Big]\nonumber
   \\[1mm]
 &=
 \, 2\epsilon^2\beta^2 \sum_{j=1}^N  (-1)^{j } { \chi_j}
\, \Big[\big( {\mathbf A}_{1j} +   {\mathbf A}_{5j } \big)      {\mathbf k}^2_j  \, e^{- 2\sqrt 2  x_j   }
  e^{-2\sqrt{2} \beta({\tilde f}_{j} -{\tilde f}_{j-1}  )}
  \\[1mm]
  &\qquad \qquad \qquad \qquad \qquad +   \big( {\mathbf A}_{3j}+{\mathbf A}_{5j} \big)    {\mathbf k}^2_{j+1}  \, e^{ 2 \sqrt 2  x_j   }    e^{-2\sqrt{2} \beta({\tilde f}_{j+1} -{\tilde f}_{j}  )}  \Big]+ O(\epsilon^{2+\sigma} ) \nonumber
 \\[1mm]  &=
 \epsilon^{\frac{3}{2}}\sum_{j=1}^N { \chi_j}\Xi_{11,j}  + O(\epsilon^{2+\sigma} ), \nonumber
\end{align}
with
\begin{equation}
\epsilon^{\frac{3}{2}}\Xi_{11,j} :=
 2\epsilon^2\beta^2   (-1)^{j }
\, \Big[   \big( {\mathbf A}_{1j} +   {\mathbf A}_{5j } \big)      {\mathbf k}^2_j  \, e^{- 2\sqrt 2  x_j   }
  +   \big( {\mathbf A}_{3j}+{\mathbf A}_{5j} \big)    {\mathbf k}^2_{j+1}  \, e^{ 2 \sqrt 2  x_j   }      \Big].
\end{equation}
On the other hand, we can rewrite $E_{07}^2$ as following
\begin{align}
E_{07}^2
\,&=\,\beta^2 \sum_{j=1}^N (-1)^{j} {\chi_j}  3(1-H_j^2)({\mathbf b}_{2j}+ {\mathbf b}_{4j})
+\beta^2 \sum_{j=1}^N (-1)^{j} {\chi_j}  {\mathbf A}_{5j} {\mathbf b}_{1j} {\mathbf b}_{3j}\nonumber
\\[1mm]
&\quad-\epsilon \beta_1\beta^2  \sum_{j=1}^N   (-1)^{j} {\chi_j}  \Big( \frac{x_j }{\beta} + { f}_j\Big)3(1-H_j^2)\,({\mathbf b}_{1j} +{\mathbf b}_{3j}) \nonumber
\\
& = 6 \sqrt 2 \epsilon^2\beta^2 \sum_{j=1}^N {\chi_j}  H_{j,x_j } \Big[    -{\mathbf k}_{j}   {\mathbf k}_{j-1}e^{- \sqrt 2  x_j   }      e^{-\sqrt{2} \beta({\tilde f}_{j} -{\tilde f}_{j-2}  )}
+  {\mathbf k}_{j+1}   {\mathbf k}_{j+2}e^{ \sqrt 2  x_j   } e^{-\sqrt{2} \beta({\tilde f}_{j+2} -{\tilde f}_{j}  )} \Big] \nonumber
\\[1mm]
&\quad -6 \sqrt 2\epsilon^2 \beta_1\beta^2  \sum_{j=1}^N  {\chi_j}  \Big( \frac{x_j }{\beta} + \dot{ f}_j+ \bar{f}_j + \check{f}_j \Big)    H_{j,x_j }     \nonumber
\\[1mm]
&\qquad \qquad \qquad \qquad \qquad \qquad \times \Big[- {\mathbf k}_{j} e^{- \sqrt 2  x_j   } e^{-\sqrt{2} \beta({\tilde f}_{j} -{\tilde f}_{j-1}  )}
+ {\mathbf k}_{j+1} e^{ \sqrt 2  x_j   } e^{-\sqrt{2} \beta({\tilde f}_{j+1} -{\tilde f}_{j}  )}   \Big] \nonumber
\\[1mm]
&\quad -4 \epsilon^2\beta^2 \sum_{j=1}^N (-1)^{j} {\chi_j}  {\mathbf A}_{5j}  {\mathbf k}_{j}   {\mathbf k}_{j+1} e^{-\sqrt{2} \beta({\tilde f}_{j+1} -{\tilde f}_{j-1}  )} + O(\epsilon^{2+\sigma} )\nonumber
\\[1mm]
& = \epsilon^2 |\ln \epsilon |\sum_{j=1}^N {\chi_j} \Xi_{12,j}+ O(\epsilon^{2+\sigma} ),\nonumber
\end{align}
with
\begin{equation}\label{Xi12j}
\begin{split}
 \epsilon^2|\ln \epsilon |\Xi_{12,j}
 & := 6 \sqrt 2 \epsilon^2\beta^2  H_{j,x_j } \Big[- {\mathbf k}_{j}{\mathbf k}_{j-1}e^{- \sqrt 2  x_j   }
 + {\mathbf k}_{j+1}{\mathbf k}_{j+2} e^{ \sqrt 2  x_j}\Big]
\\[2mm]
&\quad-6 \sqrt 2\epsilon^2 \beta_1\beta^2 \Big( \frac{x_j }{\beta} + \dot{ f}_j+ \bar{f}_j\Big) H_{j,x_j } \Big[-{\mathbf k}_{j}e^{- \sqrt 2  x_j   }
+{\mathbf k}_{j+1} e^{ \sqrt 2  x_j}  \Big]
\\[2mm]
&\quad -4 \epsilon^2\beta^2 (-1)^{j}  {\mathbf A}_{5j} {\mathbf k}_{j}   {\mathbf k}_{j+1}.
\end{split}
\end{equation}
Thanks to the constraint of ${\tilde{\bf f}} $ as in \eqref{ddotfnorm} and a Taylor expression, the terms $e^{-\sqrt{2} \beta({\tilde f}_{j+1} -{\tilde f}_{j}  )}\,-\,1$ are of order $\epsilon^{\frac{1}{2}+\sigma}$.
\medskip
For the term $ \mathcal{K}_1$ (c.f. \eqref{mathcalK1}), we can decompose it as follows
\begin{equation*}
 \mathcal{K}_1\, = \,\epsilon^{\frac{3}{2}}\mathcal{K}_3\,+\,   \mathcal{K}_4,
\end{equation*}
with
\begin{equation} \label{mathcal{K}3}
\begin{split}
\epsilon^{\frac{3}{2}}\mathcal{K}_3\, &=\,    \epsilon  \, \beta_1 \, \dot{ f}_1  \Big[  \psi_{1,\, \bar{\tau}_1 \bar{ \tau}_1 }(\bar{\tau}_1 ) + \big (1-3 H_1^2 \big)\psi_1( \bar{\tau}_1)  \Big]
 \\[1mm]
 & \quad +   \epsilon \,\beta^2 \Big[  \psi^*_{1,\, \bar{\tau}_1 \bar{ \tau}_1 }(\bar{\tau}_1, z ) + \big(1-3 H_1^2  \big)\psi^*_1( \bar{\tau}_1,z)  \Big]+\epsilon   \psi^{*}_{1,zz}( \bar{\tau}_1, z)
 \\[1mm]
 & \quad+   \epsilon\, \beta^2 \Big[  \psi^{**}_{1,\, \bar{\tau}_1 \bar{ \tau}_1 }(\bar{\tau}_1, z ) +\big(   1-3 H_1^2   \big)\psi^{**}_1( \bar{\tau}_1,z)  \Big] +\epsilon   \psi^{**}_{1,zz}( \bar{\tau}_1, z),
\end{split}
\end{equation}
\begin{align}\label{mathcalK4}
 \mathcal{K}_4\, =&  \sum_{j=2}^N   \Bigg \{    \epsilon  \, \beta_1 \, \dot{ f}_j  \, \Big[    \psi_{j,\, \tau_j \tau_j}(\tau_j ) + (1-3 u_1^2)\psi_j(\tau_j)  \Big]
 \nonumber
\\[1mm]
 & \qquad\,\, +  \epsilon \, \beta^2 \Big[    \psi^*_{j,\, \tau_j \tau_j}(\tau_j, z ) + (1-3u_1^2)\psi^*_j(\tau_j,z)  \Big]+  \epsilon \psi^{*}_{j,zz}(\tau_j, z) \nonumber
\\[1mm]
 & \qquad\,\,+  \epsilon \,\beta^2 \Big[    \psi^{**}_{j,\, \tau_j \tau_j}(\tau_j, z ) + (1-3 u_1^2)\psi^{**}_j(\tau_j,z) \Big]   + \epsilon \psi^{**}_{j,zz}(\tau_j, z)
\Bigg\}
\\[1mm]
  &  + \epsilon  \, \beta_1 \, \dot{ f}_1  \Big[ \psi_{1,\,  {\tau}_1  { \tau}_1 }( {\tau}_1 )-  \psi_{1,\, \bar{\tau}_1 \bar{ \tau}_1 }(\bar{\tau}_1 ) + \big (1-3 H_1^2 \big) \big( \psi_1( {\tau}_1)  - \psi_1( \bar{\tau}_1)   \big)\Big]  \nonumber
 \\[1mm]
 & +
  \epsilon \,  \Big[ \psi^{*}_{1,zz}(  {\tau}_1, z) -\psi^{*}_{1,zz}( \bar{\tau}_1, z) \Big] +
  \epsilon \, \Big[\psi^{**}_{1,zz}(  {\tau}_1, z)-   \psi^{**}_{1,zz}( \bar{\tau}_1, z) \Big]
  \nonumber
 \\[1mm]
 &  +   \epsilon \,\beta^2 \Big[  \psi^*_{1,\,  {\tau}_1  { \tau}_1 }({\tau}_1, z )- \psi^*_{1,\, \bar{\tau}_1 \bar{ \tau}_1 }(\bar{\tau}_1, z ) + \big(1-3 H_1^2  \big) \big (\psi^*_1(  {\tau}_1,z)-   \psi^*_1( \bar{\tau}_1,z) \big)   \Big]\nonumber
 \\[1mm]
 &  +   \epsilon\, \beta^2 \Big[  \psi^{**}_{1,\,  {\tau}_1  { \tau}_1 }( {\tau}_1, z ) -  \psi^{**}_{1,\, \bar{\tau}_1 \bar{ \tau}_1 }(\bar{\tau}_1, z ) +\big(   1-3 H_1^2   \big) \big( \psi^{**}_1(  {\tau}_1,z)   - \psi^{**}_1( \bar{\tau}_1,z)  \big)  \Big], \nonumber
 \end{align}
and
\begin{align} \label{bartau1}
\bar{\tau}_1\, =\,   -  x_1-2\beta (\dot{f}_1 + \bar{f}_1+ \check{f}_1).
\end{align}

Then  the new error can be rewritten as the following
\begin{align}\label{S(u21)}
 {\bf S} (u_2) \, = \,& \epsilon \beta_1\sum_{j=1}^{N}   \big( \check{f}_j+{\tilde f}_j \big)   \Xi_{0,j}
 \,+\,\epsilon \sum_{j=1}^N  {\chi_j} \Xi_{5,j}^1
  \,+\,\epsilon \sum_{j=1}^N  {\chi_j} \Xi_{5,j}^2 \nonumber
\\
& \,+\, \epsilon\, 3 \beta^2  \sum_{j=1}^{N}  (  H_j^2 -  u_1^2) \Big[\beta_1  \beta^{-2} \dot{ f}_j \psi_j+\psi_j^*+ \psi^{**}_j \Big]
+ \epsilon^2 |\ln\epsilon| \sum_{j=1}^{N}\Xi_{2,j} \nonumber
 \\
&
\,+\,\epsilon^2 |\ln \epsilon|   \sum_{j=1}^{N}  {\chi_j}\Xi_{6,j}
\,+\,\epsilon^2  |  \ln \epsilon|  \sum_{j=1}^N  \Xi_{7,j}
\,+\,\epsilon^2 \sum_{j=1}^N  \Xi_{8,j}
\,+\,\epsilon^2  |\ln \epsilon | \sum_{j=1}^N  \Xi_{9,j}
\\[1mm]
 \quad  & \,+\,\epsilon^2 |\ln \epsilon| \sum_{j=1}^N \Xi_{10,j}
 \,+\,\epsilon^{\frac{3}{2}} \sum_{j=1}^N  {\chi_j} \Xi_{11,j}
 \,+\, \epsilon^2|\ln \epsilon| \sum_{j=1}^N  {\chi_j} \Xi_{12,j}
\,+\,\epsilon^2 |\ln\epsilon|  E_{04}
 \nonumber
\\[1mm]
 \quad  &   \,+\,\epsilon^2 |\ln\epsilon|^2  E_{03}
 \,+\,\epsilon^2 E_{06}
\,+\, \mathcal{K}_2\,+\, \epsilon^{\frac{3}{2}}\mathcal{K}_3+ \mathcal{K}_4
 \,+\, N_{u_1}(\phi_1) \,+\, O(\epsilon^{2+\sigma}).\nonumber
\end{align}

\medskip
\subsection{The third approximation}
It turns out that, in the expression of $ {\bf S} (u_2)$, the terms of order $\epsilon^{\frac{3}{2}}$ are
$$
\epsilon \beta_1 \sum_{j=1}^N \check{f}_j \Xi_{0,j}, \quad \epsilon^{\frac{3}{2}} \mathcal{K}_3,\quad
\epsilon \sum_{j=1}^N  {\chi_j} \Xi^1_{5,j}
\quad
\mbox{and}
\quad
 \epsilon^{\frac{3}{2}} \sum_{j=1}^N  {\chi_j} \Xi_{11,j}.
 $$
 Since these terms are too large for our purpose, we add a correction to the approximate solution in order to cancel these terms and then compute the corresponding error.

\subsubsection{New correction terms}
In order to cancel the terms
$$
 \epsilon \beta_1\sum_{j=1}^N  \check{f}_j \Xi_{0,j} \,+\,  \epsilon^{\frac{3}{2}} \mathcal{K}_3\,+\,\epsilon \sum_{j=1}^N  {\chi_j} \Xi^1_{5,j}  \,+\, \epsilon^{\frac{3}{2}} \sum_{j=1}^N    {\chi_j}\Xi_{11,j},
$$
we shall choose $\check{f}_j $ such that
\begin{align} \label{eqaution 2}
\int_\mathbb{R}\Big[    \epsilon \beta_1   \check{f}_j \Xi_{0,j} +\epsilon^{\frac{3}{2}} \mathcal{K}_3 +\epsilon  {\chi_j} \Xi^1_{5,j}  +   \epsilon^{\frac{3}{2}}  {\chi_j} \Xi_{11,j} \Big]
H_{j,x_j}\,{\mathrm{d}}x_j=0,
\quad\forall\, j=1,\cdots,N.
\end{align}
Here we note that,       by using  comparison  principle  argument  and combining the effect of  the cut-off functions $ \chi_j $,    there exists  an universal constant $ C$ such that
\begin{align}
\label{estimatepsi1}
  | \psi _j ( x_j   ) |       \le C e^{  - \sqrt 2  \beta (s-f_ j )   },    \quad  | \psi ^{**}   _j ( x_j, z   ) |       \le C e^{  - \sqrt 2  \beta (s-f_ j )   },     \end{align}
when $ s$ is far away from $ f_j$,
and
\begin{align} \label{estimatepsi2}
 | \psi^*  _j ( x_j,z   ) |       \le  \begin{cases} C,        & \text{ when $ s \in   \big\{ s\in (0, + \infty)~  | ~ (s,z) \in  {\mathfrak B}_j\big\},     $  }     \\[3mm]
 C   e^{  - \sqrt 2  \beta (s-f_ j )   },        & \text{ when $ s \notin   \big\{ s \in (0, + \infty) ~  | ~ (s,z) \in  {\mathfrak B}_j\big\},    $  }
 \end{cases}
 \end{align}
 where $   {\mathfrak B}_j$ is  support of $ \chi_j$.
 Recall the expression of $ \epsilon^{\frac{3}{2}} \mathcal{K}_3 $ in \eqref{mathcal{K}3} and the  definition of $\tau_1 $ in \eqref{bartau1}, we can obtain
\begin{align*}
\epsilon^{\frac{3}{2}}\int_\mathbb{R} \mathcal{K}_3 \, H_{j,x_j}\,{\mathrm{d}}x_j  \, =\, \epsilon^{\frac 32}  \mathfrak{T}_0 (\dot{f}_1,  \bar{f}_1, \check{f}_1)\delta_{1j}+O(\epsilon^{2+\sigma}),
\end{align*}
% \begin{equation}
%\epsilon^{\frac{3}{2}}\int_\mathbb{R}   \mathcal{K}_3 \, H_{j,x_j}\,{\mathrm{d}}s \,:=\, \left\{
%\begin{array}{cr}
%  \, (N+1-j)\, e^{-\sqrt{2}\, \beta\, ({\bar f}_{j}-{\bar f}_{j-1})},  \quad &\text{when $  j=1, \cdots, N,  $ }
%\\[1mm]
% 0,  \, \quad &\text{  when $ j \, = \,   N+1$. }
%\end{array}
%\right.
%\end{equation}
where function $\mathfrak{T}_0 (\dot{f}_1,  \bar{f}_1, \check{f}_1)$  is of order $O(1)$ and we can show  easily the Lipschitz dependence on $\check{f}_1$
\begin{align*}
|\mathfrak{T}_0 (\dot{f}_1,  \bar{f}_1, \check{f}^{(1)}_1)- \mathfrak{T}_0 (\dot{f}_1,  \bar{f}_1, \check{f}^{(2)}_1)|\leq C|\check{f}^{(1)}_1-\check{f}^{(2)}_1|.
\end{align*}
Then the  equalities in \eqref{eqaution 2} give that, for $j=1, \cdots, N$,
\begin{equation} \label{checkf}
\begin{split}
& 6 \sqrt 2 \epsilon  \beta^2 \Big[ - \gamma_{1,j} {\mathbf d}_{j} \big(e^{-  \sqrt 2 \beta (\check{f}_{j}-\check{f}_{j-1})}-1 \big)
+\gamma_{2,j}{\mathbf d}_{j+1} \big(  e^{-  \sqrt 2 \beta (\check{f}_{j+1}-\check{f}_j)}  -1 \big) \Big]
\\[2mm]
&+\epsilon^{\frac 32}  \mathfrak{T}_0 (\dot{f}_1,  \bar{f}_1, \check{f}_1)\delta_{1j}\, =\,  2\epsilon^{\frac 32}  \beta^2
 \Big[\big(  {\mathbf A}_{1j} + {\mathbf A}_{5j}  \big) \gamma_{3,j}  {\mathbf k}^2_j   +   \big( {\mathbf A}_{3j}+{\mathbf A}_{5j} \big) \gamma_{4,j}  {\mathbf k}^2_{j+1} \Big] +O(\epsilon^{2+\sigma}),
\end{split}
\end{equation}
where
\begin{align*}
\epsilon^{-\frac{1}{2} }\,\gamma_{3,j}   &\,=\, (-1)^j \int_\mathbb{R}  { \chi_j}  e^{-2\sqrt 2 \, x_j } H_{j,x_j}\,  {\mathrm{d}} x_j,
\\[2mm]
\epsilon^{-\frac{1}{2} }  \,  \gamma_{4,j}  &\,=\, (-1)^j \int_\mathbb{R}  { \chi_j}  e^{ 2\sqrt 2 \, x_j } H_{j,x_j}  \,  {\mathrm{d}} x_j,
\end{align*}
and $\gamma_{3,j}, \gamma_{4,j} $   are  positive functions  and of order  $O(1)$, functions ${\mathbf d}_j$, ${\mathbf k}_j$ for $j=1, \cdots, N+1$ are defined in \eqref{mathbfdn}, \eqref{mathbf kj}.
 By using the  similar method in solving system \eqref{berf} and combining with  Contraction Mapping Principle,  we can find solution $\check {\mathbf f} (\theta)= (\check{f}_1(\theta), \cdots, \check{f}_N(\theta))^{T} $  of \eqref{checkf} with the constraints in \eqref{checkfnorm}.

\medskip
Therefore, combining Lemma \ref{lemma of solve linearized pro11} and   orthogonal conditions \eqref{eqaution 2}, we can obtain that there exists a unique solution $\omega^{*}_j (x_j,z)$  satisfies
\begin{equation}\label{omega1}
\begin{split}
\omega^{*}_{j,zz}+ \beta^{2}\Big[\omega^{*}_{j,x_jx_j}+ {F'}(H_j)\omega^{*}_j\Big]=-    \epsilon^{-\frac{3}{2}}\Big[ & \epsilon \beta_1   \check{f}_j \Xi_{0,j}
+  \epsilon^{\frac{3}{2}}\mathcal{K}_3
\\
&+\epsilon  {\chi_j} \Xi^1_{5,j}  +  \epsilon^{\frac{3}{2}}  {\chi_j} \Xi_{11,j} \Big]
\quad \mbox{in }\mathbb{R}\times \Big[0,\frac{\ell}{\epsilon} \Big),
\end{split}
\end{equation}
\begin{equation}\label{omega2}
\omega^{*}_j(x_j, 0)=\omega^{*}_j\Big(x_j, \frac{\ell}{\epsilon}\Big),
\quad
\omega^{*}_{j,z}(x_j, 0)=\omega^{*}_{j,z}\Big(x_j, \frac{\ell}{\epsilon}\Big),
\end{equation}
\begin{equation}\label{omega3}
\int_\mathbb{R}{\omega^{*}_j H_{j, x_j}}\,{\mathrm{d}}x_j =0,\quad  0<z<\frac{\ell}{\epsilon}.
\end{equation}
 Then, we will choose the second correction term as
\begin{equation}\label{phi2}
\begin{split}
\epsilon^{\frac 32 } \,\phi_{2}(s,z)
\, :=\,&
\epsilon^{\frac 32 }\, \sum_{j=1}^N \,  \omega^{*}_j(x_j, z)+ \epsilon^{\frac 32 } \sum_{j=1}^N \omega^{*}_j(- x_j -2\beta f_j, z)
\\
=\,& \epsilon^{\frac 32 }\, \sum_{j=1}^N \,  \omega^{*}_j(\beta(s+f_j), z)+ \epsilon^{\frac 32 } \sum_{j=1}^N \omega^{*}_j(- \beta(s+f_j), z).
\end{split}	
\end{equation}
 Using the symmetry, we can get the boundary error for  $\phi_2$
\begin{equation}\label{boundary3}
  {	\mathcal  {D} }(\phi_2) \, =\,  0.
\end{equation}

\subsubsection{The third   approximate solution and its error}

We define the {\bf{third  approximate solution}} as
\begin{equation}\label{u3}
u_3 (s,z) \,  :  = \, u_2(s,z) + \epsilon^{\frac 32 }\, \phi_2(s,z).
\end{equation}
The error is now expressed as
\begin{equation}\label{Su3}
\begin{split}
{\bf S}(u_3) \,=\,&{\bf S}(u_2)+  L_{u_2}( \epsilon^{\frac 32 }\, \phi_2)
\,+\, N_{u_2}(\epsilon^{\frac 32 }\, \phi_2)   \,+\,B_3 ( \epsilon^{\frac 32 }\, \phi_2)
\\[1mm]
& \,+\,B_ 2(u_2+\epsilon^{\frac 32 }\, \phi_2)\,-\,B_2(u_2)
\,+\,B_4(u_2+\epsilon^{\frac 32 }\, \phi_2)\,-\,B_4(u_2),
\end{split}
\end{equation}
where
\begin{equation*}
L_{u_2} ( \epsilon^{\frac 32 } \phi_2)
\, =\, \epsilon^{\frac 32 } \phi_{2,zz} + \epsilon^{\frac 32 }  \,  \phi_{2,ss } \,+\,   \epsilon^{\frac 32 }  \beta^2 \, (1-3{u_2}^2)\phi_2,
\end{equation*}
and
\begin{equation*}
N_{u_2}(\epsilon^{\frac 32 }\, \phi_2)\, =\,  - \beta^2\Big[        3 u_2 \big (\epsilon^{\frac 32 }\, \phi_2\big)^2+\big (\epsilon^{\frac 32 }\, \phi_2\big)^3 \Big]=O(\epsilon^{2+\sigma}).
\end{equation*}

\medskip
According to the expression of $\epsilon^{\frac 32 } \phi_{2}$ as in \eqref{phi2}, we obtain that
\begin{align}\label{phi2zz+Lu2phi}
 L_{u_2} ( \epsilon^{\frac 32 } \phi_2)
 \,& =\, \epsilon^{\frac 32 }  \sum_{j=1}^N  \Big\{\omega^{*}_{j,zz} +  \beta^2 \big[  \omega^*_{j,x_jx_j} +  (1-3 H_j^2) \omega_j^*\big ]   \Big\}+B_5( \epsilon^{\frac 32 }  \phi_2)
+ \mathcal{K}_5+\mathcal{K}_6
\nonumber\\[1mm]
\,& = -\epsilon \beta_1 \sum_{j=1}^N \check{f}_j \Xi_{0,j}
 - \epsilon^{\frac{3}{2}} \mathcal{K}_3
- \epsilon \sum_{j=1}^N \chi_j \Xi^1_{5,j}
- \epsilon^{\frac{3}{2}}  \sum_{j=1}^N \chi_j \Xi_{11,j}
+B_5( \epsilon^{\frac 32 }  \phi_2)
+ \mathcal{K}_5+\mathcal{K}_6,
\nonumber
\end{align}
where
\begin{equation*}
\begin{split}
B_5( \epsilon^{\frac 32 }  \phi_2) \, &:= \,  \epsilon^{\frac 32 }  \beta^2 \sum_{j=1}^N   3( H_j^2 - u^2_2 ) \omega_j^* +  \epsilon^{ \frac 72 } \sum_{j=1}^N \omega^{*}_{j,x_jx_j} \Big[\frac{\beta'}{\beta}x_j - \beta f_j' \Big]^2
\\[1mm]
&\quad  + \epsilon^{\frac 72 } \sum_{j=1}^N \omega^{*}_{j,x_j}   \Big[ \frac{\beta''}{\beta}x_j  - 2\beta'f_j' -\beta f_j''\Big]
+ \epsilon^{ \frac 52 }\sum_{j=1}^N \omega^{*}_{j,zx_j} \Big[\frac{\beta'}{\beta} x_j- \beta f_j' \Big],
\end{split}
\end{equation*}
\begin{equation*}
\mathcal{K}_5
\,:=\,  \epsilon^{\frac 32 }  \sum_{j=1}^N \Bigg\{ \omega^{*}_{1,zz}(\tau_j, z)+  \beta^2 \Big[    \omega^{*}_{1,\tau_j \tau_j}(\tau_j, z ) + (1-3 H_1^2)\omega^{*}_1(\tau_j,z)  \Big]\Bigg\},
\end{equation*}
\begin{equation*}
\mathcal{K}_6\,:=\,   \epsilon^{\frac 32 } \, \sum_{j=1}^N  \Big\{ \omega^{*}_{j,\tau_j \tau_j}(\tau_j,z ) \tau_{j,z}^2  +2\omega^{*}_{j,\tau_j z} (\tau_j,z ) \tau_{j,z}
+  \omega^{*}_{1,\tau_j}(\tau_j,z)  \tau_{j, zz}  \Big\}=O(\epsilon^{2+\sigma}).
\end{equation*}
We can refer to \eqref{tauj}, \eqref{tauz}, \eqref{tauzz} for the expressions of $\tau_j, \tau_{j,z}, \tau_{j,zz}$.

\medskip
Recalling the definition of operators $ B_2,B_3,B_4$ in \eqref{B2}-\eqref{B4}, we  know that the terms
\begin{align*}
B_3 ( \epsilon^{\frac 32 }\phi_2)
+B_5( \epsilon^{\frac 32 }  \phi_2)
+B_2(u_2+\epsilon^{\frac 32 }\phi_2)-B_2(u_2)
+B_4(u_2+\epsilon^{\frac 32 }\, \phi_2)-B_4(u_2)
=O(\epsilon^{2+\sigma}).
\end{align*}

\medskip
According to the definition of $\hat{f_j}(\theta) $ as in \eqref{hatfj} and the expression of  $\epsilon^2 |\ln\epsilon| \, E_{04}$ in \eqref{E04}, we can decompose $\epsilon^2 |\ln\epsilon| \, E_{04}$ as follows
\begin{equation*}
\epsilon^2 |\ln\epsilon|E_{04}
=\epsilon^2   |\ln\epsilon|    \sum_{j=1}^N \Xi_{13,j}
+\epsilon^2 \sum_{j=1}^N   \Xi_{14,j},
 \end{equation*}
with
\begin{equation}\label{Xi14j}
\begin{split}
\epsilon^2  |\ln\epsilon|  \Xi_{13,j}
& := \, \epsilon^2\, \Big[  -2 \beta'(\dot{f}'_j + \bar{f}'_j + \check{f}'_j    )
- \beta  ( \dot{f}''_j + \bar{f}''_j + \check{f}''_j    )\Big ]H_{j,x_j}
 \\[2mm]
 &\quad-2\epsilon^2  \beta'(   \dot{f}'_j+ \bar{f}'_j + \check{f}'_j  ) x_jH_{j,x_jx_j}
-\epsilon^2  k^2  \beta   (  \dot{f}_j+ \bar{f}_j + \check{f}_j    )  \, H_{j,x_j}
 \\[2mm]
 &\quad
+ \epsilon^2\beta^{-1}  \beta_2    ( \dot{f}_j + \bar{f}_j + \check{f}_j   )x_j F(H_j),
\end{split}
\end{equation}
and
\begin{equation}\label{Xi13j}
\begin{split}
\epsilon^2   \Xi_{14,j}
& := \, \epsilon^2 \big(-2 \beta'  {\tilde f}'_j  - \beta   {\tilde f}''_j   \big )H_{j,x_j}
-2\epsilon^2  \beta' {\tilde f}'_j\, x_jH_{j,x_jx_j}
 \\[2mm]
&\quad  \, -\epsilon^2  k^2  \beta     {\tilde f}_j\, H_{j,x_j}  + \epsilon^2  \beta^{-1}  \beta_2  {\tilde f}_j \, x_j F(H_j) .
\end{split}
\end{equation}

\medskip
According to the expression of $u_1$ in \eqref{expressionofu1}, we can obtain,  for $(s,z)\in {\mathfrak A}_j$, $j\, =\, 1,\cdots,N $
\begin{equation*}
\begin{split}
&H_j^2- u_1^2 \, =\, H_j^2 - \Big[ (-1)^{j} u_1 \Big]^2
 \\[1mm]
 \, &=\, H_j^2 - \Big[   H\big(\, \beta(s-f_j)\big)  -{\mathbf b}_{1j} +{\mathbf b}_{2j}-{\mathbf b}_{3j}+{\mathbf b}_{4j} +O(\epsilon^{2+\sigma}) \Big]^2
 \\[1mm]
  \, &=\,  2 H(x_j) \Big[ {\mathbf b}_{1j}-{\mathbf b}_{2j}+{\mathbf b}_{3j}-{\mathbf b}_{4j}         \Big] + O(\epsilon^{1+\sigma}).
\end{split}
\end{equation*}
Therefore, we can write it as follows
\begin{align*}
&3 \, \epsilon\beta^2 \sum_{j=1}^{N}  \,    (  H_j^2 -  u_1^2) \Big[\beta_1 \, \beta^{-2}\,\dot {f}_j \, \psi_j+\psi_j^*+ \psi^{**}_j \Big]
\\
&\,=\,  6\,\epsilon\, \beta^2\,\sum_{j=1}^{N}  H(x_j) ( {\mathbf b}_{1j}+ {\mathbf b}_{3j}   )   \Big[ \beta_1 \beta^{-2} \dot{f}_j  \psi_j+\psi_j^*+ \psi^{**}_j \Big]
+O(\epsilon^{2+\sigma})
\\
&=\, 12 \, \epsilon^2 \beta^2 \sum_{j=1}^{N}   H(x_j)\Big\{-  {\mathbf k}_j  e^{-\sqrt 2 x_j } e^{-  \sqrt 2\beta ({\tilde f}_{j}-{\tilde f}_{j-1})}
\,+{\mathbf k}_{j+1}  e^{\sqrt 2 x_j } e^{-2\sqrt{2} \beta({\tilde f}_{j+1} -{\tilde f}_{j}  )}  \Big\}
\\
& \qquad \qquad \qquad \qquad \quad
\times\Big[ \beta_1\dot{f}_j \beta^{-2} \psi_j+\psi_j^*+ \psi^{**}_j \Big]  +O(\epsilon^{2+\sigma})
\\
&= \,12\,\epsilon^2 \, \beta^2\,\sum_{j=1}^{N}H(x_j)  \Big[ -{\mathbf k}_j  e^{-\sqrt 2 x_j }
  +   {\mathbf k}_{j+1}    e^{\sqrt 2 x_j }          \Big] \times \Big[ \beta_1 \dot{f}_j  \beta^{-2} \psi_j+\psi_j^*+ \psi^{**}_j \Big]
\\
&\quad +12 \epsilon^2 \beta^2 \sum_{j=1}^{N}  H(x_j) \Big\{-  {\mathbf k}_j  e^{-\sqrt 2 x_j } \big[e^{-  \sqrt 2\beta ({\tilde f}_{j}-{\tilde f}_{j-1})} -1 \big]+{\mathbf k}_{j+1} e^{\sqrt 2 x_j }  \big[ e^{-2\sqrt{2} \beta({\tilde f}_{j+1} -{\tilde f}_{j}  )}  -1\big] \Big\}
\\
&\qquad \qquad \qquad   \qquad \qquad
\times\Big[ \beta_1\, \dot{f}_j \beta^{-2}\, \psi_j+\psi_j^*+ \psi^{**}_j \Big]  +O(\epsilon^{2+\sigma})
\\[1mm]
&= \, \epsilon^2 \, |\ln \epsilon| \, \sum_{j=1}^{N}  \Xi_{15,j} \,+\,  O(\epsilon^{2+\sigma}),
\end{align*}
with
\begin{equation}\label{Xi15}
 \epsilon^2  |\ln \epsilon|    \Xi_{15,j}
 := 12\epsilon^2  \beta^2\sum_{j=1}^{N}H(x_j) \Big( -{\mathbf k}_j  e^{-\sqrt 2 x_j }
  +{\mathbf k}_{j+1}e^{\sqrt 2 x_j }   \Big) \Big[ \beta_1 \dot{f}_j  \beta^{-2} \psi_j+\psi_j^*+ \psi^{**}_j \Big].
\end{equation}

\medskip
Combining with \eqref{S(u21)},  \eqref{Su3}, and the expressions for
$$L_{u_2} ( \epsilon^{\frac 32 } \phi_2),  \quad    \epsilon^2 |\ln\epsilon|E_{04},    \quad   3 \, \epsilon\beta^2 \sum_{j=1}^{N}  \,    (  H_j^2 -  u_1^2) \Big(\beta_1 \,\dot {f}_j \, \beta^{-2}\, \psi_j+\psi_j^*+ \psi^{**}_j \Big)   $$
 we can obtain
\begin{align}\label{S(u3)}
 {\bf S} (u_3) \, = \,&  \epsilon \,\beta_1\sum_{j=1}^{N}  {\tilde f}_j \Xi_{0,j}
  \,+\,\epsilon \sum_{j=1}^N  \chi_j \Xi_{5,j}^2
   \,+\,\epsilon^2 |\ln\epsilon| \sum_{j=1}^{N}\Xi_{2,j}
   \,+\,  \epsilon^2 |\ln \epsilon|   \sum_{j=1}^{N} \chi_j\Xi_{6,j}
      \nonumber
\\[1mm]
&  \,+\, \epsilon^2  |  \ln \epsilon|  \,  \sum_{j=1}^N  \Xi_{7,j}  +  \epsilon^2   \sum_{j=1}^N  \Xi_{8,j}
\,+\,  \epsilon^2  |\ln \epsilon | \sum_{j=1}^N  \Xi_{9,j}
\,+\, \epsilon^2 |\ln \epsilon| \sum_{j=1}^N \Xi_{10,j}
\\[2mm]
 \quad  &
 \,+\, \epsilon^2 |\ln\epsilon|   \sum_{j=1}^N  \chi_j \Xi_{12,j}
 \,+\,\epsilon^2   |\ln\epsilon|    \sum_{j=1}^N \Xi_{13,j}
 \,+\,\epsilon^2 \sum_{j=1}^N   \Xi_{14,j}
 \,+\,\epsilon^2  |\ln \epsilon|  \sum_{j=1}^{N}  \Xi_{15,j}
\nonumber
\\[2mm]
 \quad  & \,+\,\epsilon^2 |\ln\epsilon|^2 \, E_{03} \,+\, \epsilon^2 E_{06} \,+\,  N_{u_1}(\phi_1)  \,+\,  \mathcal{Q}_1\,+\, \bar{ \mathcal{K}},  \nonumber
 \end{align}
 with
 \begin{align}\label{mathcalK}
\bar{ \mathcal{K}}\,:=\,  \mathcal{K}_2 \,+\, \mathcal{K}_4 \,+\,\mathcal{K}_5\,+\,\mathcal{K}_6,
 \end{align}
 and
 \begin{equation}\label{mathcalQ1}
 \begin{split}
  \mathcal{Q}_1 := \,&  B_5( \epsilon^{\frac 32 }  \phi_2)  + N_{u_2}(\epsilon^{\frac 32 }\, \phi_2)   \,+B_3 ( \epsilon^{\frac 32 }\, \phi_2)   \,+\,B_ 2(u_2+\epsilon^{\frac 32 }\, \phi_2)
\\[1mm]
 \quad  &
 \,-\,B_2(u_2)
\,+\,B_4(u_2+\epsilon^{\frac 32 }\, \phi_2)\,-\,B_4(u_2) \,+\, O(\epsilon^{2+\sigma})
\\
=\,&O(\epsilon^{2+\sigma}).
 \end{split}
 \end{equation}

In the last of this part, we will give the  Neumann   boundary  error of the third approximate solution.
By calculating and combining the results \eqref{boundary1},  \eqref{boundary2}, \eqref{boundary3},    we can get
\begin{align}\label{Du4}
   {\mathcal{D} } (u_3)\,=\, 0.
 \end{align}

%where
%\begin{align}
% \epsilon^2  |\ln \epsilon|   \Xi_{11,j}
%  & = 6  \epsilon^2 \, \beta^2 H(x_j)  \Big[   - e^{-\sqrt 2  x_j } {\mathbf l}_j\,      +e^{ \sqrt 2 x_j } {\mathbf l}_{j+1}      \Big]   \Big[ \beta_1\,  \dot{f}_j  \, \beta^{-2}\, \psi_j+\psi_j^* \Big],
%\end{align}
%and
%\begin{align}\label{Xi12j}
%\epsilon^2 \, \Xi_{12,j} &  =  6   \epsilon^2 \beta^2 H(x_j)  \Big\{    - e^{-\sqrt 2 \, x_j } {\mathbf l}_j\,      +e^{ \sqrt 2 \, x_j } {\mathbf l}_{j+1}      \Big\}   \psi^{**}_j  \nonumber
% \\[1mm]
% &\quad +  6  \epsilon^2  \beta^2 H(x_j)  \Big\{    - e^{-\sqrt 2 \, x_j } {\mathbf l}_j\, \big(    e^{-\sqrt 2 \, \beta ({\tilde f}_j-{\tilde f}_{j-1}) } -1\big)
%+e^{ \sqrt 2 \, x_j } {\mathbf l}_{j+1}  \big(    e^{-\sqrt 2 \, \beta ({\tilde f}_{j+1}-{\tilde f}_j) } -1\big)      \Big\}
%\nonumber
% \\[1mm]
% &
% \quad
% \times \Big[ \beta_1\,  \dot{f}_j  \, \beta^{-2}\, \psi_j+\psi_j^*  +\psi^{**}_j \Big]+O(\epsilon^{2+\sigma}).
%\end{align}

\section{The gluing procedure}
\label{section4}
\setcounter{equation}{0}

Recall the set ${\mathfrak S}$ in \eqref{huaxies} and the coordinates $(s, z)$ in \eqref{sz}.
The gluing method from \cite{delPKowWei1} will be used to transform problem \eqref{originalproblem1} in $\Omega_{\epsilon}$
into a projected problem on the infinite strip ${\mathfrak S}$, in which we can use the coordinates $(s, z)$.
We define a smooth cut-off function
$
\eta_{3\delta}^{\epsilon}(s)\,=\, \eta_{3\delta}(\epsilon |s|),
$
where $\eta_{\delta}(t)$ is given by
$$
\eta_{\delta}(t)=1\ \mbox{ for }\ 0\leq t\leq \delta
\qquad \mbox{and}\qquad
\eta_{\delta}(t)=0
\ \mbox{ for }\  t>2\delta,
$$
for any fixed number $\delta<\delta_0/100$ with  small $\delta_0$ given in (\ref{fermi}). For a solution in Theorem \ref{main1},
the global approximation   $\mathbf{H}({\tilde y})$ can be defined simply in the form
\begin{equation}\label{globalapproximation1}
\mathbf{H}({\tilde y})=
\eta_{3\delta}^\epsilon(s)\big(u_3(s,  z)+  (-\, 1)^{N+1}\big)\, +  (-\, 1)^N ,\quad \forall\, {\tilde y}\in \Omega_{\epsilon},	
\end{equation}
where $u_3$ is defined as in \eqref{u3} and the relation between ${\tilde y}$ and $(s,z)$ is given in \eqref{sz}.

\medskip
For a perturbation term $\Phi(\tilde y)\,=\, \eta_{3\delta}^{\epsilon}{\tilde \phi}+{\tilde \psi}$ defined in $\Omega_\epsilon$,  the function $u(\tilde y)=\mathbf{H}(\tilde y)+\Phi(\tilde y)$
satisfies (\ref{originalproblem1}) if the pair $({{\tilde \phi}}, {{ \tilde \psi}})$ satisfies the following coupled system:
\begin{align}
\eta_{3\delta}^{\epsilon}\mathbf{L}(\tilde \phi)
\,=\, &
\eta_{\delta}^{\epsilon}\Big[-\mathcal{E}
+\mathbf{N}(\eta_{3\delta}^{\epsilon}
{\tilde \phi}+{\tilde \psi})-3\,  V \, (1-\mathbf{H}^2)\, {\tilde \psi}\Big]\quad\mbox{in } \Omega_\epsilon,
\label{equivalent-system-1}
\end{align}
\begin{equation}\label{system-boundary-1}
\frac {\partial {\tilde\phi}}{\partial\nu}\, +\, \frac {\partial \mathbf{H}}{\partial\nu}\,=\, 0\quad\mbox{on } \partial\, \Omega_\epsilon,
\end{equation}
and
\begin{align}
\begin{aligned}
\Delta {\tilde \psi}-2V{\tilde \psi}+3V(1-\eta_{\delta}^{\epsilon})(1-\mathbf{H}^2){\tilde \psi}
\,=\, &
-\epsilon^2 (\Delta \eta_{3\delta}^{\epsilon}){ \tilde \phi}-2\epsilon \nabla \eta_{3\delta}^{\epsilon} \cdot
\nabla{\tilde \phi}
-(1-\eta_{\delta}^{\epsilon})\mathcal{E}
\\[2mm]
&+(1-\eta_{\delta}^{\epsilon})\mathbf{N}(\eta_{3\delta}^{\epsilon}
{\tilde \phi}+{\tilde \psi})\quad\mbox{in } \Omega_\epsilon,
\label{equivalent-system-2}
\end{aligned}
\end{align}
\begin{equation}
\label{system-boundary-2}
\frac {\partial\,  {{\tilde \psi}}}{\partial\nu}\,=\, 0
\quad\mbox{on } \partial\, \Omega_\epsilon,
\end{equation}
where
\begin{align*}
\mathbf{L}(\tilde \phi)\,=\, \Delta \tilde \phi\, +\, V& \,  \big[1-3\mathbf{H}^2 \big]\, \tilde \phi,
\\[2mm]
\mathcal{E}\,=\, \Delta \mathbf{H}\, +V \,  \mathbf {H} \, (1-\mathbf{H}^2),
&\qquad
\mathbf{N}({\tilde \phi})\,=\, V {\tilde \phi}^3\, +\, 3 V \, \mathbf{H}\, {\tilde \phi}^2\,.
\end{align*}

\medskip
First,  given a small $ \tilde {\phi}$ satisfies the following decay property
\begin{equation}\label{decay-assumption}
|\nabla {\tilde \phi}|\, +\, |{\tilde \phi}|\, \leq\,  e^{- { \upsilon_1\,  \delta}/{\epsilon}} \, \, \, \,   {\rm for }\, \, \,  s>\delta/\epsilon,
\end{equation}
where $\upsilon_1$ is a very small constant. Note that  the nonlinear operator $\mathbf{N}$ has a power-like behavior with power greater than one, and a direct application of Contraction Mapping Principle yields
that (\ref{equivalent-system-2})-(\ref{system-boundary-2}) has a unique (small) solution ${\tilde \psi}={\tilde \psi}({\tilde \phi})$ with
\begin{equation}\label{contraction-varphi}
\|{\tilde \psi}({ \tilde \phi})\|_{L^{\infty}}\, \leq \, C\, \epsilon \lf[\|{\tilde \phi}\|_{L^{\infty}(s>\delta/\epsilon)}
\, +\, \|\nabla {\tilde \phi}\|_{L^{\infty}(s>\delta/\epsilon)}\ri]
\, +\,
e^{-\frac { \upsilon_2 \delta}{\epsilon}},
\end{equation}
where $s>\delta/\epsilon$ denotes the complement  of $\delta/\epsilon$-neighborhood of
$\partial\, \Omega_\epsilon$ and $\upsilon_2$ denotes some small constant. Moreover,   the nonlinear operator $\tilde {\psi}$ satisfies a Lipschitz condition of the form
\begin{equation}\label{varphi-lip}
\|{\tilde \psi}({\tilde \phi}_1)-{\tilde \psi}({ \tilde \phi}_2)\|_{L^{\infty}}
\leq
C\epsilon \lf[ \|{\tilde \phi}_1-{\tilde \phi}_2\|_{L^{\infty}(s>\delta/\epsilon)}
+\|\nabla {\tilde \phi}_1-\nabla {\tilde \phi}_2\|_{L^{\infty}(s>\delta/\epsilon)}\ri].
\end{equation}

\medskip
Therefore, after solving (\ref{equivalent-system-2})-\eqref{system-boundary-2},
we can concern (\ref{equivalent-system-1})-\eqref{system-boundary-1} as  a local nonlinear problem  involving ${\tilde \psi}={ \tilde \psi}({\tilde \phi})$,
which can be expressed in local coordinates $(s,  z)$.
Here is the setting-up.
\\
\\
\noindent(1).
Note first that $\frac {\partial \mathbf{H}}{\partial\nu}=0$ due to \eqref{Du4}.
The boundary condition in \eqref{system-boundary-1} has the form
\begin{align}
    {\tilde \phi}_s(0,z)  \, =\,  0,    \quad 0< z< \frac{\ell}{\epsilon}.
\end{align}

\noindent(2).
We make an extension and consider the function $\tilde \phi$ defined in the strip  ${\mathfrak S}$.
In terms of $ (s,z) $, the operator ${\mathbf L}$ can also be extended to a new one, say ${\mathcal L}$,
so that they coincide on the region $0<s< \frac{10 \delta}{\epsilon}$.
In fact, by using the computations in Section \ref{section2}, we set
\begin{align}\label{operator mathbf L}
{\mathcal L}({ \tilde \phi})
\, &=\,
{\mathcal L}_0({ \tilde \phi})
\,+\,
\eta^{\epsilon}_{10 \delta}(s)  \Big[ \, B_1(\tilde \phi) + B_6(\tilde \phi) \Big],
\end{align}
where
\begin{align}
\label{tilde L check phi}
{\mathcal L}_0({ \tilde \phi})
\,=\,{\tilde \phi}_{ss}\,+\,{\tilde \phi}_{zz}
\,+\, \beta^2 \, (1-3\mathbf{H}^2)\, \tilde \phi,
\end{align}
and $B_6(\tilde \phi)$ is defined
\begin{equation}
B_6(\tilde \phi)
\,=\,
\Big[ \epsilon \beta_1 s + \frac{\epsilon^2}{2} \beta_2  s^2  +a_4(\epsilon s, \epsilon x) \epsilon^3 s^3      \Big]
\big(1 - 3 \mathbf{H}^2\big) \tilde \phi,
\end{equation}
where $B_1$ and $a_4$ are given in \eqref{B1u} and \eqref{Vexpansion}.

\medskip
\noindent(3).
For the local form of the nonlinear part,  we denote by the notation
\begin{align}\label{Nmathcal}
{\mathcal N}(\tilde\phi)={\mathbf N}({\eta_{3\delta}^{\epsilon}{\tilde\phi}}+{\tilde\psi}({\tilde\phi}))
-3V(1-\mathbf{H}^2)\tilde \psi({\tilde\phi}).
\end{align}

\noindent(4).
The term ${ \mathcal{E}}$ can be locally recast in $(s, z)$  coordinate system   by the relation
\begin{equation}\label{errorrelationinterior}
\eta_\delta^\epsilon(s)\, {\mathcal{E}}\,=\, \eta_\delta^\epsilon(s)\, {\bf S}(u_3),
\end{equation}
where the clear expression of ${\bf S}(u_{3})$,  reader can refer to \eqref{S(u3)}.

\medskip
Note that the approximate solution $u_3$ has unknown parameters ${\tilde f}_1,\cdots,{\tilde f}_N$,  see \eqref{hatfj} and \eqref{ddotfnorm}.
We will deal with the following projected problem:
for $ {\tilde{\bf f}}=({\tilde f}_1, \cdots, {\tilde f}_N)^T$,  finding
functions $\tilde \phi\in H^2({\mathfrak S}), \,  \mathbf{c}=(c_1, \cdots, c_N)$ with $c_j \in L^2(0, \ell)$,
such that
\begin{equation}
\label{Projectedproblem1}
{\mathcal L}(\tilde \phi)\,=\, \eta_{\delta}^{\epsilon}(s)\,
\big[-{\mathcal E}+{\mathcal N}(\tilde \phi)\, \big]
\,+\, \sum_{j=1}^N\, c_j(\epsilon z)\chi_j(s, z)\,  H'\big(    \beta(s-f_j)\big)\quad\mbox{in } {\mathfrak S},
\end{equation}
\begin{equation}
{\tilde \phi }(s,0 ) = {\tilde \phi }\Big(s,    \frac{\ell}{\epsilon} \Big), \quad {\tilde \phi }_z(s,0 ) = {\tilde \phi }_z\Big(s,    \frac{\ell}{\epsilon} \Big),
\quad
0<s<+\infty,
\end{equation}
\begin{equation}
    {\tilde \phi}_s(0,z)  \, =\,  0,    \quad 0< z< \frac{\ell}{\epsilon},
   \end{equation}
\begin{equation}
\label{projectedproblem3}
\int_0^{+\infty} \tilde \phi(s, z)\, \chi_j(s, z)\, H'\big(    \beta(s-f_j)  \big) \, {\rm d}s \,=\, 0, \quad 0<z<\frac {\ell }{\epsilon}, \quad j=1, \cdots, N.
\end{equation}
The smooth cut-off functions $\chi_1,\cdots,\chi_N$ are defined by (\ref{cutoffchi}).
%\begin{align}
%\chi_j(x, z) =\eta_a^b\Big(\frac{x-f_j(\epsilon z)}{|\ln\epsilon|}\Big), \ \mbox{where }\,  a=\sqrt{2}\frac{2^6-1}{2^7}, \,  b = \sqrt{2}\frac{2^7-1}{2^8},  \,   \eta_a^b(t)=\begin{cases} 1,  \  \ |t|<a\\
%0, \ \ |t|>b.
%\end{cases}
%\label{def chim}
%\end{align}
%We notice that with this choice $\chi_j\chi_i:= 0$,  for $i\neq j$,   provided that $\epsilon$ is taken sufficiently small.
The resolution theory for $\tilde \phi$ with the constraint (\ref{decay-assumption}) can be provided in the following:
\begin{proposition}\label{proposition3.1}
There exist  numbers $D>0$,  $ \upsilon_3>0$ such that for all sufficiently small $\epsilon$ and all ${\tilde{\bf f}} =({\tilde f}_1, \cdots, {\tilde f}_N)^T $
satisfying \eqref{ddotfnorm},  problem \eqref{Projectedproblem1}-\eqref{projectedproblem3}
has a unique solution $\tilde \phi\, =\, \tilde \phi({\tilde{\bf f}})$ which satisfies
\begin{align}\label{tilde phi 1}
\|\tilde \phi\|_{H^2({\mathfrak S})} \, \le\,  D \epsilon^{\frac{3}{2}}|\ln\epsilon|^2,
\end{align}
\begin{align} \label{tilde phi 2}
\|\tilde \phi\|_{L^\infty(s>\delta/\epsilon)} +\|\nabla  {\tilde\phi}\|_{L^\infty(s>\delta/\epsilon)}\  \le\ e^{\, -{\upsilon_3 \delta}/{\epsilon}}.
\end{align}
Besides $\tilde \phi$ is a Lipschitz function of   $\, {\tilde{\bf f}}$,
and for given ${\tilde{\bf f}}_1, {\tilde{\bf f}}_2:(0, \ell )\to {\mathbb R}^N$ with constraint in \eqref{ddotfnorm}
%\begin{align}
%\|{\tilde{\bf f}}_1-{\tilde{\bf f}}_2\| \leq C \epsilon^{ \frac 12 + \sigma},
%\label{flip1}
%\end{align}
there holds
\begin{align}\label{flip2}
\|\tilde \phi({\tilde{\bf f}}_1)-\tilde \phi({\tilde{\bf f}}_2)\|_{H^2({\mathfrak S})}\leq C\epsilon \,  \|{\tilde{\bf f}}_1-{\tilde{\bf f}}_2\|_{H^2(0,\ell)}.
\end{align}
%where the norm $\|\,  {\tilde{\bf f}} \, \| $ is defined in \eqref{ddotfnorm}.
\end{proposition}

\vspace{1mm}
\begin{proof}
 Based on  linearized theory developed in \cite{delPKowWei2}  and contraction mapping theorem,
we can prove  Proposition \ref{proposition3.1} as done for Proposition 5.1 in \cite{delPKowWei2}.
The details of the proof will be omitted here.
\end{proof}

 \medskip
   \section{The deriving of the reduced equations}\label{section6}
As a standard step in reduction method to make the  Lagrange multiplier ${\mathbf c}$ vanish in \eqref{Projectedproblem1}, we will set up a system of differential equations involving ${\tilde{\bf f}}$ with constraints in \eqref{ddotfnorm}
 such that
\begin{equation}\label{equivalancec}
\mathbf{c}({\tilde{\bf f}})\,=\,  0.
\end{equation}
These equations are obtained by simply integrating the equation (\ref{Projectedproblem1})(only in $s$) against of the functions
$$H_{n, x_n}= (-1)^n H' \big( \,  \beta(s-f_n )       \,  \big),
\quad
n=1,\cdots, N.
$$
In fact,
it is easy to derive that \eqref{equivalancec} is equivalent to the following relations
   \begin{align}\label{equivalancec1}
\int_0^{+\infty}
\Big[ \eta_{\delta}^{\epsilon}(s){\mathcal E}(s,z)  - \eta_{\delta}^{\epsilon}(s){\mathcal N}(\tilde \phi)\,+{\mathcal L}(\tilde \phi) \Big]     H_{n, x_n} \,{\mathrm{d}}s\, =\, 0  \quad \text{for } n\, =\, 1,\cdots,  N.
   \end{align}
We will give the details of computations for the terms in \eqref{equivalancec1} in the sequel.

\medskip
To start the computations for the first term in \eqref{equivalancec1},
by recalling the definitions of   ${\mathfrak A}_n$   in  \eqref{{mathfrak A}_1}-\eqref{{mathfrak A}_n}, for each $ n $ with $ n = 1, \cdots, N, $
we introduce the notation
\begin{equation}
\label{mathfrakSn}
{\mathfrak R}_n=\Big\{      s\in    (0,+\infty) : (s,z) \in   {\mathfrak A}_n\Big\}.
\end{equation}
Then, we will consider the integrals for $ n = 1, \cdots, N$,
\begin{equation}
\begin{split}
\int_0^{+\infty}       \, \eta_{\delta}^{\epsilon}(s)\,
 {\mathcal E} (s,z)\,     H_{n, x_n} \,{\mathrm{d}}s
  \,=\,&    \Big\{     \int_ { {\mathfrak R}_n} +   \int_ {  \mathbb R^{+}   \setminus  {\mathfrak R}_n} \Big\}    \eta_{\delta}^{\epsilon}(s)\,
 {\mathcal E} (s,z)\,   H_{n, x_n} \,{\mathrm{d}}s
 \\[2mm]
 \, := \,&  \mathbb{E}_{n1}\,+\, \mathbb{E}_{n2}.
\end{split}
\end{equation}
Note  that, if $(s,z)\in {\mathfrak R}_n$, we have
 \begin{equation*}
\eta_\delta^\epsilon(s)\, {\mathcal{E}}(s,z)\,=\,  {\bf S}(u_{3}).
\end{equation*}
By recalling $ x_n = \beta(s-f_n)$, then we obtain that

\begin{align*}
  \mathbb{E}_{n1} \,
  & =   \epsilon  \sum_{j=1}^{N} \int_ {{\mathfrak R}_n}    \beta_1   {\tilde f}_j \Xi_{0,j} H_{n, x_n}\,{\mathrm{d}}s
  +  \epsilon   \sum_{j=1}^N   \int_ {{\mathfrak R}_n}  \chi_j \Xi_{5,j}^2     H_{n, x_n}\,{\mathrm{d}}s
  + \epsilon^2   \sum_{j=1}^N   \int_ {{\mathfrak R}_n}     \Xi_{14,j}        H_{n, x_n}\,{\mathrm{d}}s    \nonumber
\\[1mm]
 &  \quad
+\sum_{j=1}^N\int_ {{\mathfrak R}_n}\Big\{\epsilon^2 |\ln\epsilon| \Xi_{2,j}
+ \epsilon^2 |\ln \epsilon|\chi_j \Xi_{6,j}
+  \epsilon^2 |\ln \epsilon|\,  \Xi_{7,j}+  \epsilon^2 \,  \Xi_{8,j}
+  \epsilon^2 \,  |\ln \epsilon | \, \Xi_{9,j}     \nonumber
\\[1mm] &
\qquad \qquad \qquad
+  \epsilon^2   |\ln \epsilon |    \Xi_{10,j}
+  \epsilon^2|\ln \epsilon |\chi_j\Xi_{12,j}
+ \epsilon^2   |\ln\epsilon| \Xi_{13,j}
+ \epsilon^2  |\ln \epsilon| \Xi_{15,j} \Big\}  H_{n, x_n}\,{\mathrm{d}}s
\\[1mm]
& \quad  \,+\,\epsilon^2|\ln\epsilon|^2  \int_ {{\mathfrak R}_n}  E_{03}  H_{n, x_n}\,{\mathrm{d}}s
\,+\,\epsilon^2\int_ {{\mathfrak R}_n}  E_{06}  H_{n, x_n}\,{\mathrm{d}}s \nonumber
 \\[1mm]
& \quad
\,+\, \int_ {{\mathfrak R}_n}  N_{u_1}(\phi_1) H_{n, x_n}\,{\mathrm{d}}s
\,+\,\int_ {{\mathfrak R}_n}\bar{ \mathcal{K}} H_{n, x_n}\,{\mathrm{d}}s
\,+\,\int_ {{\mathfrak R}_n}   \mathcal{Q}_1  H_{n, x_n}\,{\mathrm{d}}s
\nonumber
\\[2mm]
& = I^n_0 \,+\, I^n_1\,+\,I^n_2\,+\,I^n_3\,+\,I^n_4\,+\,I^n_5\,+\,I^n_6\,+\,I^n_7\,+\,I^n_8.
 \end{align*}
For ${\tilde{\bf f}}$ satisfying \eqref{ddotfnorm}, we denote uniformly
bounded continuous functions of the form
\begin{equation*}
 \mathfrak{Z}_n^0 (  {\tilde{\bf f}},  {\tilde{\bf f}}',  {\tilde{\bf f}}''  )(\theta)
\quad \text{and} \quad
 \mathfrak{Z}_n^j(  {\tilde{\bf f}} )(\theta), \quad  \text{for}~ j=1,\cdots,8,
\end{equation*}
and    uniformly
bounded continuous functions  independent of $ {\tilde{\bf f}}$  of the form
\begin{equation*}
G_n^j(\theta), \qquad j= 1,\cdots,7.
\end{equation*}
 In the sequel, we will give the estimates of the above terms one by one.

\medskip
  By using the constraint of ${\tilde{\bf f}} = ({\tilde f}_1, \cdots, {\tilde f}_N)^{T} $  in \eqref{ddotfnorm} and $\Xi_{0,j}$ in \eqref{xi1j}, it is easy to obtain that for $n =1, \cdots, N$
\begin{align}
I^n_0
\,=\, & \epsilon\sum_{j=1}^{N} \int_ {{\mathfrak R}_n}  \beta_1 {\tilde f}_jH_j(1-H_j^2) H_{n, x_n}\,{\mathrm{d}}s \nonumber
\\[1mm]
\,=\, & \epsilon\int_ {{\mathfrak R}_n}  \beta_1 {\tilde f}_n\,H_n(1-H_n^2) H_{n, x_n}\,{\mathrm{d}}s\,+\,\epsilon\sum_{ j\neq n }^{N} \int_ {{\mathfrak R}_n}  \beta_1 {\tilde f}_jH_j(1-H_j^2) H_{n, x_n}\,{\mathrm{d}}s
\\[1mm]
\, =\,&O(\epsilon^{2}|\ln \epsilon|)  \sum_{j=1}^{N} |{\tilde f}_j(\theta)|.\nonumber
\end{align}

\medskip
According to the expression of $\Xi_{5,j}^2$ as in \eqref{Xi5j^2}, we can derive that
\begin{equation}
\begin{split}
I^n_1
 % \, =\,&\epsilon   \sum_{j=1}^N   \int_ {{\mathfrak R}_n}   \Xi_{5,j}^2     H_{n, x_n}\,{\mathrm{d}}s
 % \\
 \,=\,&  - 6\, \sqrt 2\, \epsilon\,   \frac{1}{\gamma_{0,n} }\, \beta^2
 \Big[  - \gamma_{1,n}\,   {\mathbf k}_{n} \,  \big(\,e^{\,-  \sqrt 2 \,\beta  \,({\tilde f}_{n}-{\tilde f}_{n-1})\,}-1 \, \big)
 \\[1mm]
 \,&\qquad \qquad \qquad \quad  +\, \gamma_{2,n}\,     {\mathbf k}_{n+1} \,  \big(  e^{\,-  \sqrt 2 \,\beta  \,({\tilde f}_{n+1}-{\tilde f}_n)\,}  -1 \big)   \Big] \int_ {{\mathfrak R}_n} \, {\chi_n}    \,H_{n,x_n}^2 \,{\mathrm{d}}s
 \\[2mm]
= \,&6 \,\sqrt 2\,\epsilon\,  \beta\,
 \gamma_{1,n} \,{\mathbf k}_{n}\,    \Big[e^{-  \sqrt 2 \beta  ({\tilde f}_{n}-{\tilde f}_{n-1})}-1   \Big]
 \\[2mm]
&\,-\, 6 \,\sqrt 2\,\epsilon\,  \beta\, \gamma_{2,n} \,   {\mathbf k}_{n+1} \Big[  e^{-  \sqrt 2 \beta ({\tilde f}_{n+1}-{\tilde f}_n)}  -1  \Big]\,+\, O(\epsilon^{2}|\ln \epsilon|)  \sum_{j=1}^{N} |{\tilde f}_j(\theta)|,
\end{split}
\end{equation}
where $ \gamma_{1,n}, \gamma_{2,n} $ are functions defined in \eqref{ga1n}, \eqref{ga2n}.

\medskip
By the definition of $\Xi_{14,j}$ in \eqref{Xi13j}, it follows that
  \begin{equation}
\begin{split}
I^n_2
% &=\epsilon^2   \sum_{j=1}^N   \int_ {{\mathfrak R}_n}     \Xi_{14,j}        H_{n, x_n}\,{\mathrm{d}}s
% \\
&=\, \epsilon^2\big(-2 \beta'{\tilde f}'_n- \beta {\tilde f}''_n\big ) \int_ {{\mathfrak R}_n} H_{n,x_n}^2\,{\mathrm{d}}s
-2\epsilon^2\beta' {\tilde f}'_n\int_ {{\mathfrak R}_n}x_nH_{n,x_nx_n}H_{n,x_n}\,{\mathrm{d}}s
 \\[2mm]
&\quad + \epsilon^2\beta^{-1}  \beta_2{\tilde f}_n\int_ {{\mathfrak R}_n}x_n F(H_n) H_{n,x_n}\,{\mathrm{d}}s
-\epsilon^2  k^2  \beta {\tilde f}_n\int_ {{\mathfrak R}_n}\, H_{n,x_n}^2 \,{\mathrm{d}}s
\\[2mm]
&\,=\, -\, \epsilon^2 \gamma_0   {\tilde f}''_n
   \,-\, \epsilon^2 \gamma_0 \beta' \beta^{-1} {\tilde f}'_n
   \,+\, \epsilon^2 \gamma_0 \Big[ \frac 12   \beta_2 \beta^{-1}
   \,-\,  k^2  \beta   \Big]\beta^{-1} {\tilde f}_n
 \,+\,\epsilon^{2+\sigma}  \mathfrak{Z}_n^0 (  {\tilde{\bf f}},  {\tilde{\bf f}}',  {\tilde{\bf f}}''  )(\theta),
% \\[2mm]
% &=-\epsilon^2 \gamma_0 {\tilde f}''_n
% + \epsilon^{2+\sigma} \mathfrak{Z}_n^2(  {\tilde{\bf f}} )(\theta),
\end{split}
\end{equation}
where we have used the constraint of ${\tilde{\bf f}} = ({\tilde f}_1, \cdots, {\tilde f}_N)^{T} $
 in \eqref{ddotfnorm} and the fact
\begin{equation*}
 \gamma_0= \int_\mathbb{R}\,{H_x^2}\, {\mathrm{d}}x, \quad
 2\int_{\mathbb R} x\, H_x\,H_{xx}  \,\mathrm {d }x= -\int_{\mathbb R} H_x^2 \,\mathrm {d }x.
 \end{equation*}

\medskip
When $s \in {\mathfrak R}_n$, it is easy to obtain that
\begin{equation*}
\bar{x}_j=\beta(s+f_j)\in \Big( \frac{\beta f_{n-1}+\beta f_n}{2}+ \beta f_j, \frac{\beta f_{n+1}+ \beta f_n}{2} + \beta f_j\Big).
\end{equation*}
Therefore, according to the asymptotic behavior of $H'(x)$ as in \eqref{asymptoticofH} and the expressions of $f_j\,(j= 1, \cdots, N)$ in \eqref{f01}-\eqref{f2N}, it is easy to obtain that
\begin{equation*}
\int_ {{\mathfrak R}_n} \bar{H}_{j, \bar{x}_j} H_{n, x_n}\,{\mathrm{d}}s
\,=\,  |\ln \epsilon| \,  O( e^{ - { \sqrt 2 }\, \beta \,  | f_j+f_n|}) , \qquad \forall j,n=1,\cdots, N.
\end{equation*}
Then, recalling to the expression of $\epsilon^2 |\ln\epsilon| \Xi_{2,j}$ in \eqref{E05}, we can derive that
\begin{align}\label{integralofE05}
&\epsilon^2 |\ln \epsilon|\sum_{j=1}^N   \int_ {{\mathfrak R}_n}  \Xi_{2,j}  H_{n, x_n}\,{\mathrm{d}}s \nonumber
\\
&=- \epsilon  k\beta \sum_{j=1}^N   \int_ {{\mathfrak R}_n} \bar{H}_{j, \bar{x}_j} H_{n, x_n}\,{\mathrm{d}}s
-\epsilon \, \beta_1  \sum_{j=1}^N \int_ {{\mathfrak R}_n} \Big( \frac{\bar{x}_j}{\beta}-f_j \Big)  F(\bar{H}_j)  H_{n, x_n}\,{\mathrm{d}}s\\[2mm]
%\\
% &=- \epsilon  k   \int_ {{\mathfrak R}_n} \bar{H}_{n, \bar{x}_n} H_{n, x_n}\,{\mathrm{d}}s
%-\epsilon \, \beta_1 \int_ {{\mathfrak R}_n} \Big( \frac{\bar{x}_n}{\beta}-f_n \Big)  F(\bar{H}_n)  H_{n, x_n}\,{\mathrm{d}}s
%\\
%&\quad- \epsilon  k\beta \sum_{j \neq n}^N   \int_ {{\mathfrak R}_n} \bar{H}_{j, \bar{x}_j} H_{n, x_n}\,{\mathrm{d}}s
%-\epsilon \, \beta_1  \sum_{j\neq n}^N \int_ {{\mathfrak R}_n} \Big( \frac{\bar{x}_j}{\beta}-f_j \Big)  F(\bar{H}_j)  H_{n, x_n}\,{\mathrm{d}}s\nonumber
%\\[2mm]
&=\epsilon^{2 }  |\ln \epsilon |   G_n^1(\theta)\,+\,O(\epsilon^{2}|\ln \epsilon|)  \sum_{j=1}^{N} |{\tilde f}_j(\theta)|\,+\,\epsilon^{2+\sigma}     \mathfrak{Z}_n^1(  {\tilde{\bf f}} )(\theta)\nonumber.
 \end{align}

\medskip
Recalling the definition of $\epsilon^2  |\ln \epsilon|\chi_j\Xi_{6,j}$ as in \eqref{Xi6j}, it is easy to check that
\begin{equation}
\begin{split}
 \sum_{j=1}^N &  \int_ {{\mathfrak R}_n}  \epsilon^2  |\ln \epsilon| \chi_j \Xi_{6,j}    H_{n, x_n}\,{\mathrm{d}}s
 \\
 =\,&- \,6 \,\sqrt 2\,\epsilon^2\,\beta^2 \,  \sum_{j=1}^N  {\mathbf k}_j^2  e^{-  2\sqrt 2 \beta  ( {\tilde f}_{j}- {\tilde f}_{j-1})} \int_ {{\mathfrak R}_n}      { \chi_j}  H_{j,x_j} \, e^{- 2\sqrt 2 x_j }  H_{n, x_n}\,{\mathrm{d}}s
\\[1mm]
&\,+ \,6\, \sqrt 2\,\epsilon^2\,\beta^2 \,  \sum_{j=1}^N  {\mathbf k}_{j+1}^2 e^{-  2\sqrt 2 \beta ( {\tilde f}_{j+1}- {\tilde f}_j)} \int_ {{\mathfrak R}_n} { \chi_j}  H_{j,x_j} \,  e^{ 2\sqrt 2  x_j}     H_{n, x_n}\,{\mathrm{d}}s
\\[1mm]
=\,& \epsilon^{2 }  |\ln \epsilon |   G_n^2(\theta)+  O(\epsilon^2  |\ln \epsilon |   ) \sum_{j=1}^N |{\tilde f}_j|.
\end{split}
\end{equation}

According to the definitions of $\epsilon^2  |\ln \epsilon|\Xi_{7,j},$ $\epsilon^2  \Xi_{8,j},$ $\epsilon^2  |\ln \epsilon|\Xi_{9,j} $, $\epsilon^2  |\ln \epsilon|\Xi_{10,j} $ as in \eqref{Xi7j}, \eqref{Xi8jnew}, \eqref{Xi8j}, \eqref{Xi9j}, it is easy to check that
\begin{align}\label{integralofxi6+8}
 &\sum_{j=1}^N   \int_ {{\mathfrak R}_n} \Big\{\epsilon^2  |\ln \epsilon| \Xi_{7,j}    H_{n, x_n}+\epsilon^2  \Xi_{8,j}\,+\,\epsilon^2  |\ln \epsilon|\Xi_{9,j} \,+\,\epsilon^2  |\ln \epsilon|\Xi_{10,j}\Big\}   H_{n, x_n}\,{\mathrm{d}}s \nonumber
 \\
 &= 2\epsilon^2  \beta \sum_{j=1}^N   \int_ {{\mathfrak R}_n}  \Big[     	\big( \dot f_j'  +\bar f '_j\big)  \psi^{*}_{j,zx_j} (x_j, z)
 - \big( \dot f_j'  +\bar f '_j\big) \psi^{**}_{j,zx_j} (x_j, z) \Big] H_{n, x_n}\,{\mathrm{d}}s\nonumber
 \\
 &\quad+ 2\epsilon^2 \,     \sum_{j=1}^N   \int_ {{\mathfrak R}_n}       \Big\{ \frac{\beta'}{\beta} x_j \Big(  \psi^{*}_{j,zx_j} (x_j, z)
+ \psi^{**}_{j,zx_j} (x_j, z)  \Big)     \nonumber
\\[1mm]
& \quad \quad \qquad \qquad \qquad
-     \Big( \psi^{*}_{j,\tau_j z}(\tau_j, z)  + \psi^{**}_{j,\tau_j z}(\tau_j, z)  \Big)\Big[ \frac{\beta' }{\beta}x_j+ \beta f'_j+ 2\beta' f_j\Big]   \Big\} H_{n, x_n}\,{\mathrm{d}}s
\\
&\quad+\epsilon^2  \beta_1 \sum_{j=1}^N   \int_ {{\mathfrak R}_n}      \Big(\frac{x_j}{\beta} + {\dot f}_j +{\bar f}_j \Big)  \,   (1 -3H_j^2)    (\beta_1 {\dot f}_j  \beta^{-2} \psi_j+\psi_j^*+ \psi^{**}_j )  H_{n, x_n}\,{\mathrm{d}}s \nonumber
\\
&\quad+\epsilon^2\sum_{j=1}^N   \int_ {{\mathfrak R}_n} \Big\{-\beta^{-1}  k  \beta_1  \dot{f}_j  \psi_{j, x_j}(x_j ) -\beta k    \Big[\psi^{*}_{j, x_j}(x_j, z) + \psi^{**}_{j, x_j}(x_j, z)  \Big] \Big\}H_{n, x_n}\,{\mathrm{d}}s \nonumber
\\[2mm]
 &=\epsilon^{2 }  |\ln \epsilon |   G_n^3(\theta)\,+\, \epsilon^{2+\sigma}     \mathfrak{Z}_n^2(  {\tilde{\bf f}} )(\theta).\nonumber
\end{align}

\medskip
Recalling  the expression of $\epsilon^2 |\ln \epsilon | \chi_j\Xi_{12,j}$ as in \eqref{Xi12j}, direct computation leads to
\begin{align}\label{integralofxi12}
&\sum_{j=1}^N   \int_ {{\mathfrak R}_n}  \epsilon^2 |\ln \epsilon | \chi_j\Xi_{12,j}  H_{n, x_n}\,{\mathrm{d}}s \nonumber
  \\
 &=  6 \sqrt 2 \epsilon^2\beta^2\sum_{j=1}^N \int_ {{\mathfrak R}_n} \chi_j H_{j,x_j } \Big[-{\mathbf k}_{j}{\mathbf k}_{j-1}e^{- \sqrt 2  x_j   }
 +    {\mathbf k}_{j+1}{\mathbf k}_{j+2}e^{ \sqrt 2  x_j}\Big] H_{n, x_n}\,{\mathrm{d}}s \nonumber
\\[1mm]
&\quad-6 \sqrt 2\epsilon^2 \beta_1\beta^2 \sum_{j=1}^N \int_ {{\mathfrak R}_n}{\chi_j}  \Big( \frac{x_j }{\beta} + \dot{ f}_j+ \bar{f}_j\Big) H_{j,x_j }\nonumber
\\
&\qquad \qquad \qquad \qquad \qquad \qquad \quad\times \Big[-e^{- \sqrt 2  x_j   }{\mathbf k}_{j}
+e^{ \sqrt 2  x_j}{\mathbf k}_{j+1}   \Big]H_{n, x_n}\,{\mathrm{d}}s
 \\[1mm]
 &\quad-4 \epsilon^2\beta^2 \sum_{j=1}^N \int_ {{\mathfrak R}_n} (-1)^{j}  {\mathbf A}_{5j}  \chi_j{\mathbf k}_{j}   {\mathbf k}_{j+1}   \, H_{n, x_n}\,{\mathrm{d}}s \nonumber
 \\[1mm]
 &= \epsilon^2 |\ln \epsilon|G_n^4(\theta)\,+\, \epsilon^{2+\sigma}     \mathfrak{Z}_n^3(  {\tilde{\bf f}} )(\theta).\nonumber
\end{align}

\medskip
For $j=1,\cdots, N$,  $\dot{f}_j$ (c.f.\eqref{f01}-\eqref{f2N}) are of order $|\ln \epsilon|$. According to the expression of $\epsilon^2  |\ln \epsilon|  \Xi_{13,j}$ as in \eqref{Xi14j}, direct computation leads to
\begin{align}\label{integralofxi13}
&\sum_{j=1}^N   \int_ {{\mathfrak R}_n}  \epsilon^2  |\ln \epsilon|  \Xi_{13,j}  H_{n, x_n}\,{\mathrm{d}}s \nonumber
  \\
 &= \, \epsilon^2\sum_{j=1}^N  \Big[  -2 \beta'(\dot{f}'_j + \bar{f}'_j + \check{f}'_j    )- \beta  ( \dot{f}''_j + \bar{f}''_j + \check{f}''_j    )\Big ] \int_ {{\mathfrak R}_n} H_{j,x_j}  H_{n, x_n}\,{\mathrm{d}}s \nonumber
 \\
 &\quad-2\epsilon^2  \sum_{j=1}^N \beta'(   \dot{f}'_j+ \bar{f}'_j + \check{f}'_j  )   \int_ {{\mathfrak R}_n}x_jH_{j,x_jx_j} H_{n, x_n}\,{\mathrm{d}}s
 \\
 &\quad-\epsilon^2  k^2  \beta  \sum_{j=1}^N  (  \dot{f}_j+ \bar{f}_j + \check{f}_j    ) \int_ {{\mathfrak R}_n}    H_{j,x_j}  H_{n, x_n}\,{\mathrm{d}}s \nonumber
 \\
 &\quad
+ \epsilon^2\beta^{-1}  \beta_2 ( \dot{f}_j + \bar{f}_j + \check{f}_j   ) \sum_{j=1}^N   \int_ {{\mathfrak R}_n}  x_j F(H_j)H_{n, x_n}\,{\mathrm{d}}s\nonumber
 \\[2mm]
 &=\,\epsilon^{2 }  |\ln \epsilon |   G_n^5(\theta)\,+ \,\epsilon^{2+\sigma}     \mathfrak{Z}_n^4(  {\tilde{\bf f}} )(\theta)\nonumber.
\end{align}

\medskip
From the expression of $\epsilon^2  |\ln \epsilon|  \Xi_{15,j}$ in \eqref{Xi15}, it can easily be verified that
\begin{align}\label{integralofxi15}
&\sum_{j=1}^N   \int_ {{\mathfrak R}_n}  \epsilon^2  |\ln \epsilon|  \Xi_{15,j}  H_{n, x_n}\,{\mathrm{d}}s \nonumber
  \\
 &= 12\epsilon^2  \beta^2\sum_{j=1}^{N} \int_ {{\mathfrak R}_n} H (  x_j ) \Big[ -{\mathbf k}_j  e^{-\sqrt 2 x_j }
  +   {\mathbf k}_{j+1}    e^{\sqrt 2 x_j }  \Big]  \Big[ \beta_1 \dot{f}_j  \beta^{-2} \psi_j+\psi_j^*+ \psi^{**}_j \Big]H_{n, x_n}\,{\mathrm{d}}s
 \\[2mm]
 &=\epsilon^{2 }  |\ln \epsilon |   G_n^6(\theta)\,+\, \epsilon^{2+\sigma}     \mathfrak{Z}_n^5(  {\tilde{\bf f}} )(\theta).\nonumber
\end{align}

\medskip
Adding \eqref{integralofE05}-\eqref{integralofxi15}, we can obtain that
\begin{align}
I^n_3  \, & =\, \sum_{j=1}^N\int_ {{\mathfrak R}_n}\Big\{\epsilon^2 |\ln\epsilon| \Xi_{2,j}
+ \epsilon^2 |\ln \epsilon|\chi_j \Xi_{6,j}
+  \epsilon^2 |\ln \epsilon|\,  \Xi_{7,j}+  \epsilon^2 \,  \Xi_{8,j}
+  \epsilon^2 \,  |\ln \epsilon | \, \Xi_{9,j}     \nonumber
\\[1mm] &
\qquad \qquad \qquad
+  \epsilon^2   |\ln \epsilon |    \Xi_{10,j}
+  \epsilon^2|\ln \epsilon |\chi_j\Xi_{12,j}
+ \epsilon^2   |\ln\epsilon| \Xi_{13,j}
+ \epsilon^2  |\ln \epsilon| \Xi_{15,j} \Big\}  H_{n, x_n}\,{\mathrm{d}}s,
\\[1mm]
&=   \epsilon^2   |\ln \epsilon |    \mathfrak{D}_n  (\theta)\,+ \,\epsilon^{2+\sigma} \sum_{j=1}^5    \mathfrak{Z}_n^j (  {\tilde{\bf f}} )(\theta)\,+\,O(\epsilon^{2}|\ln \epsilon|)  \sum_{j=1}^{N} |{\tilde f}_j(\theta)|,
\nonumber
\end{align}
where
\begin{equation*}
\mathfrak{D}_n  (\theta)\,:=\,      G_n^1(\theta)\,+\,G_n^2(\theta)\,+\,G_n^3(\theta)\,+G_n^4(\theta)+ \,G_n^5(\theta)\,+\,G_n^6(\theta)
\end{equation*}
is uniformly bounded function and  independent of $ {\tilde{\bf f}}$.

\medskip
 According to the expressions of $\epsilon^2|\ln\epsilon|^2E_{03}, \epsilon^2E_{06} $ in \eqref{E03}, \eqref{E06} and the fact that $H, H_{xx}$ and
 $F(H)$ are odd functions, we obtain the estimates of
\begin{align*}
I^n_4 \,+\,I^n_5
% =&\epsilon^2|\ln\epsilon|^2\sum_{j=1}^N   \int_ {{\mathfrak R}_n}   \, E_{03}   H_{n, x_n}\,{\mathrm{d}}  s
%   \\
&=  \epsilon^2 \beta^2  \sum_{j=1}^N {f'_j }^2\int_ {{\mathfrak R}_n} H_{j,x_jx_j}    H_{n, x_n}\,{\mathrm{d}}  s
+ \frac 12 \epsilon^2 \beta_2 \sum_{j=1}^N f_j^2\int_ {{\mathfrak R}_n}   F(H_j) H_{n, x_n}\,{\mathrm{d}}  s
\\
&\quad +\epsilon^2  \sum_{j=1}^N \Big(  \frac{{\beta'} }{\beta} \Big)^2\int_ {{\mathfrak R}_n} x_j^2H_{j,x_jx_j}\,H_{n, x_n}\,{\mathrm{d}}  s
+\epsilon^2  \sum_{j=1}^N \frac{\beta'' }{\beta}\int_ {{\mathfrak R}_n}x_jH_{j,x_j}\,H_{n, x_n}\,{\mathrm{d}}  s
\\
&\quad+  \epsilon^2\frac{\beta_2 }{ 2 \beta^{2} } \sum_{j=1}^N \int_ {{\mathfrak R}_n} x_j^2 F(H_j)H_{n, x_n}\,{\mathrm{d}}  s
-  \epsilon^2  k^2   \sum_{j=1}^N \int_ {{\mathfrak R}_n}x_j H_{j,x_j}H_{n, x_n}\,{\mathrm{d}}  s
\\[2mm]
& =   \epsilon^{2+\sigma}    \mathfrak{Z}_n^6(  {\tilde{\bf f}} )(\theta).
\end{align*}

Using the definition of $\bar{\mathcal{K}}$ as in \eqref{mathcalK}, we can obtain that
\begin{equation}\label{I7}
I_n^{7}\,=\,\int_ {{\mathfrak R}_n}   \bar{\mathcal{K}}  H_{n, x_n}\,{\mathrm{d}}s
\,=\,      \epsilon^2       G_n^7  (\theta)\, + \epsilon^{2+\sigma}    \mathfrak{Z}_n^7(  {\tilde{\bf f}} )(\theta).
\end{equation}
In fact, when $s \in {\mathfrak R}_n$, it is easy to obtain that
\begin{equation*}
\tau_j=-\beta(s+f_j)\in \Big(-\frac{\beta f_{n+1}+ \beta f_n}{2} - \beta f_j,\, -\frac{\beta f_{n-1}+\beta f_n}{2}- \beta f_j\Big).
\end{equation*}
Recall the  asymptotic estimates of $\psi_j, \psi^*_j$ in \eqref{estimatepsi1}, \eqref{estimatepsi2},  and  $\psi^{**}_j, \omega^{*}_j$ have similar asymptotic estimates of  $ \psi^*_j $.  Then the proof of \eqref{I7} is very straightforward. In the following, the term  $ \epsilon^2 G_n^7  (\theta)$  will be  absorbed  in  $\epsilon^2|\ln \epsilon | \mathfrak{D}_n  (\theta)$.

Next, recalling the expressions of $\mathcal{Q}_1 $ as in \eqref{mathcalQ1}, we can get
\begin{equation*}
I_n^{8}\,=\,\int_ {{\mathfrak R}_n}   \mathcal{Q}_1  H_{n, x_n}\,{\mathrm{d}}s
\,=\,\epsilon^{2+\sigma}    \mathfrak{Z}_n^8(  {\tilde{\bf f}} )(\theta).
\end{equation*}

\medskip
Recall that for $ s \in    \mathbb R^{+}   \setminus {\mathfrak R}_n$, we have
 \begin{align*}
 H_{n,x_n} = (-1)^n H' \big( \beta(s-f_n )  \big) \, =\,      \max_{j\neq n}O( e^{ -\frac{ \sqrt 2 }{2}\, \beta \,  | f_j-f_n|}) \, =\, O(\epsilon^\frac{1}{2}).
 \end{align*}
 Combining  above fact and  the assumption  \eqref{ddotfnorm}, we can get easily
 \begin{align*}
 \mathbb{E}_{n2} \, =\, O(\epsilon^{\frac{1}{2}}) \sum\limits_{i=0}^{8} I_{i}^{n}.
  \end{align*}
This finishes the estimates of the first term in \eqref{equivalancec1}.

\medskip
Finally, we turn to the estimates of the other two terms in \eqref{equivalancec1}. Denote
 \begin{align*}
 {\mathcal W}_n(\epsilon z) \, =\,     -   \int_0^{+\infty }  \,  \eta_{\delta}^{\epsilon}(s) {\mathcal N}(\tilde \phi)      H_{n, x_n}   {\mathrm{d}}s.
 \end{align*}
 Combining  the estimate of $\tilde \phi$ in  \eqref{tilde phi 1} and \eqref{tilde phi 2},  we can conclude that
 \begin{equation*}
\|{\mathcal W}_n   \|_{L^2(0, \ell)}    \, \leq \, C \epsilon^{   2+ \sigma }.
\end{equation*}
The last term in \eqref{equivalancec1} can be  rewritten as
 \begin{align}
 {{\mathfrak{V}}}_n(\epsilon z) \, &=\,    \int_0^{+\infty }   \,  {\mathcal L}(\tilde \phi) \, H_{n, x_n}  \,  {\mathrm{d}}s\nonumber
\\[1mm]
\, & = \,        \int_0^{+\infty }     {\tilde \phi}_{zz}
 \,  H_{n, x_n} \,  {\mathrm{d}}s
 +3   \beta^2       \int_0^{+\infty }    {\tilde \phi} \Big[    H^2\big(  \beta(s-f_n) \big)- \mathbf{H}^2 \,  \Big]  H_{n, x_n} \,     {\mathrm{d}}s
\\[1mm]
\, & \quad    +    \int_0^{+\infty }   \, \eta^{\epsilon}_{10 \delta}(s) \Big[ \, B_1(\tilde \phi) +  B_6(\tilde \phi) \Big]    H_{n, x_n}  \,  {\mathrm{d}}s, \nonumber
 \end{align}
A similar estimate holds
\begin{align}
\|{{\mathfrak{V}}}_n   \|_{L^2(0, \ell)} \, \leq \, C \epsilon^{ 2+ \sigma}.
\end{align}

\medskip
By the notation of
\begin{equation*}
\begin{split}
\mathbb {M}_n (\theta ,{\tilde{\bf f}},{\tilde{\bf f}}',{\tilde{\bf f}}'')
\, =\, & \epsilon^2 \gamma_0 \beta' \beta^{-1}{\tilde f}_n'
-\epsilon^2 \gamma_0 \Big[\frac{1}{2}\beta_2\beta^{-1}-k^2 \beta \Big]\beta^{-1}{\tilde f}_n
\\& \quad
-  \epsilon^{2+\sigma}  \mathfrak{Z}_n^0 (  {\tilde{\bf f}},  {\tilde{\bf f}}',  {\tilde{\bf f}}''  )(\theta)
- \epsilon^{ 2+ \sigma} \sum_{i=1}^8{{\mathfrak{Z}}}_n^i ({\tilde{\bf f}})(\theta )
\\& \quad
-   {{\mathfrak{V}}}_n(\theta ) -{\mathcal W}_n(\theta )+O(\epsilon^2 |\ln \epsilon|) \sum\limits_{j=1}^{N} |{\tilde f}_j|,
\end{split}
\end{equation*}
we can get that the relations in \eqref{equivalancec1} are equivalent to the following system of differential equations, for $n=1,\cdots, N$,
\begin{equation*}
\begin{split}
-  \epsilon^2   \gamma_0 \,{\tilde f}''_n   &  +    6 \sqrt 2\,  \epsilon   \, \beta\,
 \gamma_{1,n}\,   {\mathbf k}_{n}    \Big[\,e^{\,-  \sqrt 2 \,\beta  \,({\tilde f}_{n}-{\tilde f}_{n-1})\,}-1 \, \Big]
\\[2mm]
&
 - 6 \sqrt 2\,  \epsilon  \,  \beta  \, \gamma_{2,n}\,     {\mathbf k}_{n+1} \,\Big[  e^{\,-  \sqrt 2 \,\beta  \,({\tilde f}_{n+1}-{\tilde f}_n)}  -1 \Big] =    \,-\,
\epsilon^2|\ln \epsilon |    \mathfrak{D}_n  +    \mathbb {M}_n (\theta, {\tilde{\bf f}},{\tilde{\bf f}}',{\tilde{\bf f}}'').
\end{split}
\end{equation*}
Therefore, we can draw a conclusion in the following proposition

\begin{proposition}\label{proposition 5.1}
For the validity of (\ref{equivalancec}),  there should hold the following equations, for $n=1, \cdots, N$,
\begin{equation}\label{ts0a}
\begin{split}
- \epsilon^2   \gamma_0 {\tilde f}''_n
&=\,   -\, 6 \sqrt 2\,  \epsilon   \, \beta \,
 \gamma_{1,n}\,   {\mathbf k}_{n}    \Big[e^{\,-  \sqrt 2 \beta  ({\tilde f}_{n}-{\tilde f}_{n-1})\,}-1  \Big]
\\[2mm]
&\quad
\,+\, 6 \sqrt 2\,  \epsilon  \beta \gamma_{2,n}  {\mathbf k}_{n+1} \,\Big[  e^{\,-  \sqrt 2 \beta ({\tilde f}_{n+1}-{\tilde f}_n)\,}  -1 \Big]
-\epsilon^2|\ln \epsilon |    \mathfrak{D}_n
+ \mathbb {M}_n (\theta, {\tilde{\bf f}},{\tilde{\bf f}}',{\tilde{\bf f}}'').
\end{split}
\end{equation}
Moreover, for any $n=1, \cdots, N$,    the operator $\mathbb {M}_n (\theta, {\tilde{\bf f}},{\tilde{\bf f}}',{\tilde{\bf f}}'')  $  can be  decomposed in the following way
\begin{equation*}
 \mathbb {M}_n (\theta,{\tilde{\bf f}},{\tilde{\bf f}}',{\tilde{\bf f}}'')
 =\mathbb {M}_{n1} (\theta, {\tilde{\bf f}},{\tilde{\bf f}}')
 \,+\,\mathbb {M}_{n2} (\theta, {\tilde{\bf f}},{\tilde{\bf f}}',{\tilde{\bf f}}'')
\end{equation*}
where $\mathbb {M}_{n1} (\theta, {\tilde{\bf f}},{\tilde{\bf f}}')$ and $\mathbb {M}_{n2} (\theta, {\tilde{\bf f}},{\tilde{\bf f}}',{\tilde{\bf f}}'')$ are continuous of their arguments. Function $\mathbb {M}_{n1} (\theta, {\tilde{\bf f}},{\tilde{\bf f}}')$ satisfies
\begin{align}\label{highorder1}
     \|  \mathbb {M}_{n1} (\theta, {\tilde{\bf f}},{\tilde{\bf f}}') \|_{L^2(0, \ell)} \leq C \epsilon^{ 2+\sigma}.
\end{align}
The function $\mathbb {M}_{n2} (\theta, {\tilde{\bf f}},{\tilde{\bf f}}',{\tilde{\bf f}}'')$ has the following properties
\begin{align}\label{highorder3}
     \|  \mathbb {M}_{n2} (\theta, {\tilde{\bf f}},{\tilde{\bf f}}',{\tilde{\bf f}}'')  \|_{L^2(0, \ell)} \leq C \epsilon^{ 2+\sigma}.
\end{align}
\begin{align}\label{highorder2}
\big\|
{\mathbb M}_{n2}(\theta, {\tilde{\bf f}}^{(1)}, {\tilde{\bf f}}^{(1)}{'}, {\tilde{\bf f}}^{(1)}{''})
\,-\,
{\mathbb M}_{n2}(\theta, {\tilde{\bf f}}^{(2)}, {\tilde{\bf f}}^{(2)}{'}, {\tilde{\bf f}}^{(2)}{''})
\big\|_{L^2(0, \ell)}
\,\leq\,&
 C\epsilon ^{1+\sigma}\, \|{\tilde{\bf f}}^{(1)}\,-\,{\tilde{\bf f}}^{(2)}\|.
\end{align}
\qed
\end{proposition}
We omit the details of the proof of this proposition. The reader can refer to \cite{delPKowWei2} and \cite{delPKowWei3}

\medskip
 \section{Solve the system of reduced equations}\label{section7}

Note that the parameters ${\tilde f}_1, \cdots, {\tilde f}_N$ will determine the locations of phase transition layers and play an important role in the description of the interaction between neighboring layers in the clustering phenomena. As a consequence of Proposition \ref{proposition 5.1}, 
to find ${\tilde{\bf f}}=\big({\tilde f}_1, \cdots, {\tilde f}_N\big)$, we have to deal with the following system, for $n=1,\cdots, N$,
\begin{equation}\label{ts2a}
\begin{split}
&  \,   -  \epsilon\gamma_0  {\tilde f}''_n
 -  12  \,     \beta^2 \,
 \gamma_{1,n}\,   {\mathbf k}_{n}   ({\tilde f}_{n}-{\tilde f}_{n-1})
  +  12   \,  \beta^2  \, \gamma_{2,n}\,     {\mathbf k}_{n+1}  ({\tilde f}_{n+1}-{\tilde f}_n)
 \\[2mm]
 & \, =  {\mathcal J}_n({\tilde{\bf f}})        -   \epsilon|\ln \epsilon |    \mathfrak{D}_n
- \epsilon^{-1} \mathbb {M}_n (\theta, {\tilde{\bf f}},{\tilde{\bf f}}',{\tilde{\bf f}}''),
\end{split}
\end{equation}
with boundary conditions
\begin{equation}
{\tilde{\bf f}}(0) \,=\, {\tilde{\bf f}}(\ell),
\quad
{\tilde{\bf f}}'(0) \, =\, {\tilde{\bf f}}'(\ell ).
\end{equation}
Here, the nonlinear terms ${\mathcal J}_n       $  for $ n= 1, \cdots,  N $ are given by
   \begin{align}\label{6.3}
   {\mathcal J}_n({\tilde{\bf f}})
   \,& =\, - 6 \sqrt 2  \, \beta \,
 \gamma_{1,n}\,   {\mathbf k}_{n}    \Big[e^{-  \sqrt 2 \,\beta  \,({\tilde f}_{n}-{\tilde f}_{n-1})}-1 + \sqrt 2 \beta  \,({\tilde f}_{n}-{\tilde f}_{n-1})\Big]
\nonumber
\\[2mm]
\, & \quad +   6 \sqrt 2 \,  \beta  \, \gamma_{2,n}\,     {\mathbf k}_{n+1}   \Big[  e^{-  \sqrt 2 \,\beta  \,({\tilde f}_{n+1}-{\tilde f}_n)}  -1 +  \sqrt 2 \,\beta  \,({\tilde f}_{n+1}-{\tilde f}_n)\Big].
   \end{align}
In the above, for $n=1,\cdots, N$, the terms $\gamma_{1,n}$, $\gamma_{2,n}$, ${\mathbf k}_{n}$ are given in
\eqref{ga1n}, \eqref{ga2n}, \eqref{mathbf kj}.

\bigskip
\noindent{\bf Step 1:}
The first try is to simplify the above system.
We will denote
\begin{align*}
{\mathfrak c}_n(\theta)  \, :=\, 12  \,     \beta^2 \,
 \gamma_{1,n}\,   {\mathbf k}_{n}(\theta),
 &\qquad \quad
{\mathfrak d}_n(\theta) \, :=\, 12   \,  \beta^2  \, \gamma_{2,n}\,     {\mathbf k}_{n+1}(\theta),
\\[2mm]
{\mathcal J}\bigl({\tilde{\bf f}}\bigr)(\theta) :=\big({\mathcal J}_1\bigl({\tilde{\bf f}}\bigr)(\theta), \cdots,  {\mathcal J}_N\bigl({\tilde{\bf f}}\bigr)(\theta) \big)^T,
&\qquad\quad
\mathfrak{D} (\theta) \, :=\, \big(\mathfrak{D}_1(\theta)  \cdots, \mathfrak{D}_N (\theta)              \big)^T,
\end{align*}
and
\begin{equation*}
\Upsilon({\tilde{\bf f}})(\theta):= \Big(-\epsilon^{-1} \mathbb {M}_1 ({\tilde{\bf f}})  (\theta ),  \cdots, -  \epsilon^{-1} \mathbb {M}_N ({\tilde{\bf f}})  (\theta )              \Big)^T.
\end{equation*}
Then system (\ref{ts2a})  becomes:
\begin{equation}
  -  \epsilon \gamma_0  {\tilde{\bf f}}''
\, =\,
{\bf B} {\tilde{\bf f}}
\,+\,{\mathcal J}\bigl({\tilde{\bf f}}\bigr)
\,-\,\epsilon  |\ln \epsilon| \mathfrak{D}
\,+\,\Upsilon({\tilde{\bf f}}),
\label{auxil 2}
\end{equation}
where   ${\bf B} $    can be expressed as
\begin{align*}%\label{definitionofbfhatB}
{\bf B}\,=\,
\left(
\begin{array}{ccccccccc}
    2  {\mathfrak c}_1+{\mathfrak d}_1   &-  \mathfrak{d}_1&0&  0&\cdots &0&0&0&0
\\
  -{\mathfrak c}_2& ({\mathfrak c}_2+\mathfrak{d}_2)&-  \mathfrak{d}_2    & 0&\cdots &0&0&0&0
\\
\vdots&\vdots &\vdots &\vdots &\ddots&\vdots &\vdots & \vdots&\vdots
\\
0 & 0&0&0&\cdots&  0&     -   {\mathfrak c}_{N-1}     & (  {\mathfrak c}_{N-1} +  {\mathfrak d}_{N-1})&  - \mathfrak{d}_{N-1}
\\
0 & 0&0&0&\cdots&0&  0& - \mathfrak{c}_{N}&  \mathfrak{c}_{N}
\end{array}
\right).
\end{align*}

\medskip
We have the relations,   for $n=1,\cdots, N$,
\begin{equation*}
  \mathfrak{c}_{n}
\,=\,
{\mathfrak a}_{n-1}+O(\epsilon^{\frac 12} ),
\qquad \quad
  \mathfrak{d}_{n-1}
\,=\,
{\mathfrak a}_{n-1}+O(\epsilon^{\frac 12} ),
\end{equation*}
where $ {\mathfrak a}_{n-1}$ can be expressed as
\begin{equation*}
{\mathfrak a}_{n-1}(\theta) \, =\,  12   \beta^2 (\theta) \gamma_1 \, {\mathbf k}_n(\theta).
\end{equation*}
For $n=1,\cdots,N$, recalling the expressions of  ${\mathbf d}_n$ and  ${\mathbf k}_n$ in \eqref{mathbfdn}, \eqref{mathbf kj},
we get
\begin{equation}\label{mathfrak a}
\begin{split}
{\mathfrak a}_{n-1}
\, &=\, 12   \beta^2(\theta)  \gamma_1 \, {\mathbf k}_n(\theta)
 \,=\,  12   \beta^2(\theta)  \gamma_1 \,
 {\mathbf d}_n e^{-\sqrt 2 \, \beta (\check{f}_{n}(\theta)-\check{f}_{n-1}(\theta)) }
 \\[2mm]
 \, &=\,        12   \beta^2(\theta)  \gamma_1 \,  \frac{(N-n+1)}{9\beta^2(\theta)  \gamma_1}\Big({\mathcal H}+O(\epsilon^{\frac{1}{2}}) \Big)  e^{-\sqrt 2 \, \beta (\check{f}_{n}(\theta)-\check{f}_{n-1}(\theta)) }>0,\quad \forall\,\theta\in [0, \ell],
\end{split}
\end{equation}
where ${\mathcal H}$ is given in \eqref{meancurvaturepositive}.
It is obvious that ${\mathbf B}$ is a perturbation of a symmetric matrix ${\bf A}$ defined in the form
\begin{align*}
{\bf A}\,=\,  \left(
\begin{array}{ccccccccc}
   2  {\mathfrak a}_0 +{\mathfrak a}_1  & -  {\mathfrak a}_1&0&  0&\cdots &0&0&0&0
\\
- {\mathfrak a}_1 & ({\mathfrak a}_1+{\mathfrak a}_2)&- {\mathfrak a}_2& 0&\cdots &0&0&0&0
\\
\vdots&\vdots & \vdots& \vdots&\ddots& &\vdots &\vdots &\vdots
\\
0 & 0&0&0&\cdots&  0&        & ({\mathfrak a}_{N-2}+{\mathfrak a}_{N-1})&- {\mathfrak a}_{N-1}
\\
0 & 0&0&0&\cdots&0&  0&-  {\mathfrak a}_{N-1}& {\mathfrak a}_{N-1}
\end{array}
\right).
\end{align*}
Using elementary matrix operations, it is easy to prove that
there exists an invertible matrix ${\bf Q}$ independent of $\theta$ such that
$$
{\bf Q}{\bf A}{\bf Q}^T=\text{diag}(     2 {\mathfrak a}_0,       {\mathfrak a}_1, {\mathfrak a}_2, \cdots, {\mathfrak a}_{N-1}).
$$
Since ${\mathfrak a}_0(\theta), {\mathfrak a}_1(\theta), \cdots, {\mathfrak a}_{N-1}(\theta)$ are positive functions defined in (\ref{mathfrak a}),
we can assume that all eigenvalues of the matrix ${\bf A}$ are
\begin{equation*}
\rho_1(\theta)\geq\cdots\geq \rho_{N-1}(\theta) \geq \rho_N(\theta)>0.
\end{equation*}
%Moreover, we can show that
%\begin{align}
% \rho_n(\theta) \, = \, \Big[  k - \frac{1}{2\beta^2} \beta_1  +O(\epsilon^{\frac{1}{2}})  \Big] O(1),  \quad \forall\,  n \, =\, 1, \cdots, N.
%\end{align}

\medskip
Similarly, it is easy to prove that
there exists an invertible matrix $\hat{\bf Q}$ such that
\begin{equation*}
\hat{\bf Q}          {\bf B}   {\hat{\bf Q}}^T
\,=\, \mbox{diag}\Bigl(\, 2{\mathfrak a}_0+O(\epsilon^{\frac 12}), {\mathfrak a}_{1}+O(\epsilon^{\frac 12}), \, \cdots,   \,   {\mathfrak a}_{N-1}+  O(\epsilon^{\frac 12}) \Bigr ).
\end{equation*}
We also assume that all eigenvalues of the matrix  $        {\bf B} $  are
$ \rho^{\epsilon}_1(\theta), \cdots, \rho^{\epsilon}_N (\theta)$.
Naturally,   we have
\[
\rho^{\epsilon}_n(\theta)\,=\, \rho_n(\theta)+O (     \epsilon^{\frac 12} ), \quad\forall\, n=1, \cdots, N.
\]
Moreover,  there exists another invertible matrix ${\bf P}$
in the form
\begin{equation}\label{P}
{\bf P}^T{\bf P}={\bf I},
\qquad
{\bf P}=\left(
\begin{array}{cccc}
p_{11} & \cdots & p_{1N-1} & p_{1N}
\\
p_{21} & \cdots & p_{2N-1} & p_{2N}
\\
\vdots  & \vdots  & \vdots  & \vdots
\\
p_{N1} & \cdots & p_{NN-1} & p_{NN}
\\
\end{array}
\right),
\end{equation}
such that
\begin{equation*}
{\bf P}^T   {\bf B} {\bf P}   \, =\,  \mbox{diag}(\rho^{\epsilon}_1, \cdots, \rho^{\epsilon}_N).
\end{equation*}

\medskip
Now, we define new vectors
\begin{equation*}
\mathfrak{u}:=(\mathfrak{u}_1, \cdots, \mathfrak{u}_{N})^T={\bf P}^T{\tilde{\bf f}},
\qquad
\tilde {\mathfrak{D} }  :=\big(\tilde {\mathfrak{D} } _1, \cdots, \tilde {\mathfrak{D} }_N \big)^T \, = \,   {\bf P}^T {\mathfrak{D} },
\end{equation*}
and
\begin{align*}
\tilde{\mathcal J}(\mathfrak{u} ):=&\big( \tilde{\mathcal J}_1(\mathfrak{u} ),\cdots, \tilde{\mathcal J}_N(\mathfrak{u} ) \big)^T
 \, = \,   {\bf P}^T   {\mathcal J}({\tilde{\bf f}}),
\\[3mm]
{\bf W}(\mathfrak{u})\,:=&\big({\bf W}_1(\mathfrak{u}), \cdots, {\bf W}_N(\mathfrak{u})\big)^T
\,=\,
{\bf P}^T {\bf P}'' {\mathfrak u}
\,+\,
{\bf P}^T  {\bf P}'  {\mathfrak u}',
\\[3mm]
{\tilde\Upsilon}(\mathfrak{u}):
=&\big({\tilde \Upsilon}_1(\mathfrak{u}), \cdots, {\tilde \Upsilon}_N(\mathfrak{u})\big)^T
= {\bf P}^T \Upsilon({\tilde{\bf f}})
={\bf P}^T \Upsilon({\bf P}\mathfrak{u}).
\end{align*}
Then the system \eqref{auxil 2}  can be rewritten as
\begin{equation}\label{ts6a}
- \epsilon \gamma_0 \mathfrak{u}''
\,- \,\mbox{diag}(\rho^{\epsilon}_1, \cdots, \rho^{\epsilon}_N)  \mathfrak{u}
\,=\,
{\tilde{\mathcal J}}\bigl(\mathfrak{u} \bigr)
\,-\,\epsilon  |\ln \epsilon |   \tilde{\mathfrak{D} }
\,+\,\epsilon   \gamma_0    {\bf W}(\mathfrak{u})
\,+\,{\tilde\Upsilon} (\mathfrak{u}).
\end{equation}

\medskip
In order to cancel the term $-   \epsilon  |\ln \epsilon |    \tilde{\mathfrak{D} } $, we further set,
$$
\mathfrak{u}_n  (\theta)  =    \hat{ \mathfrak{u}}_n    (\theta)    \,+\,  \frac{ \epsilon |\ln \epsilon |  \tilde{\mathfrak{D} }_n(\theta)} {\rho^{\epsilon}_n(\theta) }
,\quad
\forall\, n =1, \cdots, N.
$$
Then   \eqref{ts6a}  becomes a system of $ \hat {\mathfrak{u} }   (\theta) = \big(\hat {\mathfrak{u} }_1  (\theta), \cdots, \hat {\mathfrak{u} }_N  (\theta)      \big)^T $,
\begin{align}
 -\epsilon\gamma_0  \hat{\mathfrak{u}}''
 \,-\,\mbox{diag}(\rho^{\epsilon}_1, \cdots, \rho^{\epsilon}_N) \hat{ \mathfrak{u}}
 \,=\, &\, {\tilde{\mathcal J}}\bigl( \hat {\mathfrak{u} }
 \,+\,\epsilon  |\ln \epsilon| \mathfrak{C}  \bigr)
 \,+\,\epsilon\gamma_0{\bf W}\big(\hat {\mathfrak{u} } + \epsilon  |\ln \epsilon| \mathfrak{C}  \big)
\nonumber\\[2mm]
\quad&
\,+\,\epsilon^2 |\ln \epsilon| \gamma_0  \mathfrak{C}''
\,+\,{\tilde\Upsilon} \big(\hat {\mathfrak{u} } + \epsilon  |\ln \epsilon| \mathfrak{C} \big),
\label{ts10a}
\end{align}
where we have denoted
\begin{equation}\label{definitionofmathfrakC}
\mathfrak{C} (\theta)=\big(\mathfrak{C}_1 (\theta), \cdots, \mathfrak{C}_N (\theta)  \big)^T = \Big(\frac{  \tilde{\mathfrak{D} }_1(\theta)} {\rho^{\epsilon}_1(\theta) }, \cdots, \frac{   \tilde{\mathfrak{D} }_N(\theta)} {\rho^{\epsilon}_N(\theta) }\Big)^T.	
\end{equation}
It is easy to check that $\hat {\mathfrak{u} }$ satisfies the boundary conditions
\begin{equation}\label{boundarymathfraku}
\hat {\mathfrak{u} }(0)=\hat {\mathfrak{u} }(\ell),
\qquad
\hat {\mathfrak{u} }'(0)=\hat {\mathfrak{u} }'(\ell).
\end{equation}

\medskip
\noindent{\bf Step 2:}
For the solvability of \eqref{ts10a}, we pause here to consider the linear operators
\begin{equation*}{\mathbb L}^\epsilon_n v_n \, :=\,
  -    \epsilon \,  \gamma_0\,    v_n''
 \,- \,\rho^{\epsilon}_n (\theta) \, v_n,
 \quad \forall\, n=1, \cdots, N,
\end{equation*}
and then provide the resolution theory for the linear problems.

\begin{lemma}\label{lemma6point3}
We consider the following system,
 for $n=1, \cdots, N$,
\begin{equation}\label{equationofvnn}
{\mathbb L}^\epsilon_n(v_n)
\,=\, h_n ,
\qquad
v_n(0) \,=\, v_n(\ell), \quad
{v}_n'(0) \, =\, {v}_n'(\ell ).
\end{equation}
There exists a sequence $\{\epsilon_l: l\in{\mathbb N}\}$ such that
there exists a unique solution $ {\bf v} \, = \,   (v_1, \cdots, v_N)^{T}$ to the system (\ref{equationofvnn})
with estimates, for all $n \, =\, 1, \cdots, N, $
\begin{equation}
\epsilon_l\, \|  v_n ''\|_{L^2(0, \ell)}
\,+\,
\sqrt{\epsilon_l}\, \|  v_n '\|_{L^2(0, \ell)}
\,+\,
\|    {v}_n \|_{L^{\infty}(0, \ell)}
\,\leq\,
C\,\sqrt{\frac{1}{\epsilon_l}}\,  {\|h_n \|_{L^2(0, \ell)}}.
\label{l2estimatev}
\end{equation}
Moreover,  if $h_n \in H^2(0, \ell )$ then there hold
\begin{equation}
\epsilon_l\, \|  v_n''\|_{L^2(0, \ell)}
\,+\,\|  v_n'\|_{L^2(0, \ell)}
\,+\,\|{v_n}\|_{L^{\infty}(0, \ell)}
\,\leq\, C\, \|h_n\|_{H^2(0, \ell)}.
\label{h2estimate1vN}
\end{equation}
 \end{lemma}

\begin{proof}
The proof is similar as the one for Lemma 8.1  in \cite{delPKowWei2}.
\end{proof}

%Using the Lemma \eqref{lemma6point3}, we will solve \eqref{ts10a} -\eqref  {boundarymathfraku}.
%%By Lemma \ref{lemma6point3}, we can conclude, for $ n =1, \cdots, N$,
%%\begin{align}
%%\epsilon_l\, \| { \hat{\mathfrak{u}}_n  }''\|_{L^2(0, \ell)}
%%\,+\,
%%\sqrt{\epsilon_l}\, \| {  \hat{\mathfrak{u}}_n  }'\|_{L^2(0, \ell)}
%%\,+\,
%%\|    \hat{\mathfrak{u}}_n  \|_{L^{\infty}(0, \ell)}
%%\,\leq\,
%%C\,{\epsilon_l}^{\frac 12}. ???
%%\end{align}
%    Replacing $\hat{{\mathfrak u}}=\check{{\mathfrak u}}+{\mathbb V}$
% will imply that problem (\ref{ts6a})-(\ref{boundaryconditionofmathfraku1}) becomes
%\begin{align}\label{ts7a}
% -     \epsilon \gamma_0   \beta      {\bf I}   \frac{\mathrm{d}^2}{\mathrm{d}\theta^2} {\mathbb V}- \, \mbox{diag}(\rho^{\epsilon}_1, \cdots, \rho^{\epsilon}_N) \, {\mathbb V}
%=\,
%{\tilde{\mathcal J}}^1\bigl(\hat{{\mathfrak u}}+{\mathbb V}  \bigr)        +    \epsilon   \gamma_0   \beta   {\bf W}(\hat{{\mathfrak u}}+{\mathbb V} )
%+{\tilde\Upsilon} ( \hat{{\mathfrak u}}+{\mathbb V}  ),
%\end{align}
%\begin{align}\label{boundary3}
%{\mathbb V}(0) \,=\, {\mathbb V}(\ell), \qquad \quad
%{\mathbb V}'(0) \, =\,  {\mathbb V}'(\ell ),
%\end{align}
% with $ {\mathbb V}= ({\mathbb V}_1, \cdots, {\mathbb V}_N     )$.

\medskip
\noindent{\bf Step 3:}
  Giving ${\bf h}=(h_1, \cdots, h_N)^T\in L^2(0, \ell)$ with $\|{\bf h}\|_{L^2(0, \ell)}\leq \epsilon^{ 1+ \sigma}$, we consider the nonlinear problem
  \begin{equation}\label{ts8a}
  \begin{split}
   - \epsilon \gamma_0      {\tilde{\mathfrak u}}''
 \,-\, \mbox{diag}(\rho^{\epsilon}_1, \cdots, \rho^{\epsilon}_N) \,  {\tilde{\mathfrak u}}
\,=\,&{\tilde{\mathcal J}}\bigl(  {\tilde{\mathfrak u}} +  \epsilon  |\ln \epsilon| \mathfrak{C}   \bigr)
\,+\,\epsilon   \gamma_0   {\bf W}\big( {\tilde{\mathfrak u}}  +  \epsilon  |\ln \epsilon| \mathfrak{C} \big)
\\[1mm]
&\,+\,\epsilon^2 |\ln \epsilon| \gamma_0  \mathfrak{C}''
\,+\, {\bf h},
  \end{split}
  \end{equation}
\begin{equation} \label{boundaryconditionofmathfraku3}
{\tilde{\mathfrak u}}(0) \,=\, {\tilde{\mathfrak u}}(\ell), \quad \quad
{\tilde{\mathfrak u}}'(0) \, =\,  {\tilde{\mathfrak u}}'(\ell ),
\end{equation}
 where $ {\tilde{\mathfrak u}}= (   {\tilde{\mathfrak u}}_1, \cdots,   {\tilde{\mathfrak u}}_N )^T$.

\medskip
Define the set
\begin{equation*}
{\mathcal X}\,=\,
\left\{\,
{\tilde{\mathfrak u}}\in H^2(0, \ell)
\,:\,
\epsilon\, \|{\tilde{\mathfrak u}}''\|_{L^2(0, \ell)}
\,+\,
\sqrt{\epsilon}\, \|{\tilde{\mathfrak u}}'\|_{L^2(0, \ell)}
\,+\,
\|{\tilde{\mathfrak u}}\|_{L^\infty (0, \ell)}
\,\leq\,     \epsilon^{ \frac 12 + \sigma}
\,\right\}.
\end{equation*}
In fact,  for any ${\tilde{\mathfrak u}}\in {\mathcal X}$, there holds
	\begin{align*}
{\tilde{\mathcal J}}\big( {\tilde{\mathfrak u}} +  \epsilon  |\ln \epsilon| \mathfrak{C} \big)
=    {\bf P}^T
\left(\begin{array}{c} {\mathcal J}_1\big({\bf P} \big(  {\tilde{\mathfrak u}} +   \epsilon  |\ln \epsilon| \mathfrak{C}  \big)\big)
\\
\vdots
\\
{\mathcal J}_N\big({\bf P}\big(  {\tilde{\mathfrak u}} +  \epsilon  |\ln \epsilon| \mathfrak{C} \big)\big)
\end{array}
\right),
\end{align*}
where $ {\mathcal J}_n$ is defined in \eqref{6.3}.
Therefore, we have
\begin{equation*}
\|	{\tilde{\mathcal J}}\big(        {\tilde{\mathfrak u}} +  \epsilon  |\ln \epsilon| \mathfrak{C} \big)\|_{L^2(0,\ell)}
\leq C\sum_{n=1}^{N}	\|{\mathcal J}_{n}\big({\bf P}\big(        {\tilde{\mathfrak u}} +  \epsilon  |\ln \epsilon| \mathfrak{C} \big)\big)\|_{L^2(0,\ell)},
\end{equation*}
and
\begin{equation}
\begin{split}
{ { \mathcal J}}_{n}\big( {\bf P}\big(   {\tilde{\mathfrak u}} +  \epsilon  |\ln \epsilon| \mathfrak{C} \big)  \big)
\, & =\,
 O\Big(\Bigl[ \big( {\bf P} ( {\tilde{\mathfrak u}} +  \epsilon  |\ln \epsilon| \mathfrak{C} )\big)_n
- ( {\bf P}\big( {\tilde{\mathfrak u}} +  \epsilon  |\ln \epsilon|\mathfrak{C} )\big)_{n-1} \Bigr]^2
\Big)
\\[2mm]
&
 \quad \,+\,
O\Big(\Bigl[\big ( {\bf P} ( {\tilde{\mathfrak u}} +  \epsilon|\ln \epsilon|\mathfrak{C}) \big)_{n+1}
- \big( {\bf P}  ( {\tilde{\mathfrak u}}+  \epsilon  |\ln \epsilon|\mathfrak{C} ) \big)_{n}  \Bigr]^2\Big).
\end{split}
\end{equation}
Recall the  definition of ${\bf P}$ in  \eqref{P},  we have the following expression
\begin{align*}
 &\big( {\bf P}  ( {\tilde{\mathfrak u}} +  \epsilon  |\ln \epsilon|\mathfrak{C} ) \big)_{n}
\,=\,
\sum_{j=1}^{N} p_{n+1,j} \big(  {\tilde{\mathfrak u}} +  \epsilon  |\ln \epsilon| \mathfrak{C}_j \big)
\,-\,
\sum_{j=1}^{N} p_{n,j} \big(  {\tilde{\mathfrak u}} +  \epsilon  |\ln \epsilon| \mathfrak{C}_j \big)  .
\end{align*}
From the definition of $\mathfrak{C}$ as in \eqref{definitionofmathfrakC}, we obtain that
\begin{equation}\label{estimatesJn0}
\|  {\mathcal J}_n\big( {\bf P}\big(    {\tilde{\mathfrak u}} +  \epsilon  |\ln \epsilon| \mathfrak{C} \big)  \big) \|_{L^2(0,    \ell)} \leq  \epsilon^{ 1+ \sigma}.
\end{equation}
We can also get
\begin{equation}\label{estimatesJn1}
\| \epsilon   \gamma_0    {\bf W}\big(   {\tilde{\mathfrak u}} +  \epsilon  |\ln \epsilon| \mathfrak{C} \big) \|_{L^2(0,    \ell)} \leq  \epsilon^{1+ \sigma},
\end{equation}
\begin{equation}\label{estimatesJn2}
\Big\|\epsilon^2 |\ln \epsilon| \gamma_0  \mathfrak{C}''\Big\|_{L^2(0,    \ell)} \leq  \epsilon^{1+ \sigma}.
\end{equation}

\medskip
Lemma \ref{lemma6point3} implies that there exists a sequence $\{\epsilon_l: l\in{\mathbb N}\}$ such that all operators ${\mathbb L}^{\epsilon_l}_n, n=1,\cdots, N$ are invertible
with estimates in \eqref{l2estimatev}-\eqref{h2estimate1vN}.
By using \eqref{estimatesJn0}-\eqref{estimatesJn2} and applying the Contraction Mapping Principle on ${\mathcal X}$ with $\epsilon=\epsilon_l$,
we can get a solution ${\tilde{\mathfrak u}}\in {\mathcal X}$ to the problem \eqref{ts8a}-\eqref{boundaryconditionofmathfraku3}
with the following estimate
\begin{align*}
&\epsilon_l\, \| {{\tilde{\mathfrak u}} }''\|_{L^2(0, \ell)}
\,+\,\sqrt{\epsilon_l}\, \| { {\tilde{\mathfrak u}}  }'\|_{L^2(0, \ell)}
\,+\,\| {\tilde{\mathfrak u}} \|_{L^{\infty}(0, \ell)}    \nonumber
\\[2mm]
&\,\leq\,
C\,\sqrt{\frac{1}{\epsilon_l}}\, \Big[\, \|	{\tilde{\mathcal J}}\big(        {\tilde{\mathfrak u}} +  \epsilon  |\ln \epsilon| \mathfrak{C} \big) \|_{L^2(0, \ell)}
\,+\,\epsilon_l  \gamma_0  \| {\bf W} \big({\tilde{\mathfrak u}} +  \epsilon_l  |\ln \epsilon_l| \mathfrak{C} \big) \|_{L^2(0, \ell)}\Big]
\nonumber
\\[2mm]
&\quad
\,+ \,\sqrt{\frac{1}{\epsilon_l}}\Big\|\epsilon^2 |\ln \epsilon| \gamma_0  \mathfrak{C}'' \Big\|_{L^2(0,    \ell)}
\,+ \,\sqrt{\frac{1}{\epsilon_l}}\,\| \bf h\|_{L^2(0, \ell)}
\\[2mm]
& \leq \epsilon_l^{\frac 12 + \sigma}
\,+ \,\sqrt{\frac{1}{\epsilon_l}}\,\| \bf h\|_{L^2(0, \ell)} \nonumber .
\end{align*}
Here is the conclusion.

\begin{proposition}\label{lemma6point1}
For given ${\bf h}=(h_1, \cdots, h_N)^T$ with $\|{\bf h}\|_{L^2(0, \ell)}\leq \epsilon^{1+ \sigma}$ for some constant $0<\sigma<1$,
there exists a sequence $\{\epsilon_l: l\in{\mathbb N}\}$
approaching $0$ such that problem (\ref{ts8a})-(\ref{boundaryconditionofmathfraku3}) admits a
solution ${\tilde{\mathfrak u}}=({\tilde{\mathfrak u}}_1, \cdots, {\tilde{\mathfrak u}}_N)^T$ with the estimates
 \begin{equation*}
\epsilon_l\, \| { {\tilde{\mathfrak u}} }''\|_{L^2(0, \ell)}
\,+\,\sqrt{\epsilon_l}\, \| {  {\tilde{\mathfrak u}} }'\|_{L^2(0, \ell)}
\,+\,\| {\tilde{\mathfrak u}}\|_{L^{\infty}(0, \ell)}
\, \leq\,  {C}\, \epsilon_l^{-\frac{1}{2}}\,  \|{\bf h}\|_{L^2(0, \ell)}
\,+\,
\epsilon_l^{\frac 12 + \sigma}.
\end{equation*}
\qed
\end{proposition}

\medskip
\noindent{\bf Step 4:}
We go back to solve (\ref{ts10a}) and \eqref{boundarymathfraku}  and finish the proof of Theorem \ref{main1}.
By recalling   (\ref{highorder1}), (\ref{highorder2}) and Proposition \ref{proposition 5.1},  combining the assumption \eqref{ddotfnorm} of ${\tilde{\bf f}}$ (this will give the constraint of ${\hat{\mathfrak u}}$),  it is readily checked the validity of the decomposition
\begin{equation}
{\tilde\Upsilon}  \big({\hat{\mathfrak u}}+ \epsilon |\ln \epsilon | \mathfrak{C}\big)
={\hat\Upsilon}({\hat{\mathfrak u}}, {\hat{\mathfrak u}}{'}, {\hat{\mathfrak u}}{''} )
\,+\,
{\check\Upsilon}( {\hat{\mathfrak u}}, {\hat{\mathfrak u}}{'} ),
\end{equation}
with properties, for $n=1,\cdots,N$,
\begin{align}
\big\|{\hat\Upsilon}_{n}( {\hat{\mathfrak u}}, {\hat{\mathfrak u}}{'}, {\hat{\mathfrak u}}{''})\big\|_{L^2(0, \ell)}
   \leq& C\epsilon^{1+ \sigma}, \mspace{150mu}
\\[1mm]
\big\|{\hat\Upsilon}_{n}( {\hat{\mathfrak u}}^{(1)}, {\hat{\mathfrak u}}^{(1)}{ '}, {{\hat{\mathfrak u}}^{(1)}{ ''}})
-
{\hat\Upsilon}_{n}({\hat{\mathfrak u}}^{(2)}, {\hat{\mathfrak u}}^{(2)}{'}, {\hat{\mathfrak u}}^{(2)}{''})\big\|_{L^2(0, \ell)}
\leq&
C\epsilon^{ \frac 12} \, \big\|{\hat{\mathfrak u}}^{(1)}-{\hat{\mathfrak u}}^{(2)}\big\|_{H^2(0, \ell)}, \label{Lipschgraceu}
\end{align}
and
\begin{equation}
\big\|{\check\Upsilon}_{n}( {\hat{\mathfrak u}}, {\hat{\mathfrak u}}{'})\big\|_{L^2(0, \ell)}\, \leq\, C\epsilon^{1+ \sigma}.
\end{equation}
We define
$$
{\mathbb D} =  \left\{\,\, {\mathbb V}\in H^2(0, \ell)
\,:\,
\epsilon_l\, \| {\mathbb V}'' \|_{L^2(0, \ell)}
\,+\,
\sqrt{\epsilon_l}\, \| {\mathbb V}'\|_{L^2(0, \ell)}
\,+\,
\|    {\mathbb V} \|_{L^{\infty}(0, \ell)}  \le D_1 \epsilon_l^{\frac 12 +\sigma}
\,\right\}
$$
with the sequence $ \{\epsilon_l\}$   given in Lemma \ref{lemma6point3}  and a fixed constant $D_1$, and then for given ${ {\mathbb V}}\in {\mathbb D}$
set the right hand side of (\ref{ts8a}) with
\begin{equation*}
h_n({\mathbb V})=
{\hat\Upsilon}_{n}({\hat{\mathfrak u}}, {\hat{\mathfrak u}}', {\hat{\mathfrak u}}'')	
\, +\,
{\check\Upsilon}_{n}({ \mathbb V},\,\, {\mathbb V}'),
\quad
n=1,\cdots,N.
\end{equation*}
%For any ${\mathbb V}^{(1)}\in {\mathbb D}$,
%we can now use Proposition \ref{lemma6point1} to get a solution ${\mathbb V}^{(2)}$ denote by
%$$
%{\mathbb V}^{(2)}:={\widetilde T}_1\Big(\, {\hat\Upsilon}_1\big({\mathbb V}^{(1)}\big),
%\, \cdots, \,
%{\hat\Upsilon}_N\big({\mathbb V}^{(1)}\big)\, \Big).
%$$
Whence,  by the facts in \eqref{Lipschgraceu}, the theory developed in Proposition \ref{lemma6point1} and the Contraction Mapping Theorem,
we find ${\mathbb V}$ for a fixed ${\tilde{\mathbb V}}$ in $\mathbb{D}$.
This will give a mapping as
$$
{\mathcal Z}({\tilde{\mathbb V}})={\mathbb V},
$$
and the solution of our problem is simply a fixed point of ${\mathcal Z}$.
Continuity of ${\check\Upsilon}_1$, $\cdots$, ${\check\Upsilon}_N$ with
respect to their parameters and a standard regularity arguments allows us to conclude
that ${\mathcal Z}$ is compact as mapping from $H^2(0, \ell)$ into itself.
The Schauder Theorem applies to yield the existence of a fixed point of
${\mathcal Z}$ as required. This ends the proof of Theorem \ref{main1}.
\qed

\bigskip
\noindent
{\bf Acknowledgements: }
J. Yang is supported by  NSFC(No. 11771167 \& No. 11831009).

\begin{appendices}
\section{The computations of \eqref{b1+b3}}\label{appendixA}

The main objective in this section is to compute the quantities in \eqref{b1+b3}. For the case $n=3,\cdots,N $, from the expression of ${\mathbf b}_{1n}, {\mathbf b}_{2n} $ as in (\ref{r5})-(\ref{r6}), we can obtain that
\begin{align}
 {\mathbf b}_{1n}&\,=\,- 2  e^{- \sqrt 2 x_n} e^{- \sqrt 2 \beta (f_n-f_{n-1})} +O(e^{- 2\sqrt 2 |x_n+\beta(f_n-f_{n-1})|})
 \nonumber
\\[1mm]
&\,=\,-2\epsilon  (N-n+1) e^{- \sqrt 2 x_n} e^{-  \sqrt 2 \beta (\mathfrak{f}_{n}-\mathfrak{f}_{n-1}) }
+O(e^{- 2\sqrt 2  |x_n+\beta(f_n-f_{n-1})|}),
\label{b1}
\\[2mm]
{\mathbf b}_{2n}&\,=\,- 2  e^{- \sqrt 2 x_n} e^{- \sqrt 2 \beta(f_n-f_{n-2})} +O(e^{- 2\sqrt 2  |x_n+\beta(f_n-f_{n-2})|})
\nonumber
\\[1mm]
&\,=\,-2 \epsilon^2 (N-n+1)(N-n+2) e^{- \sqrt 2 x_n} e^{-  \sqrt 2 \beta  (\mathfrak{f}_{n}-\mathfrak{f}_{n-2})}
+O(e^{- 2\sqrt 2  |x_n+\beta(f_n-f_{n-2})|}).
%\\[2mm]
%{\mathbf b}_{3n}&
%\,=\,2  e^{ \sqrt 2 x_n} e^{- \sqrt 2 \beta (f_{n+1}-f_{n})} +O(e^{- 2\sqrt 2 |x_n+\beta( f_{n}-f_{n+1})|})
%\nonumber
%\\[1mm]
%&\,=\, 2\epsilon  (N-n)  e^{ \sqrt 2  x_n} e^{-  \sqrt 2 \beta (\mathfrak{f}_{n+1}-\mathfrak{f}_n)}+O(e^{- 2\sqrt 2  |x_n+\beta(f_n-f_{n+ 1})|}),
%\label{b3}
%\\[2mm]
%{\mathbf b}_{4n}&\,=\,2  e^{ \sqrt 2 x_n} e^{- \sqrt 2 \beta (f_{n+2}-f_{n})} +O(e^{- 2\sqrt 2 |x_n+\beta (f_{n}-f_{n+2})|})
%\nonumber
% \\[1mm]
%&\,=\, 2 \epsilon^2(N-n)(N-n-1) e^{ \sqrt 2  x_n} e^{-  \sqrt 2 \beta (\mathfrak{f}_{n+2}-\mathfrak{f}_{n}) }
%+O \big( e^{- 2\sqrt 2 |x_n+\beta(f_n-f_{n+2})|} \big).
%\label{b4}
\end{align}
Specially, when $ n=1, 2$, we obtain
\begin{align*}
 {\mathbf b}_{11}
\,& = \,  H\big(\beta(s+ f_{1} ) \big)  -1 \nonumber
\\[1mm]
&\,=\,- 2  e^{- \sqrt 2 \,x_1} e^{-2 \sqrt 2 \,\beta \, f_1  } +O(e^{- 2\sqrt 2  \,|x_1+2 \,\beta\, f_1 |})  \nonumber
\\[1mm]
&\,=\,-2\epsilon  N \, e^{- \sqrt 2 \,x_1} e^{\,- 2 \sqrt 2 \,\beta  \, \mathfrak{f}_{1} \,} +O(e^{- 2\sqrt 2  \,|x_1+2\,\beta\, f_1  |}),
\\[2mm]
{\mathbf b}_{21}  &\,=\, H\big(\beta(s+ f_{2} ) \big)  -1     \nonumber
\\[1mm]
&\,=\,
- 2  e^{- \sqrt 2 \,x_1} e^{- \sqrt 2 \,\beta \,  (f_{2}+ f_{1})  } +O(e^{- 2\sqrt 2  \,|x_1+  \,\beta\,  (f_{2}+ f_{1})  |})\nonumber
\\[1mm]
&\,=  - 2 \epsilon^2N(N-1)     e^{- \sqrt 2 \,x_1}   e^{- \sqrt 2 \,\beta \,  ( \mathfrak {f}_{2}+  \mathfrak {f}_{1})  } +O(e^{- 2\sqrt 2  \,|x_1+  \,\beta\,  (f_{2}+ f_{1})  |}),
\end{align*}
and
\begin{align*}
{\mathbf b}_{12}
\,& = \,  H\big(\beta(s-   f_{1} ) \big)  -1 \nonumber
\\[1mm]
&\,=\,- 2  e^{- \sqrt 2 \,x_2} e^{-  \sqrt 2 \,\beta \,  (f_{2}-f_{1})  } +O(e^{- 2\sqrt 2  \,|x_2+  \,\beta \,(f_{2}-f_{1}) |})  \nonumber
\\[1mm]
&\,=\,-2  ( N -1) \epsilon \, e^{- \sqrt 2 \,x_2} e^{-  \sqrt 2 \,\beta \,  (\mathfrak {f}_{2}-\mathfrak {f}_{1})  }     +O(e^{- 2\sqrt 2  \,|x_2+  \,\beta \,(f_{2}-f_{1}) |}),
\\[2mm]
{\mathbf b}_{22}  &\,=\, H\big(\beta(s+ f_{1} ) \big)  -1     \nonumber
\\[1mm]
&\,=\,
- 2  e^{- \sqrt 2 \,x_2} e^{- \sqrt 2 \,\beta \,  (f_{2}+f_{1})  } +O(e^{- 2\sqrt 2  \,|x_2+  \,\beta\,  (f_{2}+f_{1})  |})\nonumber
\\[1mm]
&\,=  - 2 \epsilon^2  N(N-1)^2   e^{- \sqrt 2 \,x_2}   e^{-  \sqrt 2 \,\beta \,  (f_{2}+f_{1})  } +O(e^{- 2\sqrt 2  \,|x_2+ 2  \,\beta\,  f_{2}   |}).
 \end{align*}

Similarly,   for the case $ n =1, \cdots, N-2$,  we can also obtain
 \begin{align}
% {\mathbf b}_{1n}&\,=\,- 2  e^{- \sqrt 2 x_n} e^{- \sqrt 2 \beta (f_n-f_{n-1})} +O(e^{- 2\sqrt 2 |x_n+\beta(f_n-f_{n-1})|})
% \nonumber
%\\[1mm]
%&\,=\,-2\epsilon  (N-n+1) e^{- \sqrt 2 x_n} e^{-  \sqrt 2 \beta (\mathfrak{f}_{n}-\mathfrak{f}_{n-1}) }
%+O(e^{- 2\sqrt 2  |x_n+\beta(f_n-f_{n-1})|}),
%\label{b1}
%\\[2mm]
%{\mathbf b}_{2n}&\,=\,- 2  e^{- \sqrt 2 x_n} e^{- \sqrt 2 \beta(f_n-f_{n-2})} +O(e^{- 2\sqrt 2  |x_n+\beta(f_n-f_{n-2})|})
%\nonumber
%\\[1mm]
%&\,=\,-2 \epsilon^2 (N-n+1)(N-n+2) e^{- \sqrt 2 x_n} e^{-  \sqrt 2 \beta  (\mathfrak{f}_{n}-\mathfrak{f}_{n-2})}
%+O(e^{- 2\sqrt 2  |x_n+\beta(f_n-f_{n-2})|}).
%\\[2mm]
{\mathbf b}_{3n}&
\,=\,2  e^{ \sqrt 2 x_n} e^{- \sqrt 2 \beta (f_{n+1}-f_{n})} +O(e^{- 2\sqrt 2 |x_n+\beta( f_{n}-f_{n+1})|})
\nonumber
\\[1mm]
&\,=\, 2\epsilon  (N-n)  e^{ \sqrt 2  x_n} e^{-  \sqrt 2 \beta (\mathfrak{f}_{n+1}-\mathfrak{f}_n)}+O(e^{- 2\sqrt 2  |x_n+\beta(f_n-f_{n+ 1})|}),
\label{b3}
\\[2mm]
{\mathbf b}_{4n}&\,=\,2  e^{ \sqrt 2 x_n} e^{- \sqrt 2 \beta (f_{n+2}-f_{n})} +O(e^{- 2\sqrt 2 |x_n+\beta (f_{n}-f_{n+2})|})
\nonumber
 \\[1mm]
&\,=\, 2 \epsilon^2(N-n)(N-n-1) e^{ \sqrt 2  x_n} e^{-  \sqrt 2 \beta (\mathfrak{f}_{n+2}-\mathfrak{f}_{n}) }
+O \big( e^{- 2\sqrt 2 |x_n+\beta(f_n-f_{n+2})|} \big).
\label{b4}
\end{align}
And  specially, when $ n=N-1, N$,  recalling the notation  $f_{N+1} = + \infty$,  we have
\begin{align*}
{\mathbf b}_{3\,  N-1}
\, &=\,   H\big( \beta(s-f_N)  \big)  +1 \nonumber
\\[1mm]
& \, =  2  e^{- \sqrt 2 \,x_{N-1}} e^{- \sqrt 2 \,\beta \,(f_{N}-f_{N-1})} +O(e^{- 2\sqrt 2  \,|x_{N-1}+\,\beta(\,f_{N-1} -f_{N-3}\,)|}) \nonumber
\\[1mm]   \, &= \,   2  \epsilon\, e^{ \sqrt 2 \,x_{N-1}} e^{-  \sqrt 2 \,\beta  \,(\mathfrak{f}_{N}-\mathfrak{f}_{N-1})} +O(e^{- 2\sqrt 2  \,|x_{N}+\,\beta(\,f_{N}-f_{N-1}\,)|}),
\\[2mm]
{\mathbf b}_{4\,  N-1} \, &=\,  0,
\end{align*}
and
\begin{equation*}
{\mathbf b}_{3\,  N} = {\mathbf b}_{4\,  N} = 0.
\end{equation*}

\medskip
By combining the above formulas, we obtain the following:

\medskip
\noindent {\bf{Case 1:}}
When $n=3,\cdots,N-2$, \eqref{b1+b3} holds.

\medskip
\noindent {\bf{Case 2:}} When $n=1, 2$,  we obtain that
\begin{align*}
 &{\mathbf b}_{11}\,-\,{\mathbf b}_{21}\,+\,{\mathbf b}_{31}\,-\,{\mathbf b}_{41}
 \\[1mm]
% &\,=\,- 2  e^{- \sqrt 2 \,x_1} e^{- \sqrt 2 \,\beta \,(f_1-f_{0})} +O(e^{- 2\sqrt 2  \,|x_1+\,\beta\,(f_1-f_{0}\,)|})
%  \\[1mm]
%  &\,\quad\,+\, 2  e^{ \sqrt 2  \,x_1} e^{- \sqrt 2 \,\beta \,(f_{2}-f_{1})} +O(e^{- 2\sqrt 2  \,|x_1+\,\beta\,(f_1-f_{2}\,)|})
% \\[1mm]
% &\,\quad \, -2  e^{ \sqrt 2  \,x_1} e^{- \sqrt 2 \,\beta\, (f_{3}-f_{1})} +O(e^{- 2\sqrt 2  \,|x_1+\,\beta\,(f_1-f_{3}\,)|})
%   \\[1mm]
&\,=\, -2\epsilon N  \, e^{- \sqrt 2 \,x_1} e^{\,- 2 \sqrt 2 \,\beta  \, \mathfrak{f}_{1} \,} +O(e^{- 2\sqrt 2  \,|x_1+2\,\beta\, f_1  |})
\\[1mm]
&\,\quad +\, 2\epsilon   (N-1) \,  e^{ \sqrt 2  \,x_1} e^{-  \sqrt 2 \,\beta  \,(\mathfrak{f}_{2}-\mathfrak{f}_1)}+O(e^{- 2\sqrt 2  \,|x_1+\,\beta\,(f_1-f_{2}\,)|})
\\[1mm]
&\,\quad +   2\epsilon^2  N(N-1)    e^{- \sqrt 2 \,x_1}   e^{- \sqrt 2 \,\beta \,  ( \mathfrak {f}_{2}+  \mathfrak {f}_{1})  } +O(e^{- 2\sqrt 2  \,|x_1+  \,\beta\,  (f_{2}+ f_{1})  |})
\\[1mm]
&\,\quad\,- 2 \epsilon^2 (N-2)(N-1)\,  e^{ \sqrt 2  \,x_1} e^{-  \sqrt 2 \,\beta  \,(\mathfrak{f}_{3} -\mathfrak{f}_{1})} +O(e^{- 2\sqrt 2  \,|x_1+\,\beta\,(f_1-f_{3}\,)|}),
\end{align*}
 and
  \begin{equation*}
 \begin{split}
 &{\mathbf b}_{12}\,-\,{\mathbf b}_{22}\,+\,{\mathbf b}_{32}\,-\,{\mathbf b}_{42}
 \\[1mm]
% &\,=\,- 2  e^{- \sqrt 2 \,x_1} e^{- \sqrt 2 \,\beta \,(f_1-f_{0})} +O(e^{- 2\sqrt 2  \,|x_1+\,\beta\,(f_1-f_{0}\,)|})
%  \\[1mm]
%  &\,\quad\,+\, 2  e^{ \sqrt 2  \,x_1} e^{- \sqrt 2 \,\beta \,(f_{2}-f_{1})} +O(e^{- 2\sqrt 2  \,|x_1+\,\beta\,(f_1-f_{2}\,)|})
% \\[1mm]
% &\,\quad \, -2  e^{ \sqrt 2  \,x_1} e^{- \sqrt 2 \,\beta\, (f_{3}-f_{1})} +O(e^{- 2\sqrt 2  \,|x_1+\,\beta\,(f_1-f_{3}\,)|})
%   \\[1mm]
&\,=\, -2  \epsilon ( N -1) \, e^{- \sqrt 2 \,x_2} e^{-  \sqrt 2 \,\beta \,  (\mathfrak {f}_{2}-\mathfrak {f}_{1})  }     +O(e^{- 2\sqrt 2  \,|x_2+  \,\beta \,(f_{2}-f_{1}) |})  \\[1mm]
&\,\quad +\, 2 \epsilon  (N-2) \,  e^{ \sqrt 2  \,x_2} e^{-  \sqrt 2 \,\beta  \,(\mathfrak{f}_{3}-\mathfrak{f}_2)}+O(e^{- 2\sqrt 2  \,|x_2+\,\beta\,(f_2-f_{3}\,)|})
\\[1mm]
&\,\quad +    2 \epsilon^2N(N-1)^2     e^{- \sqrt 2 \,x_2}   e^{- \sqrt 2 \,\beta \,  (\mathfrak {f}_{2} +\mathfrak {f}_{1} )   } +O(e^{- 2\sqrt 2  \,|x_2+ 2  \,\beta\,  f_{2} |})
\\[1mm]
&\,\quad\,- 2 \epsilon^2(N-2)(N-3) \,  e^{ \sqrt 2  \,x_2} e^{-  \sqrt 2 \,\beta  \,(\mathfrak{f}_{4} -\mathfrak{f}_{2})} +O(e^{- 2\sqrt 2  \,|x_2+\,\beta\,(f_2-f_{4}\,)|}).
\end{split}
 \end{equation*}

\medskip
\noindent {\bf{Case 3:}}
When $n=N-1, N$, we get that
 \begin{equation*}
 \begin{split}
 &{\mathbf b}_{1 N-1}\,-\,{\mathbf b}_{2 N-1}\,+\,{\mathbf b}_{3 N-1}\,-\,{\mathbf b}_{4N-1}
  \\[1mm]
 % &\,=\,- 2  e^{- \sqrt 2 \,x_N} e^{- \sqrt 2 \,\beta \,(f_N-f_{N-1})} +O(e^{- 2\sqrt 2  \,|x_N+\,\beta(\,f_N-f_{N-1}\,)|}) \\[1mm]
 %  &\,\quad-\, 2  e^{- \sqrt 2  \,x_N} e^{- \sqrt 2 \,\beta \,(f_N-f_{N-2})} +O(e^{- 2\sqrt 2  \,|x_N+\,\beta(\,f_N-f_{N-2}\,)|})
 %   \\[1mm]
&\,=\,-4\epsilon   e^{- \sqrt 2 x_{N-1} } e^{-  \sqrt 2 \beta (\mathfrak{f}_{N-1}-\mathfrak{f}_{N-2}) }
+O(e^{- 2\sqrt 2  |x_{N-1} +\beta(f_{N-1} -f_{N-2})|}) \\[1mm]
  &\,\quad +
 2  \epsilon\, e^{ \sqrt 2 \,x_{N-1}} e^{-  \sqrt 2 \,\beta  \,(\mathfrak{f}_{N}-\mathfrak{f}_{N-1})} +O(e^{- 2\sqrt 2  \,|x_{N}+\,\beta(\,f_{N}-f_{N-1}\,)|})\\[1mm]
   &\,\quad +  12  \epsilon^2 e^{- \sqrt 2 x_{N-1}} e^{-  \sqrt 2 \beta  (\mathfrak{f}_{N-1}-\mathfrak{f}_{N-3})}
+O(e^{- 2\sqrt 2  |x_{N-1}+\beta(f_{N-1}-f_{N-3})|}),
   \end{split}
 \end{equation*}
 and
 \begin{equation*}
 \begin{split}
 &{\mathbf b}_{1N}\,-\,{\mathbf b}_{2N}\,+\,{\mathbf b}_{3N}\,-\,{\mathbf b}_{4N}
  \\[1mm]
 % &\,=\,- 2  e^{- \sqrt 2 \,x_N} e^{- \sqrt 2 \,\beta \,(f_N-f_{N-1})} +O(e^{- 2\sqrt 2  \,|x_N+\,\beta(\,f_N-f_{N-1}\,)|}) \\[1mm]
 %  &\,\quad-\, 2  e^{- \sqrt 2  \,x_N} e^{- \sqrt 2 \,\beta \,(f_N-f_{N-2})} +O(e^{- 2\sqrt 2  \,|x_N+\,\beta(\,f_N-f_{N-2}\,)|})
 %   \\[1mm]
&\,=\,-2 \epsilon \, e^{- \sqrt 2 \,x_N} e^{-  \sqrt 2 \,\beta  \,(\mathfrak{f}_{N}-\mathfrak{f}_{N-1})} +O(e^{- 2\sqrt 2  \,|x_N+\,\beta(\,f_N-f_{N-1}\,)|}) \\[1mm]
  &\,\quad  +  \, 4 \epsilon^2\, e^{- \sqrt 2 \,x_N} e^{-  \sqrt 2 \,\beta  \,(\mathfrak{f}_{N}-\mathfrak{f}_{N-2})}+O(e^{- 2\sqrt 2  \,|x_N+\,\beta(\,f_N-f_{N-2}\,)|}).
   \end{split}
 \end{equation*}

\medskip
\section{Linear problem}\label{appendixB}

We first set
\begin{equation*}
{\mathcal S}=\R\times(0, \ell/\epsilon),
\end{equation*}
and provide the following Lemma in \cite{delPKowWei2}.
\begin{lemma}(Lemma 4.2 in \cite{delPKowWei2})\label{lemma of solve linearized pro}
For a given function $\Phi_*(x, z)\in L^2({\mathcal S})$ with
$$
\int_{\mathbb R} \Phi_*(x,z) H_x \, {\mathrm{d}}x\,=\, 0,\quad   0<z<\frac{\ell}{\epsilon},
$$
let us consider the following problem
\begin{equation}\label{definition6}
\phi_{*,xx} \, +\, \phi_{*,zz} \, +\, (1-3H^2)\, \phi_{*}\,=\, \Phi_*
\quad
\mbox{in } {\mathcal S},
  \end{equation}
with the  conditions
\begin{equation}
\phi_{*}(x, 0)\,=\, \phi_{*}(x, \ell /{\epsilon}),  \quad\phi_{*, z}(x, 0)\,=\, \phi_{*, z}(x, \ell/{\epsilon}),  \quad x\in \mathbb R,
\end{equation}
\begin{equation}\label{definition7}
\int_\mathbb{R}{\phi_{*}(x,z) H_x}\,{\mathrm{d}}x =0,\quad   0<z<\frac{\ell}{\epsilon}.
\end{equation}
The problem \eqref{definition6}-\eqref{definition7} has  a unique solution $\phi_{*}
   \in H^2({\mathcal S})$.
\qed
\end{lemma}

\medskip
Then, by using the above Lemma, we can obtain following result:
\begin{lemma}\label{lemma of solve linearized pro11}
For a given function $\Phi^*(x, z)\in L^2({\mathcal S})$ with
$$
\int_{\mathbb R} \Phi^*(x,z) H_x \, {\mathrm{d}}x\,=\, 0,\quad   0<z<\frac{\ell}{\epsilon},
$$
consider the following problem
\begin{equation}\label{definition3}
   \phi^*_{zz}+ \beta^{2}\,\Big[\phi^{*}_{xx}\, +\, (1-3H^2)\phi^{*}\Big]=\, \Phi^*
   \quad
\mbox{ in } {\mathcal S},
\end{equation}
with the  conditions
\begin{equation}\label{definition4}
\phi^{*}(x, 0)\,=\, \phi^{*}(x, \ell /{\epsilon}),  \quad\phi^{*}_z (x, 0)\,=\, \phi^*_{ z}(x, \ell/{\epsilon}), \quad x\in \mathbb R,
\end{equation}
\begin{equation}\label{definition5}
\int_\mathbb{R}{\phi^{*}(x,z) H_x}\,{\mathrm{d}}x =0,\quad   0<z<\frac{\ell}{\epsilon}.
\end{equation}
There exists
 a unique solution $\phi^{*}
   \in H^2({\mathcal S})$ to problem \eqref{definition3}-\eqref{definition5}, which satisfies
   \begin{equation}\label{linftyestimate}
   \|\phi_{*}\|_{H^2( {\mathcal S})}  \leq C \|\Phi^* \|_{L^2({\mathcal S})}.
   \end{equation}
\end{lemma}

\begin{proof}
Let
\begin{equation*}
\phi^{*}(x,z)= \tilde \phi^{*} (x, \iota(z)), \quad \iota(z) = \epsilon^{-1} \int_0^{\epsilon z} \beta(r) {\mathrm{d}}r.
\end{equation*}
Here, the map
$$
\iota: \Big[ 0, \frac{\ell}{\epsilon}\Big)\rightarrow \Big[ 0, \frac{\hat \ell}{\epsilon}\Big),
\quad
{\tilde z} = \iota(z)
$$
 is a diffeomorphism, where $ \hat \ell = \int_0^\ell \beta(r) {\mathrm{d}}r. $

\medskip
 It is easy to derive that
 $$
 \phi^{*}_{z} (x,z) \,=\, \beta\,\tilde \phi^{*}_{{\tilde z}}(x,{\tilde z}),
 \qquad
 \phi^{*}_{zz} (x,z)\,=\, \beta^2\, \tilde  \phi^{*}_{{\tilde z}{\tilde z}}(x,{\tilde z}) \,+\, \epsilon\, \beta'\,\tilde\phi^{*}_{{\tilde z}}(x,{\tilde z}),
 $$
while differentiation in $x $ does not change.
Therefore, problem \eqref{definition3}-\eqref{definition5} can be rewritten as
 \begin{align}\label{definition8}
    \tilde  \phi^{*}_{{\tilde z}{\tilde z}}
    \,+\,
    \Big[ \tilde \phi^{*}_{xx}\, +\, (1-3H^2) \tilde \phi^{*}\Big]
    &=\,  \tilde{ \Phi}^*\,-\, \epsilon\,\beta^{-2}\, \beta'\,\tilde \phi^{*}_{{\tilde z}}\quad \mbox{ in }  \mathbb R\times \Big[0,\frac{\hat \ell}{\epsilon}\Big),
 \\
     \int_{\mathbb R}  \tilde\Phi^*(x, {\tilde z}) H_x \, {\mathrm{d}}x\,=\,& 0,\quad   0<{\tilde z}< \frac{\hat \ell}{\epsilon},
  \end{align}
with the  conditions
\begin{equation}\label{definition9}
 \tilde\phi^{*}(x, 0)\,=\, \phi^{*}(x, \hat \ell /{\epsilon}),  \quad \tilde\phi^{*}_{{\tilde z}} (x, 0)\,=\,  \tilde \phi^*_{ {\tilde z} }(x, \hat \ell/{\epsilon}), \quad x\in \mathbb R,
\end{equation}
\begin{equation}\label{definition10}
\int_\mathbb{R} {\tilde \phi}^*(x,{\tilde z}) H_x\,{\mathrm{d}}x=0,\quad   0<{\tilde z}<\frac{\hat \ell}{\epsilon}.
\end{equation}
From Lemma \ref{lemma of solve linearized pro},  we can know that \eqref{definition8}-\eqref{definition10}
have a unique solution $\tilde \phi^{*}(x, {\tilde z}) $.
The result follows by transforming  $\tilde \phi^{*}(x, \iota(z) )$ into $ \phi^{*}(x, z)$ via change of variables.
By using the  method of sub-supersolutions, we can get the estimate \eqref{linftyestimate}.
This concludes the proof of the lemma.
\end{proof}

%\medskip
%  We consider the following linearized problem.
%\begin{equation}
%\label{tildeLmathcal}
%{\mathcal L}(\tilde \phi)\,=\, g
%\, +\, \sum_{j=1}^N\, c_j(z)\chi_j(s, z)\,  H'\big(    \beta(s-f_j)\quad\mbox{in } {\mathfrak S},
%\end{equation}
%\begin{align}
%  \frac{  \partial {\tilde \phi (0,z)}}{\partial s} \, =\, 0, \quad  \tilde {\phi}(\pm\infty, z)= 0,
%\end{align}
%\begin{equation}
%\label{projectedproblem5}
%\int_{\mathbb{R}} \tilde \phi(s, z)\, \chi_j(s, z)\, H'\big(    \beta(s-f_j)  \big) \, {\rm d}s \,=\, 0 \quad 0<z<\frac {\ell }{\epsilon}, \quad j=1, \cdots, N.
%\end{equation}
%
%
%
%
%
%\begin{lemma}
% Let $(\tilde \phi, c_j, g)$ be the solution  of \eqref{tildeLmathcal}-\eqref{projectedproblem5}. Then for $\epsilon $ small enough, we have the following estimate
% \begin{align}
% \| \tilde \phi \|_{H^2}+ \|{\bf{c}}\|_{L^\infty(0, \frac{\ell}{\epsilon})} \, \leq \, C \|g\|_{L^2}
% \end{align}
%
%\end{lemma}
%

\end{appendices}

 \end{document}